\documentclass[preprint,10pt]{elsarticlearxiv}

\usepackage{amssymb, amsmath, amsthm, amsfonts}
\usepackage{mathrsfs,comment}
\usepackage{graphicx}
\usepackage{placeins}
\usepackage{rotating}
\usepackage{tikz}
\usepackage{float}
\usepackage{subfigure}
\usepackage{hvfloat}
\usepackage{sseq}
\usepackage{caption}
\usepackage{pdflscape}
\usepackage{hyperref}  
\usepackage{url}
\usepackage[all,arc,2cell]{xy}
\UseAllTwocells
\usepackage{enumerate}
\usepackage{mathrsfs}
\usepackage{color}
\usepackage[linecolor=black, color=green!40]{todonotes}
\usepackage{chngcntr}
 \usepackage{lineno}

\hypersetup{%
  bookmarksnumbered=true,%
  bookmarks=true,%
  colorlinks=true,%
  linkcolor=blue,%
  citecolor=blue,%
  filecolor=blue,%
  menucolor=blue,%
  urlcolor=blue,%
  pdfnewwindow=true,%
  pdfstartview=FitBH}

\def\@url#1{{\tt\def~{\lower3.5pt\hbox{\char'176}}\def\_{\char'137}#1}}

\let\fullref\autoref
\def\makeautorefname#1#2{\expandafter\def\csname#1autorefname\endcsname{#2}}
\makeautorefname{equation}{Equation}%
\makeautorefname{footnote}{footnote}%
\makeautorefname{item}{item}%
\makeautorefname{figure}{Figure}%
\makeautorefname{table}{Table}%
\makeautorefname{part}{Part}%
\makeautorefname{appendix}{Appendix}%
\makeautorefname{chapter}{Chapter}%
\makeautorefname{section}{Section}%
\makeautorefname{subsection}{Section}%
\makeautorefname{subsubsection}{Section}%
\makeautorefname{paragraph}{Paragraph}%
\makeautorefname{subparagraph}{Paragraph}%
\makeautorefname{theorem}{Theorem}%
\makeautorefname{theo}{Theorem}%
\makeautorefname{thm}{Theorem}%
\makeautorefname{addendum}{Addendum}%
\makeautorefname{addend}{Addendum}%
\makeautorefname{add}{Addendum}%
\makeautorefname{maintheorem}{Main theorem}%
\makeautorefname{mainthm}{Main theorem}%
\makeautorefname{corollary}{Corollary}%
\makeautorefname{corol}{Corollary}%
\makeautorefname{coro}{Corollary}%
\makeautorefname{cor}{Corollary}%
\makeautorefname{lemma}{Lemma}%
\makeautorefname{lemm}{Lemma}%
\makeautorefname{lem}{Lemma}%
\makeautorefname{sublemma}{Sublemma}%
\makeautorefname{sublem}{Sublemma}%
\makeautorefname{subl}{Sublemma}%
\makeautorefname{proposition}{Proposition}%
\makeautorefname{proposit}{Proposition}%
\makeautorefname{propos}{Proposition}%
\makeautorefname{propo}{Proposition}%
\makeautorefname{prop}{Proposition}%
\makeautorefname{property}{Property}
\makeautorefname{proper}{Property}
\makeautorefname{scholium}{Scholium}%
\makeautorefname{step}{Step}%
\makeautorefname{conjecture}{Conjecture}%
\makeautorefname{conject}{Conjecture}%
\makeautorefname{conj}{Conjecture}%
\makeautorefname{question}{Question}
\makeautorefname{questn}{Question}
\makeautorefname{quest}{Question}
\makeautorefname{ques}{Question}
\makeautorefname{qn}{Question}
\makeautorefname{definition}{Definition}%
\makeautorefname{defin}{Definition}%
\makeautorefname{defi}{Definition}%
\makeautorefname{def}{Definition}%
\makeautorefname{defn}{Definition}%
\makeautorefname{dfn}{Definition}%
\makeautorefname{notation}{Notation}
\makeautorefname{nota}{Notation}
\makeautorefname{notn}{Notation}
\makeautorefname{remark}{Remark}%
\makeautorefname{rema}{Remark}%
\makeautorefname{rem}{Remark}%
\makeautorefname{rmk}{Remark}%
\makeautorefname{rk}{Remark}%
\makeautorefname{remarks}{Remarks}%
\makeautorefname{rems}{Remarks}%
\makeautorefname{rmks}{Remarks}%
\makeautorefname{rks}{Remarks}%
\makeautorefname{example}{Example}%
\makeautorefname{examp}{Example}%
\makeautorefname{exmp}{Example}%
\makeautorefname{exmps}{Examples}%
\makeautorefname{exam}{Example}%
\makeautorefname{exa}{Example}%
\makeautorefname{algorithm}{Algorith}%
\makeautorefname{algo}{Algorith}%
\makeautorefname{alg}{Algorith}%
\makeautorefname{axiom}{Axiom}%
\makeautorefname{axi}{Axiom}%
\makeautorefname{ax}{Axiom}%
\makeautorefname{case}{Case}%
\makeautorefname{claim}{Claim}%
\makeautorefname{clm}{Claim}%
\makeautorefname{assumption}{Assumption}%
\makeautorefname{assumpt}{Assumption}%
\makeautorefname{conclusion}{Conclusion}%
\makeautorefname{concl}{Conclusion}%
\makeautorefname{conc}{Conclusion}%
\makeautorefname{condition}{Condition}%
\makeautorefname{condit}{Condition}%
\makeautorefname{cond}{Condition}%
\makeautorefname{construction}{Construction}%
\makeautorefname{construct}{Construction}%
\makeautorefname{const}{Construction}%
\makeautorefname{cons}{Construction}%
\makeautorefname{criterion}{Criterion}%
\makeautorefname{criter}{Criterion}%
\makeautorefname{crit}{Criterion}%
\makeautorefname{exercise}{Exercise}%
\makeautorefname{exer}{Exercise}%
\makeautorefname{exe}{Exercise}%
\makeautorefname{problem}{Problem}%
\makeautorefname{problm}{Problem}%
\makeautorefname{probm}{Problem}%
\makeautorefname{prob}{Problem}%
\makeautorefname{solution}{Solution}%
\makeautorefname{soln}{Solution}%
\makeautorefname{sol}{Solution}%
\makeautorefname{summary}{Summary}%
\makeautorefname{summ}{Summary}%
\makeautorefname{sum}{Summary}%
\makeautorefname{operation}{Operation}%
\makeautorefname{oper}{Operation}%
\makeautorefname{observation}{Observation}%
\makeautorefname{observn}{Observation}%
\makeautorefname{obser}{Observation}%
\makeautorefname{obs}{Observation}%
\makeautorefname{ob}{Observation}%
\makeautorefname{convention}{Convention}%
\makeautorefname{convent}{Convention}%
\makeautorefname{conv}{Convention}%
\makeautorefname{cvn}{Convention}%
\makeautorefname{warning}{Warning}%
\makeautorefname{warn}{Warning}%
\makeautorefname{note}{Note}%
\makeautorefname{fact}{Fact}%

  \makeatletter
                   \let\c@lemma\c@theorem
                  \makeatother

\newtheorem{thm}{Theorem}[subsection]
\newtheorem{cor}{Corollary}[subsection]
\newtheorem{prop}{Proposition}[subsection]
\newtheorem{lem}{Lemma}[subsection]
\theoremstyle{definition}
\newtheorem{defn}{Definition}[subsection]

\newtheorem{rem}{Remark}[subsection]
\newtheorem{warn}{Warning}[subsection]

\makeatletter
\let\c@lem=\c@thm
\let\c@cor=\c@thm
\let\c@prop=\c@thm
\let\c@lem=\c@thm
\let\c@defn=\c@thm
\let\c@exmps=\c@thm
\let\c@rem=\c@thm
\let\c@warn=\c@thm
\makeatother

\numberwithin{equation}{subsection}

\makeautorefname{Athm}{Theorem}
\makeautorefname{Arem}{Remark}
\makeautorefname{Alem}{Lemma}
\makeautorefname{Acor}{Corollary}
\makeautorefname{Aprop}{Proposition}
\makeautorefname{Adefn}{Definition}

\theoremstyle{theorem}
\newtheorem{Athm}{Theorem}
\newtheorem{Acor}{Corollary}
\newtheorem{Alem}{Lemma}
\newtheorem{Aprop}{Proposition}

\theoremstyle{definition}
\newtheorem{Arem}{Remark}

\makeatletter
\let\c@Alem=\c@Athm
\let\c@Acor=\c@Athm
\let\c@Aprop=\c@Athm
\let\c@Alem=\c@Athm
\let\c@Adefn=\c@Athm
\let\c@Arem=\c@Athm
\let\c@Anot=\c@Athm
\makeatother

\newcommand{\FF}{\mathbb{F}}
\newcommand{\Z}{\mathbb{Z}}

\newcommand{\F}{\mathbb{F}}

\newcommand{\G}{\mathbb{G}}
\newcommand{\W}{\mathbb{W}}
\newcommand{\Sn}{\mathbb{S}}
\newcommand{\cC}{\mathcal{C}}
\newcommand{\cE}{\mathcal{E}}

\newcommand{\ob}{\overline{b}}
\newcommand{\od}{\overline{\Delta}}
\newcommand{\dd}{\Delta}
\newcommand{\db}{b}

\newcommand{\Wu}{\mathbb{W}[\![u_1]\!]}

\DeclareSymbolFontAlphabet{\scr}{rsfs}
\newcommand{\sC}{\scr{C}}
\newcommand{\sE}{\scr{E}}

\newcommand{\smsh}{\wedge}
\newcommand{\ra}{\rightarrow}
\newcommand{\xra}{\xrightarrow}
\newcommand{\ox}{\otimes}
\newcommand{\ot}{\otimes}

\def\quickop#1{\expandafter\newcommand\csname #1\endcsname{\operatorname{#1}}}
\quickop{Hom} \quickop{End} \quickop{Aut} \quickop{Tel} \quickop{Mic} 
\quickop{Ext} \quickop{Tor} \quickop{Cotor} \quickop{Id} \quickop{Coker} \quickop{Ker}
\quickop{Lim} \quickop{Colim} \quickop{Holim} \quickop{Hocolim}
\quickop{id} \quickop{tel} \quickop{mic} \quickop{coker} 
\quickop{colim} \quickop{holim} \quickop{hocolim} \quickop{im}
\DeclareMathOperator{\Gal}{Gal}

\usepackage{etoolbox}
\makeatletter
\def\@mkboth#1#2{}
\newlength\appendixwidth
\preto\appendix{\addtocontents{toc}{\protect\patchl@section}}
\newcommand{\patchl@section}{%
  \settowidth{\appendixwidth}{\textbf{Appendix }}%
  \addtolength{\appendixwidth}{1.5em}%
  \patchcmd{\l@section}{1.5em}{\appendixwidth}{}{\ddt}%
}
\makeatother

\journal{Advances in Mathematics}

\begin{document}

\begin{frontmatter}



\title{Towards the homotopy of the $K(2)$-local Moore spectrum at $p=2$}


\author{Agn\`es Beaudry}
\address{Department of Mathematics, University of Colorado Boulder}

\begin{abstract}
Let $V(0)$ be the mod $2$ Moore spectrum and let $\mathcal{C}$ be the supersingular elliptic curve over $\mathbb{F}_4$ defined by the Weierstrass equation $y^2+y=x^3$. Let $F_{\mathcal{C}}$ be its formal group law and $E_{\mathcal{C}}$ be the spectrum classifying the deformations of $F_{\mathcal{C}}$. The group of automorphisms of $F_{\mathcal{C}}$, which we denote by $\mathbb{S}_{\mathcal{C}}$, acts on $E_{\mathcal{C}}$. Further, $\mathbb{S}_{\mathcal{C}}$ admits a surjective homomorphism to $\mathbb{Z}_2$ whose kernel we denote by $\mathbb{S}_{\mathcal{C}}^1$. The cohomology of $\mathbb{S}_{\mathcal{C}}^1$ with coefficients in $(E_{\mathcal{C}})_*V(0)$ is the $E_2$-term of a spectral sequence converging to the homotopy groups of $E_{\mathcal{C}}^{h\mathbb{S}_{\mathcal{C}}^1}\wedge V(0)$, a spectrum closely related to $L_{K(2)}V(0)$. In this paper, we use the algebraic duality resolution spectral sequence to compute an associated graded for $H^*(\mathbb{S}_{\mathcal{C}}^1;(E_{\mathcal{C}})_*V(0))$. These computations rely heavily on the geometry of elliptic curves made available to us at chromatic level $2$. 
\end{abstract}


\end{frontmatter}
\tableofcontents

\section{Introduction}
This paper can be read as a sequel to \cite{Paper1}. For this reason, this section builds upon the deeper discussion of \cite[Section 2]{Paper1}. We give an overview of the tools that were not introduced in the prequel and state our results. More background and motivation can be found in \cite{Paper1}.
\subsection{Background}
In this paper, we work at the prime $p=2$. Recall that Morava $K$-theory $K(2)$ is the unique ring spectrum with coefficients $K(2)_* = \F_{2}[v_2^{\pm 1}]$,
for $v_2$ in degree $6$, and with formal group law the Honda formal group law $F_2$ of height $2$. The group $\Sn_2$ is the group of automorphisms of $F_2$ over $\F_{4}$. The extended Morava stabilizer group $\G_2$ is the extension of $\Sn_2$ by the Galois group.  Morava $E$-theory $E_2$ is the Landweber exact spectrum for which $\pi_0E_2$ corepresents isomorphism classes of deformations of $F_{2}$. Its homotopy groups can be described as follows. Let $\zeta$ be a primitive third root of unity and let
\begin{equation*}
\W := W(\F_{4})\cong\Z_2[\zeta]
\end{equation*}
be the Witt vectors on $\F_{4}$. Then $(E_{2})_*\cong \W[\![u_1]\!][u^{\pm 1}]$, where $u_1$ has degree zero and $u$ has degree $-2$. The group $\G_{2}$ acts on the spectrum $E_{2}$. For a finite spectrum $X$, $L_{K(2)}X \simeq E_2^{h\G_2} \smsh X$.
Further, for closed subgroups $G$ of $\G_2$ and finite spectra $X$, there are descent spectral sequences
\begin{equation}\label{ANSS0}
E_2^{s,t} := H^s(G, (E_{2})_tX) \Longrightarrow \pi_{t-s}(E_2^{hG}\smsh X).
\end{equation}

The groups $\Sn_2$ and $\G_2$ both admit a surjective homomorphism to $\mathbb{Z}_2$
whose kernels are denoted by $\Sn_2^1$ and $\G_2^1$ respectively and
\begin{align}\label{eqn:SS21Z2}
\Sn_2 &\cong \Sn_2^1 \rtimes \Z_2, & 
\G_2 &\cong \G_2^1 \rtimes \Z_2.
\end{align}
The group $\Sn_2$ has a unique conjugacy class of maximal finite subgroups, which can be described as follows. The automorphism of $F_2$ given by $[-1]_{F_2}(x)$ generates a central subgroup $C_2$. The power series $\omega(x) = \zeta x$ generates a subgroup of order three in $\Sn_2$, denoted $C_3$. The group $C_3$ acts on a quaternion subgroup $Q_8$ of $\Sn_2$ whose center is $C_2$. The semi-direct product $G_{24} = Q_8 \rtimes C_3$ is a maximal finite subgroup of $\Sn_2$. The subgroup $C_6=C_2 \times C_3$ of $G_{24}$ will also play a central role.

Both $C_6$ and $G_{24}$ are contained in $\Sn_2^1$. However, $\Sn_2^1$ has two conjugacy classes of maximal finite subgroups. A representative for the other conjugacy class is given by $G_{24}' = \pi G_{24} \pi^{-1}$ for $\pi$ a topological generator of $\Z_2$ in the decomposition $\mathbb{S}_2 \cong \mathbb{S}_2^1 \rtimes \mathbb{Z}_2$.

The next theorem follows from \cite[Theorem 1.2.1, Theorem 1.2.4, Corollary 3.4.6]{Paper1}. (See \fullref{sec:E1} below for more details).
\begin{thm}\label{thm:resnormS2}
Let $\Z_2$ be the trivial $\Z_2[\![\Sn_{2}^1]\!]$--module. There is an exact sequence of complete left $\Z_2[\![\Sn_{2}^1]\!]$--modules
\begin{align*}
0 \rightarrow \sC_3  \xra{\partial_3} \sC_2  \xra{\partial_2}\sC_1  \xra{\partial_1}\sC_0 \xra{\varepsilon} \Z_2 \rightarrow 0, 
\end{align*}
where $\sC_0 \cong  \Z_2[\![\Sn_{2}^1/G_{24}]\!]$, $\sC_3 \cong  \Z_2[\![\Sn_{2}^1/G_{24}']\!]$ and $\sC_1 \cong  \sC_2 \cong  \Z_2[\![\Sn_{2}^1/C_6]\!]$. Let $e$ be the unit in $\Z_2[\![\Sn_{2}^1]\!]$ and $e_p$ be the resulting generator of $\sC_p$. The maps $\partial_p$ can be chosen to satisfy:
\begin{enumerate}[(a)]
\item $\partial_1(e_1) = (e-\alpha)e_0$,
\item $\partial_2(e_2) \equiv ( e+\alpha +\sE) e_1$ for $\sE \in (2, (IS_{2}^1)^2)$. 
\item $\partial_3(e_3) =  \pi (e+i+j+k)(e-\alpha^{-1}) \pi^{-1} e_2$.
\end{enumerate}
Let $F_0 = G_{24}$, $F_{1} =F_2 = C_6$ and $F_3 = G_{24}'$. For a profinite $\mathbb{Z}_2[\![\Sn_2^1]\!]$--module $M$, there is a first quadrant spectral sequence
 \[E_1^{p,q} = \Ext^q_{\Z_2[\![\Sn_2^1]\!]}(\sC_p, M) \cong H^q(F_p, M) \Longrightarrow H^{p+q}(\Sn_2^1,M)\]
with differentials $d_r : E_r^{p,q}\ra E_r^{p+r,q-r+1}$.
\end{thm}
\begin{rem}
The exact sequence of \fullref{thm:resnormS2} is called the \emph{algebraic duality resolution} because it satisfies a certain duality. This is described in Theorem 1.2.2 of \cite{Paper1}. 
The associated spectral sequence is called the \emph{algebraic duality spectral sequence} which we abbreviate as ADSS.
\end{rem}

Let $V(0)$ be the mod $2$ Moore spectrum, that is, the cofiber of multiplication by $2$ on the sphere spectrum $S^0$.  
The goal of this paper is to compute the $E_{\infty}$-term of the ADSS for $M=(E_{2})_*V(0)$. We obtain an associated graded for $H^*(\Sn_2^1;(E_2)_*V(0))$. Taking the Galois fixed points of the $E_{\infty}$-term gives an associated graded for the cohomology $H^*(\G_2^1;(E_{2})_*V(0))$. Therefore, this computation can be used to understand the $E_2$--term of the descent spectral sequence (\ref{ANSS0}) when $G = \G_2^1$ and $X = V(0)$, that is
\[ H^s(\G_2^1 ; (E_{2})_tV(0) ) \Longrightarrow \pi_{t-s}(E_2^{h\G_2^1} \smsh V(0)).\]
Further, recall that there is a fiber sequence
\[ L_{K(2)}V(0) \ra E_2^{h\G_2^1}\smsh V(0)  \ra E_2^{h\G_2^1}\smsh V(0).\]
Hence, computing $H^*(\Sn_2^1;(E_2)_*V(0))$ is a first step for computing $\pi_*L_{K(2)}V(0)$.

The computations will be done using the fact that, at chromatic level $n=2$, one can replace Morava $K$-theory $K(2)$ by a spectrum $K_{\cC}$ whose formal group law is the formal group law of a supersingular elliptic curve $\cC$. This allows us to use the geometry of elliptic curves to get a better understanding of the action of the Morava stabilizer group $\Sn_2$ on $(E_2)_*$. Before stating the results, we explain this point of view. 

Consider the supersingular elliptic curve $\mathcal{C} : y^2+y=x^3$ defined over $\F_4$. Let $F_{\mathcal{C}}$ be the formal group law of $\cC$. It satisfies
$[-2]_{F_{\mathcal{C}}}(x) = x^4$. 
Let $K_{\cC}$ denote the complex oriented ring spectrum whose ring of coefficients is $({K_{\mathcal{C}}})_* = \F_4[u^{\pm 1}]$, 
where $u$ is in degree $-2$, and whose formal group law is $F_{\mathcal{C}}$. 

In this paper, $E_{\cC}:= E(\F_4, F_{\cC})$ will denote the complex oriented Landweber exact spectrum for which $\pi_0E_{\mathcal{C}}$ corepresents isomorphism classes
of deformations of $F_{\mathcal{C}}$. There is an abstract isomorphism
$(E_{\cC})_* \cong (E_2)_*$, but it cannot be realized by a map of $E_{\infty}$-ring spectra. Such a map would induce an $\F_4$-isomorphism on the formal group laws $F_{\cC}$ and $F_2$. These formal group laws are not isomorphic over $\F_4$, but become isomorphic after passing to the algebraic closure of $\F_2$.

Let $\Sn_{\cC}:= \Aut(F_{\mathcal{C}})$ be the group of automorphisms of $F_{\mathcal{C}}$ over $\F_4$. The groups $\Sn_2$ and $\Sn_{\mathcal{C}}$ are isomorphic. An explicit isomorphism is constructed in \fullref{thm:isoS2SC}. The group $\Sn_{\cC}$ admits an action of the Galois group and the group $\G_{\cC}$ is the extension of $\Sn_{\cC}$ by this action. The group $\G_{\cC}$ acts on the deformations. By the Goerss-Hopkins-Miller Theorem (see Goerss and Hopkins \cite[Section 7]{GMmod}), it acts on $E_{\cC}$ by maps of ${E}_{\infty}$-ring spectra.

The isomorphism of \fullref{thm:isoS2SC} does not extend to an isomorphism of the groups $\G_2$ and $\G_{\cC}$. In fact, these groups are not isomorphic. However, over an algebraic closure of $\F_2$, the formal group laws $F_2$ and $F_{\mathcal{C}}$ are isomorphic. Therefore, the Bousfield classes of $K(2)$ and $K_{\cC}$ are the same.
Their localization functors are weakly equivalent, so that
$L_{K(2)}X \simeq L_{K_{\cC}}X$. As before, it follows from the work of Devinatz and Hopkins in \cite{MR2030586} that for $X$ a finite spectrum $L_{K_{\cC}} X \simeq E_{\cC}^{h\G_{\cC}}\smsh X$. Further, for any closed subgroup $G$ of $\G_{\cC}$, there is a spectral sequence analogous to (\ref{ANSS0}). 

The groups $\Sn_{\cC}$ and $\G_{\cC}$ also admit a surjective homomorphism to $\Z_2$ and $\Sn_{\cC}^1$ and $\G_{\cC}^1$ are defined to be the kernel of this homomorphism as before. Further, since $\Sn_{\cC}$ is isomorphic to $\Sn_{2}$, the results of \cite{Paper1} also hold for $\Sn_{\cC}$. In particular, the resolution of \fullref{thm:resnormS2} can be constructed using $\Sn_{\cC}^1$ and the algebraic duality resolution gives rise to an algebraic duality resolution spectral sequence 
\begin{equation}\label{eqn:ADRSSV0C}
E_1^{p,q} \cong H^q(F_p, M)  \Longrightarrow H^{p+q}(\Sn_{\cC}^1;(E_{\cC})_*V(0)). 
\end{equation}
In this paper, we compute the $E_{\infty}$-term of (\ref{eqn:ADRSSV0C}). 

The main advantage of using $\Sn_{\cC}$ is that the elliptic curve $\cC$ has a large automorphism group.
In fact, $\Aut(\cC)$ is isomorphic to $G_{24}$ and it injects into $\Aut(F_{\cC})$. Its image is a choice of maximal finite subgroup. Using level structures, Strickland has computed the action of $\Aut(\cC)$ on $(E_{\cC})_*$. We use this result and, since it is not in print, we describe it \fullref{subsec:AutC}. From now on, we will let $G_{24}$ denote the image of $\Aut(\cC)$ in $\Sn_{\cC}$.

\subsection{Statement of Results} In order to state the results, we will describe the $E_1$-term of (\ref{eqn:ADRSSV0C}). First, note that
$(E_{\cC})_*V(0) \cong \F_4[\![u_1]\!][u^{\pm1}]$, where $u_1$ has degree $0$ and $u$ has degree $-2$. Let $F_{E_{\cC}}$ be the graded formal group law of $E_{\cC}$. Then
\begin{align*}
[2]_{F_{E_{\cC}}}(x) \equiv u_1u^{-1}x^2+\ldots \mod (2),
\end{align*}
(see \fullref{sec:appC}), hence we define $v_1 = u_1 u^{-1}$ in $(E_{\cC})_*V(0)$. The element $v_1$ is invariant under the action of $\Sn_{\cC}$ on $(E_{\cC})_*V(0)$.
Let $\delta$ be the connecting homomorphism associated to the exact sequence
\[0 \ra (E_2)_* \xra{2} (E_2)_* \ra (E_2)_*V(0) \ra 0. \]
Let $\eta = \delta(v_1)$ and $v_2 = u^{-3}$. Then
\begin{align*}
H^*(C_{6} ; (E_{\cC})_*V(0)) \cong \F_4[\![u_1^3]\!][v_1, v_2^{\pm 1}, h]/(v_2^{-1}v_1^3 = u_1^3),
\end{align*}
for a class $h$ in $H^1(C_{6}; (E_{\cC})_*V(0))$ satisfying $\eta  = hv_1$. In particular,
$\{v_2^{n}h^s\}_{n \in \Z}$ is a set topological generators of $H^s(C_{6} ; (E_{\cC})_*V(0))$ as an $\F_4[v_1]$--module.
That is, in the category of profinite graded $\F_4[v_1]$--modules, there is an isomorphism
\[H^s(C_{6} ; (E_{\cC})_*V(0)) \cong  \prod_{n \in \Z} \F_4[v_1]\{v_2^n h^s\} .\]

The cohomology $G_{24}$ is related to the cohomology of the Hopf algebroid classifying Weierstrass curves over $\mathbb{F}_4$ with their strict isomorphisms, a computation originally due to Hopkins and Mahowald and presented by Bauer in \cite{tbauer}. In particular, the $G_{24}$ fixed points are related to modular forms modulo $2$. However, we have included a self--contained computation of $H^*(G_{24}, (E_{\cC})_*V(0))$ in an appendix (see \ref{sec:appcoh}). This computation is based on unpublished notes of Hans-Werner Henn. In \ref{sec:appcoh}, it is shown that there is an isomorphism
\[ H^*(G_{24}, (E_{\cC})_*V(0))\cong \F_4[\![j]\!][v_1,\Delta^{\pm 1}, k, \eta, \nu, x, y]/(\sim)\]
where $\Delta$ of degree $(0, 24)$ is the reduction modulo $(2)$ of the discriminant of a universal deformation of $\cC$ over $(E_{\cC})_*$ (see \fullref{thm:CU}), $j$ of degree $(0,0)$ is the $j$--invariant of this deformation. The element $k$ in $ H^4(G_{24}, (E_{\cC})_*V(0))$ is the image of the periodicity generator in $ H^4(G_{24}, \F_4) \cong  H^4(Q_8, \F_4)^{C_3}$ under the natural inclusion of $\F_4$ into $(E_{\cC})_0V(0)$.
The class $\nu$ has degree $(1,4)$, $x$ has degree $(1,8)$ and $y$ has degree $(1,16)$.
The relations $(\sim)$ contain $ \eta^4=v_1^4 k$ and $ v_1^{12}=j\Delta$. We refer the reader to \fullref{TBG24} for the complete ideal of relations.

A set of topological $\F_4[v_1]$--module generators for $H^0(G_{24},(E_{\cC})_*V(0))$ is given by $\{\Delta^n\}_{n\in \Z}$ so that, in the category of profinite graded $\F_4[v_1]$--modules,
\[H^0(G_{24} ; (E_{\cC})_*V(0)) \cong  \prod_{n \in \Z} \F_4[v_1]\{\Delta^n\} .\]
Note that conjugation by $\pi$ induces an $\F_4[v_1,\eta]$--linear isomorphism 
\[H^*(G_{24}', (E_{\cC})_*V(0)) \cong H^*(G_{24}, (E_{\cC})_*V(0)).\] 
We let $\Delta'$ and $j'$ be the image of $\Delta$ and $j$ under this isomorphism. For $z$ in positive cohomological dimension in $H^*(G_{24}, (E_{\cC})_*V(0))$, we abuse notation and denote its image under the conjugation isomorphism by the same name. In this spirit, letting $k$ act on the left via this isomorphism, we will treat $H^*(G_{24}', (E_{\cC})_*V(0))$ as an $\F_4[v_1,\eta,k]$--module (see \fullref{rem:deltaprime}).

Finally, $H^*(C_{6}, (E_{\cC})_*V(0))$ is an $\F_4[v_1, \eta, k]$--module where $k$ acts by multiplication by $h^4$ and $\eta$ by multiplication by $v_1h$.

In the following result, we adopt notation similar to that of Henn, Karamanov and Mahowald \cite[Theorem 1.2]{HKM}.
\begin{thm}
\label{thm:E2}
The ADSS converging to $H^*(\Sn_2^1,(E_{\cC})_*V(0))$ collapses at the $E_2$-term. The spectral sequence is an $\F_4[v_1, \eta, k]/(\eta^4-v_1^4k)$--module. There exist $\F_4[v_1]$-generators $\Delta_n \in E_1^{0,0}$, $b_n \in E_1^{1,0}$, $\ob_n \in E_1^{2,0}$ and $\od_n \in E_1^{3,0}$ with
\begin{align*}
\Delta_n & \equiv \Delta^n  &  \db_n &\equiv \ob_n \equiv v_2^n,  & \od_n &\equiv (\Delta')^n
\end{align*}
where the congruences are modulo the ideal $(v_1)$
and such that, for $r\geq 0$ and $t \in \Z$,
\begin{align*}
d_1( \dd_n) & =\begin{cases}
v_1^{6\cdot 2^{r}}\db_{2^{r+1}(1+4t)}  & \text{$n = 2^r(1+2t)$} \\
    0 &  \text{$n=0$}.\end{cases} \\ 
d_1( \db_{n}) & =\begin{cases}
v_1^{3\cdot 2^{r}}\ob_{2^{r+1}(1+2t)}  & \text{$n = 2^r(3+4t)$} \\
     v_1^{3\cdot 2^{r+1}}\ob_{1+2^{r+1}(1+4t)} & \text{$n = 1+2^{r+2}(1+2t)$} \\
    0 & \text{otherwise} \end{cases} \\ 
 d_1( \ob_{n}) & =\begin{cases}
 v_1^{3(2^{r+1}+1)}\od_{2^{r}(1+2t)}   & \text{$n=1+2^{r+1}(3+4t)$} \\
0 & \text{otherwise}.\end{cases}
 \end{align*}
For $q>0$, a the differential $d_1 : E_1^{p,q} \ra E_{1}^{p,q}$ is non-zero if and only if it is forced by $\eta$--linearity and the $\F_4[v_1]$--module structure. All differentials $d_r : E_r^{p,q} \ra E_{r}^{p+1,q}$ for $r\geq 2$ are zero, so that $E_2 =E_{\infty}$.
\end{thm}

\fullref{thm:E2des} below gives an explicit description of the $E_{\infty}$--term for the interested reader.
\fullref{thm:E2} and \fullref{thm:E2des} are also displayed in \fullref{fig:SSRS-E1} and \fullref{fig:SSRS-E2}.
\begin{thm}\label{thm:E2des}
As an $\mathbb{F}_4[v_1, k]$--module, the $E_{\infty}$--term of the ADSS with coefficients in $(E_{\cC})_*V(0)$ is isomorphic to a direct sum of cyclic modules generated by the following elements and with the following annihilator ideals. 
\begin{enumerate}[(a)]
\item For $E_{\infty}^{0, *}$,
\begin{align*}
&\eta^{s} \dd_0  &  & 0\leq s\leq 3 & &  \\
&\nu^{s}  \dd_t &  & 1 \leq s\leq 3, \ t\in \Z &  & (v_1) \\
& \eta^{s} x^r \dd_t  & & 1 \leq r \leq 2, \ 0\leq s \leq 1, \ t\in \Z  & & (v_1^2) \\
&  \nu^s y \dd_t   & & 0 \leq s \leq 2, \ t\in \Z & & (v_1) 
\end{align*}
\item For $E_{\infty}^{1, *}$,
\begin{align*}
& h^q \db_{s}  &  & 0 \leq s \leq 1, \ 0 \leq q \leq 3 & &  \\
&h^q \db_{2^{r+1}(1+4t)}   &  & 0\leq r, \ 0 \leq q \leq 3  &  & (v_1^{6 \cdot 2^r+\overline{q}}) 
\end{align*}
\item For $E_{\infty}^{2, *}$,
\begin{align*}
& h^q  \ob_{s}  &  & 0 \leq s \leq 1, \ 0 \leq q \leq 3 &  \\
&h^q  \ob_{2^{r+1}(1+2t)}   &  & 0\leq r, \ t \in \Z, \  0 \leq q \leq 3  &  & (v_1^{3\cdot 2^{r}})  \\
& h^q \ob_{1+2^{r+1}(1+4t)}  & &0\leq r, \ t \in \Z, \   0 \leq q \leq 3  & & (v_1^{3\cdot 2^{r+1}})
\end{align*}
\item For $E_{\infty}^{3, *}$,
\begin{align*}
&  \eta^q \od_0 & & 0 \leq q \leq 3 & &  \\
& \eta^q \od_{2^{r}(1+2t)}   & & 0\leq r, \ t \in \Z, \  0 \leq q \leq 3   & & (v_1^{3 \cdot (2^{r+1}+1)-q}) \\
&\nu^{s} \od_t  &  & 1 \leq s\leq 3, \ t\in \Z &  & (v_1) \\
&\eta^{s} x^r \od_t   & & 1 \leq r \leq 2, \ 0\leq s \leq 1, \ t\in \Z  & & (v_1^2) \\
&  \nu^s y  \od_t  & & 0 \leq s \leq 2, \ t\in \Z & & (v_1) 
\end{align*} 
\end{enumerate}
\end{thm}
If one inverts $v_1$, the situation is much simpler. The following result is an immediate consequence of \fullref{thm:E2} and \fullref{thm:E2des}.
\begin{cor}\label{cor:v1inv}
The spectral sequence
\begin{equation*}
v_1^{-1}E_1^{p,q} = v_1^{-1}H^q(F_p, (E_{\cC})_*V(0))   \Longrightarrow v_1^{-1}H^{p+q}(\Sn_{\cC}^1;(E_{\cC})_*V(0))
\end{equation*}
collapses at the $E_2$--term. As an $\F_4[v_1^{\pm 1},\eta]$--module,
\[v_1^{-1}H^*(\Sn_{\cC}^1, (E_{\cC})_*V(0)) \cong \F_4[v_1^{\pm 1},\eta]\{\dd_0, \db_0, \db_1, \ob_0, \ob_1, \od_0\} .\]
Let $s$ be the cohomological degree and $t$ be the internal degree. Then the $(s,t)$--degrees of the generators are $|\dd_0| =(0,0)$, $|\db_0| = (1,0)$, $|\db_1| = (1,6)$, $|\ob_0| = (2,0)$, $|\ob_1| = (2,6)$ and $|\od_0 |= (3,0)$.
\end{cor}
\noindent
This result is important because $v_1^{-1}H^*(\Sn_{\cC}^1, (E_{\cC})_*V(0)) $ is the $E_2$--term of a spectral sequence that computes the homotopy groups of $L_1(E_2^{h\Sn_{\cC}^1} \wedge V(0))$. This spectrum plays a central role in the study of the chromatic splitting conjecture at $n=p=2$.

It is worth mentioning here that the related computation of the $E_2$-page of the Johnson-Wilson $E(2)$-local Adams-Novikov spectral sequence converging to $L_2S$ was done by Shimomura and Wang in \cite{MR1935487}. These computations were done independently. However, historically, they depend on the work of Shimomura and Wang. Indeed, results similar to those of \fullref{thm:E2} can be extracted from \cite{MR1935487}, and it is using Shimomura and Wang's computation that Mahowald conjectured the existence of the duality resolution for the $K(2)$-local sphere. 

\subsection{Organization of the paper}

In \fullref{subsec:SSC}, we describe a choice of universal deformation $F_{\cC_U}$ of $F_{\cC}$, where $F_{\cC_U}$ is the formal group law of an elliptic curve $\cC_{U}$. This allows us to define $E_{\cC}$. The choice for the curve $\cC_{U}$ is due to Strickland and, in \fullref{subsec:AutC}, we present his formulas for the right action of $\Aut(\cC)$ on $(E_{\cC})_*$. In \fullref{sec:ActGn}, we tie this to the right action of $\Aut(F_{\cC})$ on $(E_{\cC})_*$.
In \fullref{sec:leftact} we adopt conventions that allow us to use the corresponding left action as in Henn, Karamanov and Henn \cite{HKM} and settle our notation for the rest of the paper.

\fullref{sec:TvS} is dedicated to describing the structure of $\Sn_{\cC}$. In \fullref{subsec:isoS2SC}, we give an explicit isomorphism between the group of automorphisms $\Sn_2$ of the Honda formal group law and the group $\Sn_{\cC}$. In \fullref{subsec:filt-norm}, we recall the standard filtration on $\Sn_{\cC}$. In \fullref{subsec:act}, we give the information about the action $\Sn_{\cC}$ on $(E_{\cC})_*$ that will be used in the computation of $H^*(\Sn_{\cC}^1, (E_{\cC})_*V(0))$.

The goal of \fullref{sec:E1} is to introduce the ADSS for $\Sn_{\cC}$ and to give the information necessary to begin the computation. The ADSS is not multiplicative, but it has some nice properties which we describe in \fullref{subsec:ss}. In \fullref{subsec:E1}, we describe the $E_1$-term.  The discriminant $\Delta$ of the curve $\cC_U$ has useful linearity properties which are given in \fullref{subsec:Dlin}.

The bulk of the paper is the computation of the $E_{\infty}$-term of the ADSS with coefficients in $(E_{\cC})_*V(0)$. This is done in \fullref{sec:compe2}. In \fullref{subsec:d10}, \fullref{subsec:d12} and \fullref{subsec:d13}, we compute the differentials $d_1 : E_1^{p,0} \ra E_{1}^{p+1,0}$. In \fullref{subsec:d1rest}, we compute the differentials $d_1 : E_1^{p,q} \ra E_1^{p+1,q}$ for $q>0$. In \fullref{sec:higher}, we prove that all differentials $d_r : E_r^{p,q} \ra E_r^{p,1}$ for $r\geq 2$ are zero.

In \fullref{sec:appC}, we describe the action of $\Sn_{\cC}$ on $(E_{\cC})_*$ and deduce the formulas used in our computations. We give formulas for the minus two series of $\cC$ and $\cC_{U}$ and use them to give estimates for the action. 

\ref{sec:appcoh} gives a self--contained computation of the cohomology of $G_{24}$ with coefficients in $(E_{\cC})_*V(0)$. It consists of unpublished notes of by Hans-Werner Henn, which were edited by the author. The author is grateful for his blessing to include them in this document.

\subsection{Acknowledgements} 
I thank Paul Goerss, Hans-Werner Henn and Peter May for their unfaltering support. I thank Neil Strickland for his unpublished notes, which were crucial to these computations. I thank Mark Mahowald whose insight continues to inspire me.

This material is based upon work supported by the National Science Foundation under Grant No. DMS-1612020.

\section{Morava $E$-theory and Elliptic Curves}\label{sec:Etheory}

In this section, we define the spectrum $E_{\cC}$ and compute the action of $\Aut(\cC)$ on $(E_{\cC})_*$. This computation is due to Strickland. For the deformation theory of formal group laws, we refer the reader to Rezk \cite{MR1642902} or Goerss and Hopkins \cite[Section 7]{GMmod}. In particular, if $F$ is a formal group law over a field $k$, then $E(F,k)$ denotes the associated Lubin-Tate spectrum.

\subsection{The supersingular elliptic curve}\label{subsec:SSC} 

Consider the elliptic curve $\cC : y^2+y = x^3$ over $\F_4$. It is a standard fact that the formal group law of $\cC$ has height $2$ (see Silverman \cite[Section V.4]{silverman} or \fullref{minustwomod2} below) so that $\cC$ is a supersingular elliptic curve. Elliptic curves over fields of characteristic $p>0$ admit a theory of deformations which is analogous to that of formal group laws. It follows from the Serre--Tate theorem that the deformation theory of a supersingular elliptic curve is equivalent to that of its formal group law. However, in our case, we make this concrete by the following result.
\begin{thm}
\label{thm:CU}
The formal group law $F_{\cC_U}$ of the elliptic curve
\begin{align*}
\cC_U : y^2+3u_1xy+(u_1^3-1)y = x^3
\end{align*}
defined over $\W[\![u_1]\!]$ is a universal deformation of $F_{\cC}$.
This specifies a Lubin--Tate spectrum $E_{\cC}=E(F_{\mathcal{\cC}}, \F_4)$ with an isomorphism
$(E_{\cC})_* \cong \W[\![u_1]\!][u^{\pm 1}]$, where $u$ has degree $-2$ and $u_1$ has degree zero,
and a graded formal group law 
\begin{align*}
F_{E_{\cC}} = uF_{\cC_U}(u^{-1}x, u^{-1}y).
\end{align*}
\end{thm}
\begin{proof}
We can verify directly that $F_{\cC_U}$ satisfies the criteria for a universal deformation (see Lubin and Tate \cite[Proposition 1.1]{lubintate}). Indeed, in \fullref{intminustwo} below, we compute that
\[F_{\cC_U}(x,y) \equiv x+y-3u_1xy-2(u_1^3-1)xy(x^2+y^2)-3(u_1^3-1)x^2y^2 \mod (x,y)^5 \]
so that
\begin{align*}
F_{\cC_U}(x,y) &\equiv x+y+ u_1 xy \mod (2, (x,y)^3) \\
F_{\cC_U}(x,y) &\equiv x+y+ x^2y^2 \mod (2, u_1, (x,y)^5).\qedhere
\end{align*}
\end{proof}

\begin{rem}
A more obvious choice of universal deformation of $\cC$ is
\[ \cC_{U}': y^2+u_1xy+y =x^3\]
defined over $ \W[\![u_1]\!]$.
Lubin and Tate \cite[Section 3.5]{lubintate} prove that the formal group law associated to this curve is a universal deformation of $F_{\cC}$. However, the choice of the curve $\cC_U$, due to Strickland, yields nice formulas for the action of $\Aut(\cC)$ on $(E_{\cC})_*$. 
For $\varphi$ the $\W$--linear isomorphism of $\Wu$
determined by $\varphi(u_1) =-3u_1 (1-u_1^3)^{-\frac{1}{3}}$,
the change of coordinates
\begin{align*}
x &=(u_1^3-1)^{-\frac{2}{3}}x' & y&=(u_1^3-1)^{-1}y' \end{align*}
is an isomorphism from $\cC_U$ to $\varphi^*\cC_{U'}$, where $(1-u_1^3)^{-\frac{1}{3}}$ is interpreted as its Taylor expansion.
\end{rem}

\subsection{The automorphisms of $\cC$}\label{subsec:AutC}

The group $\Aut(F_{\cC})$ acts on $(E_{\cC})_*$ and, hence, so does its subgroup $\Aut(\cC)$. In unpublished notes, Strickland has computed the right action of the group $\Aut(\cC)$ on $(E_{\cC})_*$. We explain his results in this section. 

The automorphisms of the supersingular curve $\cC$ are computed in Silverman \cite[Appendix A]{silverman}. The results are stated here without proof.
Fix a primitive third root of unity $\zeta \in \F_4$. For the curve $\cC$ over $\F_4$, the group $\Aut(\cC)$ is generated by the maps on points in the affine chart given by
\begin{align*}
a(x,y) &= (\zeta^2 x,\zeta^3y),\\
b(x,y)&= (x+1,y+x+\zeta^2) .
\end{align*}
The elements $a$ and $b$ generate a group of order $24$ which will be described in more details at the end of \fullref{sec:leftact}.

Let $\psi_{a}, \psi_b :(E_{\cC})_0 \to (E_{\cC})_0 $ be the $\W$--linear isomorphisms determined by
\begin{align}\label{eqn:actab}
   \psi_{a}(u_1) &=\zeta u_1, &
 \psi_{b}(u_1) &= \frac{u_1+2}{u_1-1}.
\end{align}
Since $a$ and $b$ are generators of $\Aut(\cC)$, this determines a right action of $\Aut(\cC)$ on $(E_{\cC})_0$. The action of $\gamma$ is denoted by $\psi_{\gamma}$.

For $\ell$, $r$, $s$ and $t$ in $\Wu$ and $\ell$ a unit, let 
\[A(\ell, r,s,t) : \mathbb{P}^2(\Wu) \to \mathbb{P}^2(\Wu)\] 
be the automorphism
\[A(\ell, r,s,t)[x:y:z] = [\ell^2x+rz: \ell^3 y + \ell^2 s x + tz:z].\] 
Let
\begin{align*}
f_{a} &= A(\zeta, 0,0,0)\\
f_b &= A\left(\frac{\zeta^2-\zeta}{u_1-1},3\frac{1-u_1^3}{(u_1-1)^3} ,3\frac{\zeta^2u_1-1}{u_1-1} ,3\frac{u_1^3-1}{(u_1-1)^4}((1-\zeta)+(1-\zeta^2)u_1) \right).
\end{align*}
This determines a left action of $\Aut(\cC)$ on $\mathbb{P}^2(\Wu)$ and, for $\gamma$ in $\Aut(\cC)$, let $f_{\gamma} = A(\ell_{\gamma}, r_\gamma,s_\gamma,t_\gamma)$ be the corresponding automorphism. Then $f_{\gamma}$ induces an isomorphism $f_{\gamma} : \cC_{U} \to \psi_{\gamma}^*\cC_{U}$ which pulls back to $\gamma$ over $\F_4$.

The pair $(f_{\gamma}, \psi_{\gamma})$ induces an isomorphism of formal group laws. (For the formal group law associated to an elliptic curve, see Silverman \cite[Chapter IV]{silverman}). We abuse notation and denote this isomorphism 
\[ f_{\gamma} : F_{\cC_U} \ra \psi_{\gamma}^*F_{\cC_{U}} \]
by $(f_{\gamma}, \psi_{\gamma})$ for $f_{\gamma}$ the corresponding power series in $(E_{\cC})_0 [\![x]\!]$. One can verify directly by chasing through the change of coordinates in \cite[Chapter IV]{silverman} that
\begin{equation}\label{eqn:f'0}
f_{\gamma}'(0) =  \ell_{\gamma}^{-1}.
\end{equation}

\subsection{The right action of $\Sn_{\cC}$}\label{sec:ActGn}

We first describe a right action $\Sn_{\cC} = \Aut(F_{\cC})$
on $E_{\cC_U}$, following Rezk \cite{MR1642902} as this meshes well with the geometry of the elliptic curves. However, it will be convenient in our computations to have a left action of $\Sn_{\cC}$ and the reader should be warned that in \fullref{sec:leftact}, we will make the switch from a right to a left action and adopt the conventions of Henn, Karamanov and Mahowald in \cite[Section 4]{HKM}.

Throughout, for $\psi: R_1 \ra R_2$ a ring homomorphism and $F$ a formal group law over a ring $R_1$, then
\[\psi^*F(x,y) =  \sum_{i,j} \psi(a_{i,j})x^{i}y^j\] 
is a formal group law over $R_2$. 

An element $\gamma \in \Sn_{\cC}$ is a power series in $\F_4[\![x]\!]$. Let $g$ be any lift of $\gamma$ in  $(E_{\cC})_0[\![x]\!]$. Define a new formal group law by
\begin{align*}
F_{g}(x,y) = g F_{\cC_U}(g^{-1}(x),g^{-1}(y)).
\end{align*}
Then $F_g$ is a deformation of $F_{\cC}$ over $(E_{\cC})_0$ and there is a unique ring isomorphism 
$\psi_{\gamma}: (E_{\cC})_0 \rightarrow (E_{\cC})_0$
and a unique $\star$-isomorphism
$f_g:   F_g \rightarrow (\psi_{\gamma})^*F_{\cC_U}$
which classify $F_g$. If $g'$ is another lift of $\gamma$, then 
\begin{align*}
f_{g'}= f_g \circ (g (g')^{-1}) :  F_{g'} \rightarrow   (\psi_{\gamma})^*F_{\cC_U}
\end{align*}
is a $\star$-isomorphism. Therefore, $\psi_{\gamma}$ is independent of the choice of lift $g$. The composite
\begin{align*}
F_{\cC_U}  \xra{g} F_{g} \xra{f_g}  (\psi_{\gamma})^*F_{\cC_U}
\end{align*}
is also independent of $g$ and is denoted by $f_{\gamma}$. This gives a right action of $\Sn_{\cC}$ on $(E_{\cC})_0$ where $\gamma$ acts via $\psi_{\gamma}$. This action extends to an action of $\Sn_{\cC}$ on $(E_{\cC})_*$ with
\begin{align}\label{eqn:psiuu}
\psi_{\gamma}(u) := f_{\gamma}'(0)^{-1}u
\end{align}
(see Rezk \cite[Section 6.7]{MR1642902}, noting that $u_{F_i}$ in this reference is in degree $2$, while our $u$ is in degree $-2$ and corresponds to the inverse of $u_{F_i}$).
Note that by the Goerss-Hopkins-Miller theorem, this action can be realized through maps of $E_{\infty}$-ring spectra on $E_{\cC_U}$ (see Goerss and Hopkins \cite{GMmod}). 

Further, $\Gal(\F_4/\F_2)$ acts on $\W$; hence, it acts on the coefficients $(E_{\cC_U})_*$. Since $F_{\cC}$ is fixed by $\Gal(\F_4/\F_2)$, this extends the action of $\Sn_{\cC}$ to an action of 
\[\G_{\cC} = \Sn_{\cC} \rtimes \Gal(\F_4/\F_2).\]

Now, let $\gamma$ in $\Sn_{\cC}= \Aut(F_{\cC})$ be induced by an element of $\Aut(\cC)$. The action of $\gamma$ is classified by the pair $(f_{\gamma},\psi_{\gamma})$ described in \fullref{subsec:AutC}.
Hence, the action of $\psi_{\gamma}$ is determined by the formulas (\ref{eqn:actab}), (\ref{eqn:f'0}) and (\ref{eqn:psiuu}). In particular, 
\begin{equation*}
 \psi_{\gamma}(u) =f_\gamma'(0)^{-1} u = \ell_{\gamma} u. \end{equation*}

\subsection{The left action of $\Sn_{\cC}$}\label{sec:leftact}
The left action of an element $\gamma$ in $\Sn_{\cC}$ on $(E_{\cC})_*$ is naturally given by $\psi_{\gamma^{-1}}$. For convenience, we thus adopt the notation
\begin{align*} \phi_{\gamma} &= \psi_{\gamma^{-1}} & h_{\gamma} &= (f_{\gamma^{-1}})^{-1}
\end{align*}
so that $h_{\gamma} : \phi_{\gamma}^*F_{\cC_U} \to F_{\cC_{U}}$. With these conventions, we have a left action of $\Sn_{\cC}$ on $(E_{\cC})_*$ which satisfies
\begin{align}\label{eqn:phiu}
\phi_{\gamma}(u)& = h_{\gamma}'(0)u.
\end{align}
These correspond to the conventions of Henn, Karamanov and Mahowald in \cite[Section 4]{HKM}.

Finally, we fix our notation for the group $G_{24}=\Aut(\cC)$. We let
\begin{align*}
 \omega &= a^{-1}  & i &= b^{-1} 
\end{align*}
so that $\phi_{\omega} = \psi_a$ and $\phi_i = \psi_b$. It then follows that
\begin{align*}
\phi_{\omega}(u_1) &=\zeta u_1, &   \phi_{\omega}(u) &=\zeta u  \\
\phi_i (u_1) &= \frac{u_1+2}{u_1-1}, &  \phi_{i}(u) &= u\frac{\zeta^2-\zeta}{u_1-1}.
\end{align*} 
The element $\omega$ has order $3$ and, from now on, we denote the group it generates by $C_3$. The element $i$ has order four and $i^2=-1$, the inversion of the curve $\cC$. We let
\begin{align*}
j&:=\omega i \omega^2, & k&:=\omega^2 i \omega
\end{align*}
and note that $ij=k$. The elements $i$ and $j$ generate a normal subgroup isomorphic to the quaternions $Q_8$ and hence
$G_{24} \cong \Aut(\cC) = Q_8 \rtimes C_3$.

\section{The Morava stabilizer group}\label{sec:TvS}

The Morava stabilizer group $\Sn_{2}$ is the group of automorphisms of the Honda formal group law $F_2$, which is the $p$-typical formal group law over $\F_4$ specified by the series $\![2]_{F_2}(x) =x^4$. The standard presentation for $\Sn_{2}$ is the non-commutative extension 
\begin{align*}
\Sn_{2} \cong \left(\W\left< S \right>/(S^2=2, aS=Sa^{\sigma})\right)^{\times},
\end{align*}
where $S$ is the automorphism $S(x) =x^2$, $a \in \W$ and $\sigma$ is the Frobenius (see Ravenel \cite[Appendix A2]{MR860042} for more details). In this section, we specify an isomorphism $\Sn_{2} \cong \Sn_{\cC}$,
whose construction is due to Henn. We also recall some of the key properties of the structure of the group $\Sn_{2}$, which transfer to properties of $\Sn_{\cC}$ via this isomorphism.

\subsection{The isomorphism of $\Sn_2$ and $\Sn_{\cC}$}\label{subsec:isoS2SC}
As opposed to the Honda formal group law, it is the $[-2]$-series
of the formal group law $F_{\mathcal{C}}$ which has a nice presentation. The following result is proved in \fullref{minustwomod2} of \fullref{sec:appC}.
\begin{lem}\label{lem:-2}
Let $\cC: y^2+y=x^3$ be defined over a field of characteristic two. If $F_{\mathcal{C}}$ is the associated formal group law, then $[-2]_{F_{\mathcal{C}}}(x) = x^4$.
\end{lem}

The curve $\cC$ and its formal group law $F_{\mathcal{C}}$ are defined over $\F_2$. Therefore, 
$T(x) = x^2$ is an endomorphism of $F_{\mathcal{C}}$. \fullref{lem:-2} implies that $T(T(x)) = [-2](x)$. The element $\omega$ defined in \fullref{subsec:AutC} induces the isomorphism
$\omega(x) = \zeta x $
of $F_{\mathcal{C}}$, so that $\omega T = T\omega^{\sigma}$.
This shows that
\begin{align*}
\W\left< T \right>/(T^2=-2,  \omega T = T \omega^{\sigma}) \subseteq \End(F_{\mathcal{C}})
\end{align*}
and this is in fact an equality (see for example Hazewinkel \cite[Proposition 21.8.17]{MR506881}). Therefore, 
\begin{align*}
\Sn_{\cC} \cong \left(\W\left< T \right>/(T^2=-2,  \omega T = T \omega^{\sigma})\right)^{\times}.
\end{align*}
The action of the Galois group on $\W$ induces an action on $\Sn_{\cC}$ defined by $ \sigma(a+bT) = a^{\sigma}+b^{\sigma}T$.

Let $\alpha = (1-2\omega)(-7)^{-\frac{1}{2}}$, where $(-7)^{\frac{1}{2}}$ in $\W$ is chosen to be congruent to $1$ modulo $(4)$. It follows that $\alpha = 1+\omega T^2 + T^4$ modulo $(T^6)$ and $\alpha \alpha^{\sigma}=-1$. The proof of the following result follows immediately.
\begin{lem}\label{thm:isoS2SC}
The map $\phi: \Sn_{\cC} \rightarrow \Sn_{2}$ given by
$\phi(a+bT) = a+b(\alpha S)$
is an isomorphism.
\end{lem}

\subsection{The filtration and the norm}\label{subsec:filt-norm}
\fullref{thm:isoS2SC} implies that all the results of \cite{Paper1} can be restated for the group $\Sn_{\cC}$ instead of $\Sn_2$. Here, we briefly review the results which will be important for the computations of this paper.

As in \cite{Paper1}, any element $\gamma \in \Sn_{\cC}$ can be expressed as a power series
\begin{align*}
\gamma = \sum_{n=0}^{\infty}a_nT^n,
\end{align*}
where the $a_i$'s satisfy the equation $x^4-x=0$ and $a_0 \neq 0$. Let $F_{0/2}\Sn_{\cC} := \Sn_{\cC}$. For $n>0$, let
\begin{align*}
F_{n/2}\Sn_{\cC} := \{\gamma \in \Sn_{\cC} \ | \ \gamma \equiv 1 \mod T^{n} \ \}.
\end{align*}
Define $S_{\cC} := F_{1/2}\Sn_{\cC}$. Then $S_{\cC}$ is the $2$-Sylow subgroup of $\Sn_{\cC}$. This filtration is compatible with the $2$-adic filtration on $\W^{\times}$. Further, $\{F_{n/2}\Sn_{\cC}\}_{n \geq 0}$ forms a system of open subgroups and $\Sn_{\cC}$ is a profinite topological group. 

The group $\Sn_{\cC}$ acts on $\End(F_{\mathcal{C}})$ by right multiplication. This gives rise to a representation $\rho: \Sn_{\cC} \rightarrow GL_2(\W)$, given by
\begin{align*} \rho(a+bT) =  \left( \begin{matrix}
a & -2b^{\sigma}  \\
b & a^{\sigma} \end{matrix} \right).
\end{align*}
The restriction of the determinant to $\Sn_{\cC}$ is given by $\det(a+bT) = aa^{\sigma}+2bb^{\sigma}$.
Therefore, the determinant induces a map $\det : \Sn_{\cC} \rightarrow \Z_2^{\times}$. The \emph{norm} is defined as the composite
\begin{align*}
N : \Sn_{\cC} \xra{\det} \Z_2^{\times} \rightarrow \Z_2^{\times}/\{\pm1\} \cong \Z_2.
\end{align*}
The norm is split surjective. Indeed, let $\pi = 1+2\omega$.
Then $\det(\pi) =3$ projects to a topological generator of $\Z_2^{\times}/\{\pm1\}$. The subgroup $\Sn_{\cC}^1$ is then defined by the short exact sequence,
\begin{align*}
1 \rightarrow \Sn_{\cC}^1 \rightarrow \Sn_{\cC} \xra{N} \Z_2^{\times}/\{\pm1\} \rightarrow 1,
\end{align*}
and $\Sn_{\cC} \cong \Sn_{\cC}^1 \rtimes \Z_2^{\times}/\{\pm1\}$. Note that $ \Z_2^{\times}/\{\pm1\} \cong \Z_2 $ is torsion-free; hence, $G_{24}$ is a subgroup of $\Sn_{\cC}^1$.

As discussed in \cite[Lemma 2.27]{Paper1}, the group $\Sn_{\cC}$ has a unique conjugacy class of maximal finite subgroups isomorphic to $G_{24}$. The group $\Sn_{\cC}^1$ has two, and they are represented by $G_{24}$ and $G_{24}' = \pi G_{24} \pi^{-1}$.

\subsection{The action of the Morava stablizer group}\label{subsec:act}

In order to compute the cohomology of $\Sn_{\cC}$, we must have some understanding of its action on $(E_{\cC})_*$. Recall from \cite[Theorem 2.29]{Paper1} that $\Sn_{\cC}$ is topologically generated by $G_{24}$, $\pi$ and $\alpha$. Since the action of $G_{24}$ was described in \fullref{subsec:AutC}, it remains to study the action of $\alpha$ and $\pi$.

A concrete method for approximating the action of $\Sn_{\cC}$ on $(E_{\cC})_*$ is explained by Henn, Karamanov and Mahowald in \cite[Section 4]{HKM}. We describe it in \fullref{sec:appC} and give detailed proofs of the results needed for the following computations. For the sake of exposition, we recall the key points here.

It follows from \fullref{int} that, for $\gamma$ in $\Sn_{\cC}$, there exists continuous functions
$t_0 : \mathbb{S}_{\cC} \to (E_{\cC})_0^{\times}$ and $ t_1 : \mathbb{S}_{\cC} \to (E_{\cC})_0$
such that
\begin{align}\label{eqn:actint}
\phi_{\gamma}(u) &= t_{0}(\gamma) u, & \phi_{\gamma}(u_1) &= t_{0}(\gamma)u_1+\frac{2}{3}\frac{t_{1}(\gamma)}{t_0(\gamma)}.
\end{align}
In particular, modulo $(2)$, $\phi_{\gamma}(u_1) \equiv  t_0(\gamma)u_1$ and
$\phi_{\gamma}(u) \equiv t_0(\gamma) u$.
Therefore, $v_1=u_1u^{-1}$ is fixed by the action of $\Sn_{\cC}$ modulo $(2)$.

Any $\gamma \in S_{\cC}$ can be expressed as $\gamma =1+ \sum_{i=1}^{\infty}a_i(\gamma) T^i$ for $a_i(\gamma) \in \W$ satisfying $a_i(\gamma)^4-a_i(\gamma) =0$.
It follows from \fullref{cor:S21modu14} that
\begin{align}\label{eqn:t0gamma12}
t_0(\gamma) \equiv 1+a_1(\gamma)^2u_1+a_1(\gamma)u_1^2 \mod (2,u_1^3).\end{align}
In particular, $\phi_{\gamma} \equiv id \mod (2, u_1)$.

For $\gamma \in F_{2/2}\Sn_{\cC}$, we obtain better approximations. We prove in \fullref{t0gammafinally} that, modulo $(4, 2u_1^2,u_1^{9})$,
\begin{align*}
  t_{0}(\gamma) &\equiv 1+2a_2(\gamma)+2a_3(\gamma)^2u_1+(a_2(\gamma) + a_2(\gamma)^2)u_1^3 +a_3(\gamma)u_1^5 +a_3(\gamma) u_1^8.
\end{align*}
It also follows from \fullref{t0t1again} that $t_1(\gamma) \equiv a_2(\gamma)^2u_1$ modulo $(2, u_1^3)$.

We apply this to study the action of $\alpha$ and $\pi$. Modulo $(T^6)$, we have
\begin{align*}
\alpha &\equiv  1+\omega T^2 +T^4, & \pi &\equiv 1+\omega T^2 +\omega T^4.
\end{align*}
\begin{prop}\label{alphav2}
Let $\gamma=\alpha$ or $\gamma =\pi$. The unit $t_0(\gamma)$ satisfies:
\begin{align*}
t_0(\gamma) &\equiv 1+2\omega + u_1^3  \mod (4, 2u_1^2,u_1^{9}) \\
t_1(\gamma) &\equiv \omega^2u_1 \mod (2, u_1^3).
\end{align*}
Therefore, for $v_2 = u^{-3}$
\begin{align*}
\phi_{\gamma}(v_2) &= v_2+v_1^3 \mod (2, u_1^9).
\end{align*}
Further
\[\phi_{\gamma} \equiv \phi_{\gamma^{-1}} \mod (2,u_1^9).\]
\end{prop}
\begin{proof}
Since $\pi \equiv \alpha \mod T^4$, it follows from $a_i(\alpha)=a_i(\pi)$ for $i=1$ and $i=2$. Therefore, modulo $(2, u_1^9)$, they have the same action.
The claim for $t_0(\gamma)$ follows immediately using the fact that for $\gamma$ either $\alpha$ or $\pi$, the coefficient $a_2(\gamma)=\omega$ and $a_3(\gamma)=0$. 
Modulo $(2)$,
\begin{align*}
\phi_{\gamma}(v_2) &\equiv t_0(\gamma)^{-3}v_2 \\
&\equiv t_0(\gamma)t_0(\gamma)^{-4}v_2 \\
&\equiv t_0(\gamma)v_2 \mod (u_1^{12}),
\end{align*}
which proves the identity
for $\phi_{\gamma}(v_2)$. 

Note that $\gamma^2 \in F_{4/2}\Sn_{\cC}$. It then follows from the formula for $t_0(\gamma^2)$ that  
$t_0({\gamma^2} ) \equiv 1$ modulo $(2, u_1^9)$. Hence, $\phi_{\gamma^2} = \phi_{\gamma} \circ \phi_{\gamma}\equiv id$ modulo $(2, u_1^9)$.
\end{proof}

\section{The algebraic duality resolution spectral sequence}\label{sec:E1}

\subsection{Preliminaries}\label{subsec:ss}

The groups $\Sn_{2}^1$ and $\Sn_{\cC}^1$ are isomorphic. Further, the isomorphism we constructed restricts to the identity on $\W$, so it preserves $\alpha$ and $\pi$ (see \fullref{thm:isoS2SC}). Therefore, \fullref{thm:resnormS2} holds if we replace $\Sn_{2}^1$ by $\Sn_{\cC}^1$.
\begin{rem}
This is a good place to justify the slight differences between \fullref{thm:resnormS2} and the results of \cite{Paper1}. The existence of the resolution is \cite[Theorem 1.2.1]{Paper1}, but for the isomorphic group $\Sn_{2}^1$. However, the descriptions of the maps is different from \cite[Theorem 1.2.6]{Paper1} and this requires an explanation. The map $\partial_1$ is unchanged. However, in the notation of \cite[Theorem 1.2.6]{Paper1}, we replace $\partial_2$ with $g_2^{-1} \circ \partial_2$ and $\partial_3$ with $\partial_3'$. The resulting chain complex of $\Z_2[\![\Sn_{2}^1]\!]$--modules is isomorphic to that of \cite[Theorem 1.2.1]{Paper1}.
As $g_2^{-1}$ is an isomorphism, the map $\F_2\ot_{\Z_2[\![S_{2}^1]\!]} g_2^{-1} : \F_2 \rightarrow \F_2$ is non-zero, so that
\begin{align}\label{fpep}
g_2^{-1}(e_2) = e_2 \mod (2, IS_{2}^1).
\end{align}
Since $e+\alpha \in (2, IS_{2}^1)$, the description of the middle map follows from the fact that $\Theta \equiv e+\alpha$ modulo $(2, (IS_{2}^1)^2)$. The last map clearly satisfies (c).
\end{rem}

In our computation, we will need to use some additional structure in the algebraic duality resolution spectral sequence (ADSS). We record that here.
For any complete $\Z_2[\![\Sn_{\cC}^1]\!]$--modules $A$ and $B$, let 
\[\Ext(A,B) := \Ext_{\Z_2[\![\Sn_{\cC}^1]\!]}(A,B).\] 
If $B$ is an $\Z_2[\![\Sn_{\cC}^1]\!]$--module which is free over the $2$-adics $\Z_2$, then the Bockstein 
\[\beta : \Ext^*(A, B/2) \rightarrow \Ext^{*+1}(A, B/2)\] 
is the connecting homomorphism associated to the exact sequence 
\begin{align*} 
0 \rightarrow B/2 \rightarrow B/4 \rightarrow B/2 \rightarrow 0.
\end{align*}
The algebraic duality resolution
\[0 \rightarrow \sC_3 \rightarrow \sC_2 \rightarrow \sC_1 \rightarrow \sC_0 \rightarrow \Z_2 \rightarrow 0\]
is obtained from splicing exact sequences
\begin{align} \label{N}
0 \rightarrow N_p \rightarrow \sC_p \rightarrow N_{p-1} \rightarrow 0\end{align}
with $\sC_3=N_2$ and $N_{-1} = \Z_2$ (see \cite{Paper1}). For $B$ a profinite $\Z_2[\![\Sn_{\cC}^1]\!]$--module, the spectral sequence 
\[ E_r^{p,q}=  \Ext_{\Z_2[\![\Sn_{\cC}^1]\!]}^q(\sC_p,B) \cong H^q(F_p, B)  \Longrightarrow H^{p+q}(\Sn_{\cC}^1, B)\]
associated to the exact couple
\begin{align}\label{exact-coup}
\xymatrix{ \Ext(N_*, B) \ar@{.}[rr]^-{\delta_*} & &  \Ext(N_{*-1}, B) \ar[dl]^-{r_*}  \\ &  \Ext(\sC_*, B) \ar[ul]^-{i^*} &   }\end{align}
is the ADSS with coefficients in $B$.
In (\ref{exact-coup}), the dotted arrows are the connecting homomorphisms for the exact sequences (\ref{N}).

\begin{lem}\label{lem-bock}
For $x \in E_r^{p,q}$, $\beta(x) \in E_r^{p,q+1}$ and $d_r(\beta(x)) = \beta(d_r(x))$.
\end{lem}
\begin{proof}
The maps $r_*$, $i_*$ and $\delta_*$ in the exact couple (\ref{exact-coup}) commute with $\beta$. A diagram chase shows that $d_r(\beta(x)) = \beta(d_r(x))$.
\end{proof}

\begin{lem}\label{S21-lin}
Let $R$ be an $\Z_2[\![\Sn_{\cC}^1]\!]$--module which is also a ring. Suppose that the action of $\Sn_{\cC}^1$ is given by ring homomorphisms. The ADSS with coefficients $R$ is a module over the cohomology $H^*(\Sn_{\cC}^1;R)$.
\end{lem}
\begin{proof}
Note that $\Ext(A, R)$ is a module over $\Ext(\Z_2, R)$ for any $\Z_2[\![\Sn_{\cC}^1]\!]$--module $A$. Further, the maps in the algebraic duality resolution are maps of $\Z_2$--modules. Therefore, the maps $r_*$, $i_*$ and $\delta_*$ in (\ref{exact-coup}) are morphisms of $\Ext(\Z_2, R)$--modules, hence so are the differentials in the ADSS.
\end{proof}

Recall that
\[(E_{\cC})_*V(0) \cong (E_{\cC})_*/(2) \cong \F_4[\![u_1]\!][u^{\pm 1}].\] 
The spectrum $E_{\cC}$ was chosen so that $F_{E_{\cC}} = uF_{\cC_U}(u^{-1}x, u^{-1}y)$, where $\cC_U$ is the curve
\begin{align*}
\cC_U : y^2+3u_1xy+(u_1^3-1)y = x^3
\end{align*}
(see \fullref{thm:CU}).
It follows from Silverman \cite[Section IV.1]{silverman} that
\[\![2]_{F_{E_{\cC}}}(x) \equiv u^{-1}u_1x^2+u^{-3}(u_1^3+1)x^4 + \ldots \mod (2).\]
We adopt the notation
\begin{align}\label{eqn:v1v2}
v_1&=u^{-1}u_1, & v_2&=u^{-3}.
\end{align}
\begin{warn}
The reader should note that the formal group law $F_{E_{\cC}}$ is not $2$--typical. The image of the Araki generators ``$v_1$'' and ``$v_2$'' under the map $\varphi: BP_* \to (E_{\cC})_*$ which classifies the $2$--typification of $F_{E_{\cC}}$ does not correspond to our choice of notation (\ref{eqn:v1v2}). 
\end{warn}

The element $v_1$ is invariant under the action of $\Sn_{\cC}$ on $(E_{\cC})_*V(0)$ so it is an element of $H^0(\Sn_{\cC}^1, (E_{\cC})_*V(0))$.  However, it does not lift to an invariant in $(E_{\cC})_*$.  Therefore, the image of $v_1$ in $H^1(\Sn_{\cC}^1, (E_{\cC})_*)$ under the connecting homomorphism $\delta$ for the exact sequence 
\begin{align}\label{eqn:connect}
0 \to (E_{\cC})_* \xra{2} (E_{\cC})_* \ra (E_{\cC})_*V(0) \to 0\end{align}
is non-zero. Let $\eta = \delta(v_1)$ in $H^1(\Sn_{\cC}^1, (E_{\cC})_*)$. We also call the image of $\eta$ in $H^1(\Sn_{\cC}^1, (E_{\cC})_*V(0))$ by the same name and note that $\eta$ is the image of
$v_1$ under the Bockstein $\beta : H^0(\Sn_{\cC}^1, (E_{\cC})_*V(0)) \to H^1(\Sn_{\cC}^1, (E_{\cC})_*V(0))$, i.e., $\eta = \beta(v_1)$.
\begin{lem}\label{v1linear}
The ADSS is a spectral sequence of modules over $\F_4[v_1, \eta]$. 
\end{lem}
\begin{proof}
The ADSS is a module over $H^*(\Sn_{\cC}^1,(E_{\cC})_*V(0))$ (see \fullref{S21-lin}), which is a module over $\F_4[v_1, \eta]$.
\end{proof}

\subsection{The $E_1$-term}\label{subsec:E1}
The input for the ADSS is the group cohomology of $G_{24}$ and $C_6$ with coefficients in $(E_{\cC})_*V(0)$. These cohomology groups are described in this section.

The computation of $H^*(G_{24},(E_{\cC})_*V(0))$ is related to that of $H^*(A,\Gamma)$ for $(A,\Gamma)$ the Hopf algebroid classifying Weierstrass curves and their strict isomorphisms. This result is originally due to Hopkins and Mahowald and can be found in Bauer in \cite[Section 7]{tbauer}. A self--contained presentation of the computation of $H^*(G_{24},(E_{\cC})_*V(0))$ is included in \ref{sec:appcoh}.

Some invariants of the curve $\cC_U$ play a central role. The following classes of $(E_{\cC})_*$ are invariant under the action of $\Aut(\cC)$. The reader may refer to either \ref{sec:appcoh} or to Silverman \cite[Section III.1]{silverman}).
\begin{align*}
\Delta &=  27v_2^3 (v_1^3-v_2)^3   & c_4  &=9(v_1^4+8v_1v_2) \\
c_6 &=  -27(8v_2^2+20 v_1^3v_2-v_1^6)  & j &= \frac{c_4^3}{\Delta}.
\end{align*}

\begin{warn}\label{warn:c4}
The reader must be careful not to confuse the $j$--invariant above and the element of $G_{24}$. Similarly, $c_4$ and $c_6$ differ from the elements of \fullref{thm:E2}. Our meaning should be clear from the context.

We abuse notation and call the corresponding invariants in $(E_{\cC})_*V(0)$ by the same name, so that $\Delta \equiv v_2^3 (v_1^3+v_2)^3 $, $j\equiv v_1^{12}\Delta^{-1}$, $c_4 \equiv v_1^4$ and $c_6 \equiv v_1^6$ in $(E_{\cC})_*V(0)$.
\end{warn}

The main result of \ref{sec:appcoh} is the following theorem.
\begin{thm}\label{TBG24}
There is an ismorphism
\[ H^*(G_{24}, (E_{\cC})_*V(0))\cong \F_4[\![j]\!][v_1,\Delta^{\pm 1}, k, \eta, \nu, x, y]/(\sim)\]
where $(\sim)$ is the ideal generated by the relations
\begin{align*}
v_1^{12}&=j\Delta & v_1 \nu&=0, & v_1^2x&=0, & v_1y&=0 , \\ 
\eta \nu&=0, & \nu x&=v_1\eta x, & \eta y&=v_1x^2, & xy&=0 ,    \\
   \eta^2x& =\nu^3,  & x^3&=\nu^2 y,   & y^2&=\nu^2\Delta, & \eta^4&=v_1^4 k 
\end{align*}
and degrees $(s,t)$, for $s$ the cohomological grading and $t$ the internal grading,
\begin{align*}
 |j|&=(0,0), & |v_1|&=(0,2), & |\Delta|&=(0, 24), &  |\eta|  &= (1,2), \\
  |\nu| &=(1,4) & |x|&=(1,8), & |y|&=(1,16) & |k|&=(4,0).
\end{align*}
\end{thm}

\begin{lem}\label{HC6}
The cohomology of $C_6$ with coefficients in $(E_{\cC})_*V(0)$ is given by
\begin{align*}
H^*(C_{6} ; (E_{\cC})_*V(0)) &= \mathbb{F}_4[\![u_1^3]\!][v_1,v_2^{\pm 1},h]/(v_1^3 = v_2u_1^3),
\end{align*}
where $|h| =(1,0)$, $|v_2| =(0,6)$, $|v_1| =(0,2)$ and $|u_1^3| =(0,0)$. Further, the action of $\eta$ is determined by
\begin{align*}
\eta \cdot 1 = v_1 h.
\end{align*}
\end{lem}
\begin{proof}
Recall that $C_2 = \{\pm 1\}$ denotes the center of $G_{24}$ and that $C_6 = C_{2}\times C_3$. Because $C_2$ acts trivially on $(E_{\cC})_*V(0)$ and $C_3$ has order coprime to $2$,
\begin{align*}
H^*(C_6, (E_{\cC})_*V(0)) &\cong H^*(C_2 ; (E_{\cC})_*V(0))^{C_3} \\
&=((E_{\cC})_*V(0))^{C_3}[h] \\
&= \mathbb{F}_4[\![u_1^3]\!][v_1,v_2^{\pm 1},h]/(v_1^3 = v_2u_1^3),
\end{align*}
where $h$ is in $(s,t)$ degree $(1,0)$. To prove that $\eta = v_1 h$, note that the action of $C_2$ on $(E_{\cC})_* = \W[\![u_1]\!][u^{\pm 1}]$ is given by
$\phi_{-1}(u) = -u$ and $\phi_{-1}(u_1) = u_1$. One computes that $\delta(v_1)= v_1h$ for $\delta$ the connecting homomorphism associated to (\ref{eqn:connect}). The claim follows from the fact that $\delta(v_1) = \eta$ (see \fullref{subsec:ss}).
\end{proof}

\begin{lem}\label{HG24'}
Let $\pi=1+2\omega$ in $\Sn_{\cC}$. Let $G_{24}' = \pi G_{24} \pi^{-1}$. Let $\phi_{\pi}: (E_{\cC})_* \rightarrow (E_{\cC})_*$ give the action of $\pi$ on $(E_{\cC})_*$. Then $\phi_{\pi}$ induces an $\F_4[v_1, \eta]$--linear isomorphism 
\[ H^*(G_{24}, (E_{\cC})_*V(0)) \cong H^*(G'_{24}, (E_{\cC})_*V(0)). \]
\end{lem}

\newpage
\hvFloat[%
floatPos=H,%
capVPos=c,%
capWidth=w,
rotAngle=90,
objectPos=c]{figure}{\includegraphics[height=0.72\textwidth]{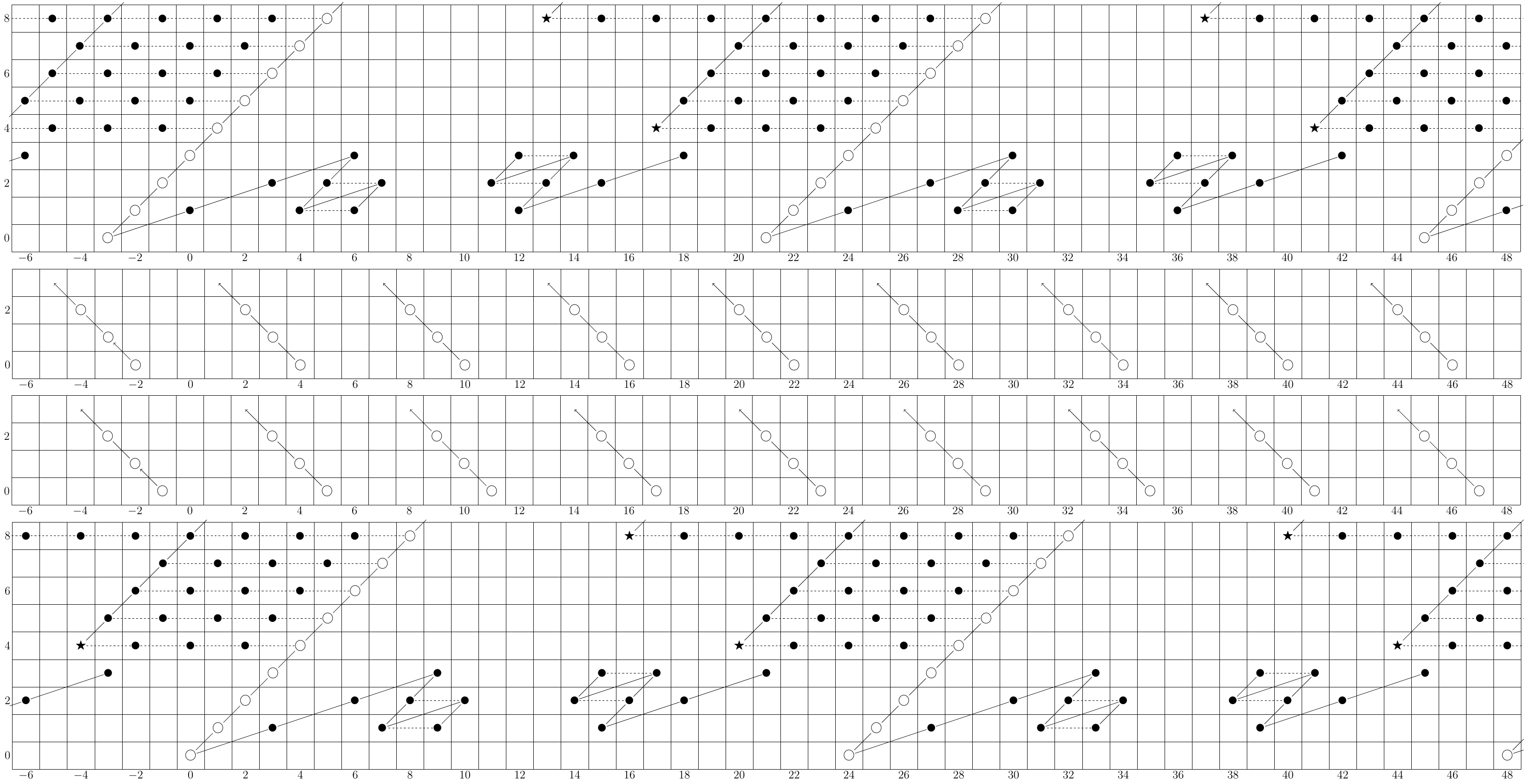}}%
{The $E_1$-term for the ADSS with coefficients $(E_{\cC})_*V(0)$. The rows represent $E_{1}^{p, *}$, the top row corresponding to $p=3$. The grading is given by $(t-q-p,q)$, where $t$ is the internal grading, so that $d_r : E_{r}^{p, q} \rightarrow E_{r}^{p+r, q-r+1}$ decreases the horizontal grading by $1$. A $\bullet$ denotes a copy of $\F_4$. Dashed horizontal lines denote multiplication by $v_1$, and a $\bigcirc$ denotes a copy of $\F_4[v_1]$. A $\bigstar$ is a copy of \fullref{fig:star}.}{fig:SSRS-E1}

\begin{figure}[H]
\includegraphics[width=\textwidth]{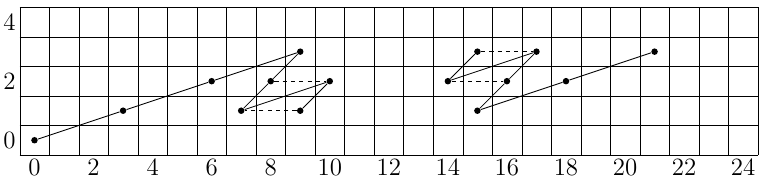}
\caption{The pattern $\bigstar$ in \fullref{fig:SSRS-E1}.}
\label{fig:star}
\end{figure}

\begin{proof}
Conjugation by any element of $\Sn_{\cC}$ induces an isomorphism on cohomology.
The linearity follows from the fact that $v_1$ is invariant under the action of $\pi$, and $\eta = \beta(v_1)$. 
\end{proof}
\begin{rem}\label{rem:deltaprime}
To avoid ambiguities, define $\Delta' := \phi_{\pi}(\Delta)$ and $j':= \phi_{\pi}(j)$. For an element $z$ of positive dimension in the cohomology of $G_{24}$, we will abuse notation and denote $\phi_{\pi}(z)$ by $z$ since there will be little room for confusion. Further, we let $k$ act on $H^*(G'_{24}, (E_{\cC})_*V(0))$ via $\phi_{\pi}(k)$ and treat the isomorphism of \fullref{HG24'} as one of $\F_4[v_1,\eta,k]$--modules.
\end{rem}

\subsection{Approximate $\Delta$--linearity}\label{subsec:Dlin}
Here, we make the key observation for the computations of \fullref{sec:compe2}. Namely, that the action of $(IS_{\cC}^1)^2$ on $(E_{\cC})_*V(0)$ is \emph{approximately $\Delta$--linear}. 
\begin{thm}\label{LIG-2}
Let $x$ be in $(E_{\cC})_*V(0)$. Let $g$ and $h$ be elements of $S_{\cC}$. Then,
\begin{enumerate}[(a)]
\item $\phi_h(\Delta)\equiv \Delta$ modulo $(2, u_1^6)$,
\item $(id-\phi_g)(id-\phi_h)(\Delta) \equiv 0$ modulo $(2, u_1^8)$,
\item $(id-\phi_g)(id-\phi_h)\left(x\Delta^{2^{k}(1+2t)}\right) \equiv (id-\phi_g)(id-\phi_h)(x) \Delta^{2^{k}(1+2t)}$ modulo $(2, u_1^{1+3\cdot 2^{k+1}})$.
\end{enumerate}
\end{thm}
\begin{proof}
For (a), note that by (\ref{eqn:t0gamma12}),
\[t_0(h) \equiv 1+a_1(h)^2u_1+a_1(h)u_1^2 \mod (2,u_1^3).\]
where $a_i(h)^4=a_i(h)$.
It follows that $t_0(h)^{16} \equiv 1 \mod (2, u_1^{16})$. Since $\Delta = u^{-12}(1+u_1^3)^3$, modulo $(2,u_1^6)$ we compute
\begin{align*}
(id-\phi_h)(\Delta) &\equiv u^{-12}(1-t_0(h)^{-12}) + u_1^3u^{-12}(1-t_0(h)^{-9}) \\
&\equiv u^{-12}(1-t_0(h)^{4}) + u_1^3u^{-12}(1-t_0(h)^{7}) \\
&\equiv u^{-12}(a_1(h)^8u_1^{4}) + u_1^3u^{-12}(a_1(h)^2u_1+a_1(h)u_1^2+a_1(h)^4u_1^2) \\
&\equiv 0 \mod (2, u_1^6).
\end{align*}
To prove (b) Since $id-\phi_g$ applied to the ideal $(u_1^6)$ is contained in the ideal $(u_1^8)$, the claim follows from (a).
Finally, it follows from (a) that for $h \in S_{\cC}$, there exists $y_h$ such that
$\phi_h(\Delta^{2^k(1+2t)}) =\Delta^{2^k(1+2t)}+v_1^{6\cdot 2^k}y_h$. 
Hence,
\[ (id-\phi_h)(\Delta^{2^k(1+2t)}) =\Delta^{2^k(1+2t)}(id-\phi_h)(x)+v_1^{6\cdot 2^k}y_h \phi_h(x). \]
Therefore,
\begin{align*}
(id-\phi_g)(id-\phi_h)(\Delta^{2^k(1+2t)}) &= (id-\phi_g)(\Delta^{2^k(1+2t)}(id-\phi_h)(x)+v_1^{6\cdot 2^k}y_h \phi_h(x)) \\
 &= \Delta^{2^k(1+2t)}(id-\phi_g)(id-\phi_h)(x)  \\
 & +  v_1^{6\cdot 2^k}y_g \phi_g((id-\phi_h)(x)) +v_1^{6\cdot 2^k}(id-\phi_g)(y_h \phi_h(x))).
\end{align*}
Since $h$ and $g$ are in $S_{\cC}$, 
\[(id-\phi_h)(x)\equiv (id-\phi_g)(y_h \phi_h(x))) \equiv 0 \mod (2,u_1).\]
This proves (c).
\end{proof}

\section{Computation of the $E_{\infty}$-Term}\label{sec:compe2}
Now we turn to the computation of the ADSS
 \begin{equation*}
 E_1^{p,q}  = \Ext^q_{\Z_2[\![\Sn_{\cC}^1]\!]}(\sC_p, (E_{\cC})_*V(0)) \Longrightarrow H^{p+q}(\Sn_{\cC}^1,(E_{\cC})_*V(0)),\end{equation*}
 with
$ E_1^{p,q} \cong H^q(F_p,  (E_{\cC})_*V(0))$, 
whose construction was described in \fullref{sec:E1}. Recall also that the spectral sequence comes from a resolution
\begin{align*}
0 \rightarrow \sC_3  \xra{\partial_3} \sC_2  \xra{\partial_2}\sC_1  \xra{\partial_1}\sC_0 \xra{\varepsilon} \Z_2 \rightarrow 0, 
\end{align*}
Further, $d_1: E_1^{p,q}  \rightarrow E_1^{p+1,q}$
is induced by $\Ext_{\Z_2[\![\Sn_{\cC}^1]\!]}^q(\partial_{p+1}, (E_{\cC})_*V(0))$
for $\partial_{p+1}$ as described in \fullref{thm:resnormS2}. We use these descriptions together with our partial knowledge of the action of $\Sn_{\cC}$ on $(E_{\cC})_*$ to compute the $d_1$--differentials. 

Recall that $E_1^{0,0} \cong (E_{\cC})_*^{G_{24}}$ and $E_1^{p,0} \cong (E_{\cC})_*^{C_6}$ for $p=1$ and $p=2$. Since there is an inclusion
\[ (E_{\cC})_*^{G_{24}} \ra (E_{\cC})_*^{C_6}, \]
there is an action of $ (E_{\cC})_*^{G_{24}}$ on $E_1^{p,0}$ for $0 \leq p \leq 2$. Therefore, it will make sense to talk about the image of $\Delta$ defined in \fullref{TBG24} in $E_1^{p,0}$. To avoid ambiguity, we will use the convention
\[\Delta^k[p] = \Delta^k \cdot 1 \in E_{1}^{p,0}\]
in the statement of the results. However, in the proofs, we will assume that the context is sufficient to determine which elements are meant. Similarly, for $v_2$ in $(E_{\cC})_*^{C_6}$, to distinguish between $E_1^{1,0}$ and $E_1^{2,0}$, we let
\[v_2^k[p] = v_2^{k} \cdot 1 \in E_{1}^{p,0}. \]

Although the differentials $d_1$ are not ring homomorphisms, they are induced by the action of elements in $\Z_2[\![\Sn_{\cC}]\!]$. Since $\Z_2[\![\Sn_{\cC}]\!]$ is generated by ring homomorphisms and $(E_{\cC})_*V(0)$ is an $\F_2$--vector, it follows that the differentials commute with the squaring operation. That is, for any $b$ in $E_1^{p,q}$,
\begin{align}\label{eqn:squaring}
d_1(b^2)=d_1(b)^2 \mod (2).\end{align}

\subsection{The differential $d_1 : E_1^{0,0} \ra E_1^{1,0}$}\label{subsec:d10}
The differential $d_1 : E_1^{0,0} \ra E_1^{1,0}$ is induced by the map
$\partial_1 : \sC_1 \ra \sC_0$, 
given by $\partial_1(\gamma e_1 ) = \gamma (e-\alpha)e_0$. Here, $e_i$ is the canonical generator of $\sC_i$. Therefore,
\[d_1 = id+\phi_{\alpha} :  E_1^{0,0} \rightarrow E_1^{1,0}.\]

Recall from {\color{red}\fullref{TBG24}} that the powers of 
$\Delta =  v_2(v_2+v_1^3)^3$
generate $H^0(G_{24}, (E_{\cC})_*V(0)) \cong (E_{\cC})_*^{G_{24}}$ as an $\F_4[v_1]$--module. So it is sufficient to compute $d_1$ on $\Delta^n[0]$ for $n\in \Z$. Therefore, we begin by recording a result on the action of $\alpha$ on the powers of $\Delta$.
\begin{prop}\label{prop:delta}
Let $n =2^k(2t+1)$, then
\begin{align*}
\phi_{\alpha}(\Delta^n) \equiv \Delta^{n} (1+v_1^{6\cdot 2^k}v_2^{-2^{k+1}}) \equiv \Delta^n+ v_1^{6\cdot 2^k}v_2^{2^{k+1}(4t+1)} \mod (2, u_1^{9\cdot 2^k}). \end{align*}
\end{prop}
\begin{proof}
By \fullref{alphav2}, $\phi_{\alpha}(v_2) \equiv v_2+v_1^3$ modulo $(2,u_1^9)$. Since $\phi_{\alpha}(v_1)\equiv v_1$ modulo $(2)$,
\begin{align*}
\phi_{\alpha}(\Delta) &\equiv \phi_{\alpha}(v_2)(\phi_{\alpha}(v_2)+v_1^3)^3 \\
&\equiv (v_2+v_1^3) v_2^3 \mod (2, u_1^9)\\
&\equiv \Delta (1+v_2^{-2}v_1^6) \mod (2, u_1^9).
\end{align*}
It suffices to prove the claim for $n=2t+1$ odd as the more general statement then follows from (\ref{eqn:squaring}). Using $\Delta^n \equiv v_2^{4n}$ modulo $(2,u_1)$, we have
\begin{align*}
\phi_{\alpha}(\Delta^n) &= (\Delta(1+v_2^{-2}v_1^6))^n\mod (2, u_1^9) \\
&\equiv  \Delta^n(1+n \cdot v_2^{-2}v_1^6) \mod (2, u_1^9) \\
&\equiv \Delta^n + v_2^{2(4t+1)}v_1^{6} \mod (2, u_1^{9}). \qedhere
\end{align*}
\end{proof}

\begin{cor}\label{partial0}
Let $n = 2^k(2t+1)$, then for $d_1 : E_1^{0,0} \to E_1^{1,0}$,
\begin{align*}
d_1(\Delta^n[0]) = v_1^{6\cdot 2^k} v_2^{2^{k+1}(4t+1)}[1] \mod (u_1^{9\cdot 2^k}).
\end{align*}
\end{cor}
\subsection{The differential $d_1 : E_1^{1,0} \ra E_1^{2,0}$}\label{subsec:d12}
The goal of this section is to prove the following result.
\begin{prop}\label{partial1}
Let $n=2^k (1+2t)$ where $t \in \Z$ and $k\geq 0$. There exist homogeneous elements $\db_n$, such that
$\db_n \equiv v_2^n[1]$ modulo $(u_1)$.
The elements $\db_n$ satisfy
\begin{align*} d_1( \Delta^n[0]) &= \left\{ \begin{array}{ll}
v_1^{6\cdot 2^{k}}\db_{2^{k+1}(1+4t)}  & \mbox{$ \ \ \ \ \ \ \ n = 2^k(1+2t)$}\\
    0 & \mbox{$\ \ \ \ \ \ \ n=0$},\end{array} \right.
    \end{align*}
    and
    \begin{align*}
d_1( \db_{n}) &= \left\{ \begin{array}{ll}
 v_1^{3\cdot 2^{k}}v_2^{2^{k+1}(1+2t) }[2] \mod (2, u_1^{3\cdot 2^{k}+3}) & \mbox{$n =2^k(3+4t)$}\\
     v_1^{3\cdot 2^{k+1}}v_2^{m-2^{k+1 }}[2]  \mod (2, u_1^{3\cdot 2^{k+1}+3})  & \mbox{$n = 1+2^{k+2}+2^{k+3}t$}\\
     0  & \mbox{$n = 0, 1$ and $2^{k+1}(1+4t)$}.\end{array} \right. \end{align*}
\end{prop}
\noindent
We will break up the proof into a series of propositions.

The differential $d_1: E_1^{1,0} \rightarrow E_1^{2,0}$ is induced by the map
$\partial_2 : \sC_2 \ra \sC_1$. 
Recall from \fullref{thm:resnormS2} that 
\[\partial_2(\gamma e_2) = \gamma(e+\alpha + \sE) e_1\] 
where $\sE \in (2, (I\Sn_{\cC}^1)^2)$. Therefore, modulo $(2)$, $\sE = \sum a_{g,h} (e-g)(e-h)$ is in $(IS_{\cC}^1)^2$. It is to be thought of as the \emph{error}. Let 
\begin{align}\label{BB}
\phi_{\sE} = \sum a_{g,h} (id-\phi_g)(id-\phi_h).\end{align}

We first construct the $d_1$--cycle $\db_1$. The idea for its construction comes from Mahowald and Rezk \cite[Corollary 6.2]{MR2508904}. We need the following result.
\begin{lem}\label{lem:c4zero}
Let $c_4$ in $(E_{\cC})_8^{G_{24}}$ be given by 
\[c_4 = 9(v_1^4+8v_1v_2)=9u^{-4}u_1(u_1^3+8)\]
as in \fullref{subsec:E1}. For any $\gamma$ in $\mathbb{G}_{\cC}$, 
\[\phi_{\gamma}(c_4) \equiv c_4  \mod (16).\]
Further, if an element $\gamma$ in $\mathbb{S}_{\cC}$ has the form $\gamma \equiv 1+a_2(\gamma)T^2$ modulo $T^3$ for $a_2(\gamma)$ as in \fullref{subsec:act}, then 
\[c_4 - \phi_{\gamma}(c_4) \equiv 16(a_2(\gamma)+a_2(\gamma)^2)u_1u^{-4}  \mod (32, 16 u_1^2).\]
\end{lem}
\begin{proof}
The first step is to show that
$\phi_{\gamma}(c_4)  \equiv c_4$ modulo $16$. Since the Galois group acts trivially on $c_4$, it suffices to prove the claim for $\gamma$ in $\mathbb{S}_{\cC}$.
Let $t_0 = t_0(\gamma)$ and $t_1 = t_1(\gamma)$ as defined in (\ref{eqn:actint}).
A direct computation using (\ref{eqn:actint})
implies that
\begin{align*}
c_4 - \phi_{\gamma}(c_4)  & \equiv 8u^{-4} \left(u_1+\frac{3 u_1}{t_0^3}+\frac{t_1^2 u_1^2}{t_0^4}+\frac{t_1 u_1^3}{t_0^2}+\frac{2 t_1}{t_0^5}+\frac{2 t_1^4}{t_0^8} \right) \mod (32) \\
&\equiv  8u^{-4} u_1t_0^{-4} \left( t_0^4 +t_0 + u_1t_1^2+ u_1^2t_1t_0^2  \right) \mod (16)
\end{align*}
It follows from \fullref{prop:t0t1eqn} of \fullref{sec:appC} that
\[t_0 \equiv t_0^4  + u_1t_1^2+ u_1^2t_1t_0^2   \mod (2).\]
This proves that $c_4 - \phi_{\gamma}(c_4) \equiv 0$ modulo $(16)$.

Let $a_i = a_i(\gamma)$ as in \fullref{subsec:act}. By \fullref{t0t1again} and \fullref{t0gammafinally} applied to $\gamma$, we have
\begin{align*}
t_0 &\equiv 1+2a_2+2 a_3^2 u_1 \mod (4, u_1^2), & t_1 &\equiv a_2^2 u_1 \mod (2, u_1^2).
\end{align*}
Therefore, $t_0^4 \equiv 1$ modulo $(4,u_1^2)$ and $t_1^4 \equiv 0$ modulo $(4,u_1^2)$ so that
\begin{align*}
c_4 - \phi_{\gamma}(c_4) &\equiv   8u^{-4} \left(u_1+3 u_1 t_0+2 t_1t_0^3 \right) \mod (32,u_1^2)  \\
&\equiv 8u^{-4}  \left(  u_1 + 3u_1( 1+2a_2+2 a_3^2 u_1) +2a_2^2u_1 \right) \mod (32,u_1^2)  \\
&\equiv 16(a_2+a_2^2)u^{-4}u_1 \mod (32,u_1^2) 
\end{align*}
This implies that $c_4 - \phi_{\gamma}(c_4) =16((a_2+a_2^2)u^{-4}u_1+ \ldots)$.
\end{proof}

Consider the spectral sequence
 \begin{equation}\label{DRSS-ag-S}
\widetilde{E}_1^{p,q} = \Ext^q_{\Z_2[\![\Sn_{\cC}^1]\!]}(\sC_p, (E_{\cC})_*) \Longrightarrow H^{p+q}(\Sn_{\cC}^1,(E_{\cC})_*).\end{equation}
Let 
\begin{equation}\label{fEtildeE}
f : \widetilde{E}_1^{p,q}  \ra {E}_1^{p,q}\end{equation}
be the map of spectral sequences induced by the reduction modulo $(2)$ on the coefficients. Let $d_1: \widetilde{E}_1^{0,0} \to \widetilde{E}_1^{1,0}$ denotes the differential in the spectral sequence $\widetilde{E}_r^{p,q}$. Since $d_1(x) = x- \phi_{\alpha}(x)$, and $\alpha \equiv 1+\omega T^2 $ modulo $T^3$, it follows from \fullref{lem:c4zero} that
\[d_1(c_4) = 16(v_1v_2 +\ldots).\] 
\begin{defn}\label{defn:B1}
Let $B_1 \in \widetilde{E}_1^{1,0}$ be defined by $B_1 = \frac{d_1(c_4)}{16}$.
Since $ \widetilde{E}_1^{1,0}$ is torsion free, this specifies $B_1$ uniquely and
\[B_1 \equiv v_1v_2 \mod (2, u_1^2).\]
\end{defn}
\begin{prop}\label{prop:d1two1} There is an element $\db_1 \in E_1^{1,0}$ specified by the identity $f(B_1)=v_1\db_1$
 such that
$\db_1 \equiv v_2[1]$ modulo $(u_1^3)$ and $d_1(\db_1) =0$. 
\end{prop}
\begin{proof}
Let $B_1$ be as in \fullref{defn:B1}.
Then, $f(B_1)$ is divisible by $v_1$. Therefore, we can define an element $\db_1 \in E_1^{1,0}$ by
$\db_1:= v_1^{-1}f(B_1)$. Since $E_1^{1,0}$ is $v_1$--torsion free, this specifies $\db_1$ uniquely.
Further, it implies that $\db_1 \equiv v_2$ modulo $(2, u_1)$. Since $\db_1$ is an element of 
\[(E_{\cC})_6V(0)^{C_{6}} = \F_4[\![u_1^3]\!]\{v_2\}.\]
This forces the congruence $\db_1 \equiv v_2$ modulo $(2, u_1^3)$.

Finally, in the spectral sequence $\widetilde{E}_r^{p,q}$, we have
$d_1^2(c_4) = 16d_1(B_1)$.  
Since there is no torsion in $\widetilde{E}_1^{p,0}$ and $d_1^2=0$, $d_1(B_1)=0$ so that $d_1(v_1\db_1)=0$ in $E_1^{2,0}$. Since there is no $v_1$--torsion in $E_1^{p,0}$, this implies that $d_1(\db_1)=0$.
\end{proof}

\begin{prop}\label{prop:d1two2}
Let $n= 2^k( 3 + 4t)$. Define $\db_n = v_2^n$. Then
\begin{align*}
d_1(\db_n)&\equiv v_1^{3\cdot 2^k} v_2^{2^{k+1}(1+2t)} \mod (2, u_1^{6 \cdot 2^{k}}).
\end{align*}
\end{prop}
\begin{proof}
By (\ref{eqn:squaring}) is sufficient to prove the claim when $k=0$. First, note that if $h \in {S}_{\cC}$, by (\ref{eqn:t0gamma12}), 
\[t_0(h) \equiv 1+a_1(h)^2u_1+a_1(h)u_1^2 \mod (2,u_1^3)\]
where $a_1(h)^4 = a_1(h)$.
Therefore, $t_0(h)^4 \equiv 1$ modulo $(2, u_1^4)$ and
\begin{align*}
(id-\phi_h)(v_2^{3+4t})&= v_2^{3+4t}+t_0(h)^{-3(3+4t)}v_2^{3+4t}  \\
&\equiv v_2^{3+4t}+  t_0(h)^{3}v_2^{3+4t}  \mod (2, u_1^3) \\
&\equiv v_2^{3+4t} + v_2^{3+4t} (1+a_1(h)^2u_1+a_1(h)u_1^2+a_1(h)^4u_1^2)\mod (2, u_1^3) \\
&\equiv  v_2^{3+4t}a_1(h)^2u_1  \mod (2, u_1^3).
\end{align*}
For any $g \in S_{\cC}$, the image of $(2,u_1^3)$ under $id -\phi_g$ is in $(2,u_1^4)$. Further, since $t_0(g)^4 \equiv 1 \mod (2, u_1^4)$
\begin{align*}
(id-\phi_g)(v_2^{3+4t}u_1) &\equiv (1-t_0(g)^{-8-12t})v_2^{3+4t}u_1 \\
&\equiv 0 \mod (2, u_1^5).
\end{align*}
Hence, $(id-\phi_g)(id-\phi_h)(v_2^{3+4t}) \equiv 0$ modulo $(2,u_1^4)$. For $\phi_{\sE}$ be as in (\ref{BB}), it follows that
\begin{align*}
\phi_{\sE}(v_2^{3+4t}) \equiv 0 \mod (2,u_1^4).
\end{align*} 
Hence, since $\phi_{\alpha}(v_2) \equiv v_2 +v_1^3 \mod (2, u_1^9)$ (see \fullref{alphav2}),
\begin{align*}
d_1(v_2^{3+4t})  &=   v_2^{3+4t}+ \phi_{\alpha}(v_2^{3+4t}) +  \phi_{\sE}(v_2^{3+4t}) \\
&\equiv v_2^{3+4t} +(v_2+v_1^3)^{3+4t}   \\
&\equiv v_1^3 v_2^{2+4t} \mod (2, u_1^4).
\end{align*}
Since $d_1(\db_n)$ is $C_6$--invariant, the congruence can be improved to
$d_1(\db_n) \equiv v_1^3 v_2^{2(1+2t)}$ modulo $(2, u_1^{6})$.
\end{proof}

\begin{prop}\label{prop:d1two3}
Let $n=1+2^{k+2}(1+2t)$. Define $\db_n = \db_1\Delta^{2^{k}(1+2t)}$. Then
\begin{align*}
d_1(\db_{n}) 
&\equiv v_1^{3\cdot 2^{k+1}}v_2^{1+2^{k+1}(1+4t)} \mod (u_1^{3\cdot 2^{k+1}+3}).
\end{align*}
\end{prop}
\begin{proof}
By \fullref{LIG-2},
\[\phi_{\sE}(\db_n) \equiv \phi_{\sE}(\db_1) \Delta^{2^{k}(1+2t)}  \mod (2,u_1^{1+3\cdot 2^{k+1}}).\]
Therefore,
\begin{align*}
d_1(\db_{n}) &\equiv \db_n+\phi_{\alpha}(\db_n)+ \phi_{\sE}(\db_n) \\
&\equiv  \db_1\Delta^{2^{k}(1+2t)}+ \phi_{\alpha}(\db_1)\Delta^{2^{k}(1+2t)}(1+v_2^{-2}v_1^6)^{2^{k}(1+2t)}+ \phi_{\sE}(\db_1)\Delta^{2^{k}(1+2t)}  \\
&\equiv  \db_1\Delta^{2^{k}(1+2t)}+ \phi_{\alpha}(\db_1)\Delta^{2^{k}(1+2t)}(1+v_2^{-2^{k+1}}v_1^{3\cdot2^{k+1}})+ \phi_{\sE}(\db_1)\Delta^{2^{k}(1+2t)} \\
&\equiv \left(\db_1+\phi_{\alpha}(\db_1)+\phi_{\sE}(\db_1)\right)\Delta^{2^{k}(1+2t)}+ \phi_{\alpha}(\db_1)v_2^{-2^{k+1}}v_1^{3\cdot2^{k+1}}\Delta^{2^{k}(1+2t)} \\
&\equiv d_1(\db_1)\Delta^{2^{k}(1+2t)}+ \phi_{\alpha}(\db_1)v_2^{-2^{k+1}}v_1^{3\cdot 2^{k+1}}\Delta^{2^{k}(1+2t)} \mod (2, u_1^{1+3\cdot 2^{k+1}}).
\end{align*}
But $d_1(\db_1)=0$ and $\phi_{\alpha}(\db_1) \equiv v_2$ modulo $(2, u_1^3)$.
Furthermore, $\Delta^{2^{k}(1+2t)} \equiv v_2^{2^{k+2}+2^{k+3}t}$, so that
\begin{align*}
d_1(\db_{n}) &\equiv v_1^{3\cdot 2^{k+1}}v_2^{1-2^{k+1}+2^{k+2}+2^{k+3}t} \\
&\equiv v_1^{3\cdot 2^{k+1}}v_2^{1+2^{k+1}+2^{k+3}t} \mod (2, u_1^{3\cdot 2^{k+1}+1}).
\end{align*}
Since $d_1(\db_n)$ is $C_6$ invariant, the congruence holds modulo $(2,u_1^{3\cdot 2^{k+1}+3})$.
\end{proof}

\begin{proof}[Proof of \fullref{partial1}.]
Let $t \in \Z$ and $k\geq 0$
\[ \db_{n}:= \left\{ \begin{array}{ll}
 \db_1^n   & \mbox{$n = 0,1$};\\
 v_2^n   & \mbox{$n = 2^k( 3 + 4 t)$};\\
    \db_1\Delta^{2^{k}+2^{k+1}t}  & \mbox{$n=1+2^{k+2}(1+2t )$};\\
    v_1^{-6\cdot 2^k} d_1\left(\Delta^{2^k(2t+1)}\right)& \mbox{$n = 2^{k+1} (4t+1)$}.\end{array} \right. \] 
The element $\db_{n}$ is in degree $6n$ and
$\db_{n} \equiv v_2^n $ modulo $(2, u_1^3)$.
That $d_1(\db_0)=0$ follows from the fact that it is invariant under the action of $\Sn_{\cC}$. It  is the content of \fullref{prop:d1two1} that $d_1(\db_1)=0$. Let $n = 2^{k+1} (1+4t)$. Since $d_1$ is $v_1$--linear and there is no $v_1$--torsion in $E_1^{2,0}$, it follows from
\begin{align*}
d_1( v_1^{6 \cdot 2^k} \db_n) &= d_1^2\left( \Delta^{2^k(2t+1)}  \right) =0,
\end{align*}
that $d_1(\db_{n})=0$. The remaining claims follow from \fullref{prop:d1two2} and \fullref{prop:d1two3}.
\end{proof}

\subsection{The differential $d_1 : E_1^{2,0} \ra E_1^{3,0}$}\label{subsec:d13}

Recall that 
\[ E_1^{3,0}  \cong H^q(G_{24}', (E_{\cC})_*V(0)) = \F_4[\![j']\!][v_1, \Delta']/(j' = v_1^{12}\Delta'^{-1})\]
where $\Delta' = \phi_{\pi}(\Delta)$. We let $\Delta'[3] = \Delta' \cdot 1 \in E_1^{3,0}$. The next goal is to prove:
\begin{prop}\label{Tp2}
Let $n=2^k (1+2t)$ where $t \in \Z$ and $k\geq 0$. There exist homogeneous elements $\ob_{n}$ such that 
\begin{align}\label{cv}
\ob_{n} \equiv v_2^n[2]  \mod (u_1)
\end{align}
and
\begin{align*}
d_1( \db_{n}) &= \left\{ \begin{array}{ll}
v_1^{3\cdot 2^{k}}\ob_{2^{k+1}(1+2t)}  & \mbox{$ \ \ \ \ \ n = 2^k(3+4t)$}\\
     v_1^{3\cdot 2^{k+1}}\ob_{1+2^{k+1}(1+4t)} & \mbox{$\ \ \ \ \ n = 1+2^{k+2}(1+2t)$}\\
    0 & \mbox{\ \ \ \ \ otherwise} .\end{array} \right.  \end{align*}
Further, 
\[d_1(\ob_n) = v_1^{3(1+2^{k+1})}\Delta'^{2^k(1+2t)}[3] \mod (u_1^{3(1+2^{k+1})+12})\] 
if $n=1+2^{k+1}(3+4t)$ and is zero otherwise.
\end{prop}

\begin{proof}[Proof of \fullref{Tp2}.]
For $n=2^k(3+4t)$ and $n =1+ 2^{k+2}(1+2t)$, define $\ob_n$ by the identities
\[ d_1( \db_{n}) = \left\{ \begin{array}{ll}
v_1^{3\cdot 2^{k}}\ob_{2^{k+1}(1+2t)}  & \mbox{$n = 2^k(3+4t)$};\\
     v_1^{3\cdot 2^{k+1}}\ob_{1+ 2^{k+1}(1+4t)} & \mbox{$n =1+ 2^{k+2}(1+2t)$}.\end{array} \right. \] 
The classes $\ob_{2^{k+1}(1+2t)}$ and $\ob_{1+ 2^{k+1}(1+4t)}$ are well--defined since $E_1^{2,0}$ is torsion free. Further, the $\ob_n$ satisfies equation (\ref{cv}) and
$d_1(\ob_n) = 0$.

Let $m =1+ 2^{k+1}(1+4t)$. For $n= 1+2^{k+1}(3+4t)$, define
\[\ob_n = \ob_{m}(\Delta')^{2^{k}}.\]
Because $(\Delta')^{2^{k}} \equiv v_2^{2^{k+2}}$ modulo $(2, u_1)$, the elements $\ob_n$ satisfy (\ref{cv}). We will prove that
\begin{align}\label{eqn:step}
d_1(\ob_n) \equiv v_1^{3(2^{k+1}+1)}(\Delta')^{2^k(2t+1)} \mod (u_1^{3(2^{k+1}+1)+1}).
\end{align}
Because $d_1(\ob_n)$ is $G_{24}'$ invariant, if (\ref{eqn:step}) holds, then the congruence also holds modulo $(u_1^{3(1+2^{k+1})+12})$. This will finish the proof of the theorem.

By \fullref{thm:resnormS2}, the map $d_1:   E_1^{2,0} \rightarrow  E_1^{3,0}$ is given by
\begin{align*}
\phi_{\pi}(id+\phi_i+\phi_j+\phi_k)(id+\phi_{\alpha}^{-1}) \phi_{\pi}^{-1}.
\end{align*}
Since $ \phi_{\pi}^{-1}(\Delta') =\Delta$,
\[d_1(\ob_n)= \phi_{\pi}(id+\phi_i+\phi_j+\phi_k)(id+\phi_{\alpha}^{-1})(\phi_{\pi}^{-1}(\ob_m)\Delta^{2^k}).\]
By \fullref{alphav2}, we can $\phi_{\alpha} = \phi_{\alpha^{-1}}$ modulo $u_1^9$. By \fullref{prop:delta}, this implies that
\begin{align*}
\phi_{\alpha^{-1}}(\Delta^{2^k}) =\Delta^{2^k}(1+v_1^{6\cdot 2^k}v_2^{-2^{k+1}}) \mod (2,u_1^{9\cdot 2^k}).
\end{align*}
Hence, modulo $ (2, u_1^{9\cdot 2^k})$,
\begin{align*}
(id+\phi_{\alpha^{-1}})(\phi_{\pi}^{-1}(\ob_m)\Delta^{2k})&\equiv \phi_{\pi}^{-1}(\ob_m)\Delta^{2^k}+ \phi_{\alpha^{-1}}(\phi_{\pi}^{-1}(\ob_{m})) \Delta^{2^k}(1+ v_1^{6\cdot 2^k}v_2^{-2^{k+1}})  \\
&\equiv (id+\phi_{\alpha^{-1}})(\phi_{\pi}^{-1}(\ob_{m})) \cdot \Delta^{2^{k}} \\
& \ \ \ \ \ \ \ +  \phi_{\alpha^{-1}}(\phi_{\pi}^{-1}(\ob_{m}))(v_1^{6 \cdot 2^{k}}v_2^{-2^{k+1}}) \Delta^{2^{k}} .
\end{align*}
We treat both terms separately. First, note that $i$, $j$ and $k$ fix $\Delta$, so that
\begin{align*}
\phi_{\pi}\left((id + \phi_i+\phi_j+\phi_k)\left( (id+\phi_{\alpha^{-1}})(\phi_{\pi}^{-1}(\ob_m) ) \cdot \Delta^{2^{k}}\right)\right) &= d_1(\ob_m) \cdot (\Delta')^{2^{k}} =0.
\end{align*}
Next, note that $\phi_{\alpha^{-1}}\phi_{\pi^{-1}}(\ob_{m})  = \phi_{(\pi\alpha)^{-1}}(\ob_m)$. Since $\pi\alpha \in F_{2/2}\Sn_{\cC}$, it follows from \fullref{alphav2} that
 \[\phi_{\alpha^{-1}}\phi_{\pi^{-1}}(\ob_{m}) \equiv \ob_{m} \mod (2, u_1^3).\]
Since $\ob_m \equiv v_2^{1+2^{k+1}(1+4t)} \mod (2, u_1^3)$, this implies that
\[  \phi_{\alpha^{-1}}(\phi_{\pi}^{-1}(\ob_{m}))(v_1^{6 \cdot 2^{k}}v_2^{-2^{k+1}})  \equiv  v_1^{6 \cdot 2^{k}} v_2^{1+2^{k+3}t}  \mod  (2, u_1^{3(2^{k+1}+1)})\]
Further,  $(2, u_1^{9 \cdot 2^k}) \subseteq  (2, u_1^{3(2^{k+1}+1)})$ and, since $(id + \phi_i+\phi_j+\phi_k)$ is in $IS_{\cC}$ it maps $ (2, u_1^{3(2^{k+1}+1)})$ to the ideal $(2, u_1^{3(2^{k+1}+1)+1})$, we can ignore the error terms. Hence,
\begin{align*}
d_1(\ob_n) &= \phi_{\pi}((id + \phi_i+\phi_j+\phi_k)(v_2^{1+2^{k+3}t} )) \cdot  v_1^{6 \cdot 2^{k}}  (\Delta')^{2^{k}} \mod (2,v_1^{3(2^{k+1}+1)+1}).
\end{align*}
From \fullref{subsec:AutC}, 
\begin{align*}
t_0(i)^{-1} &= 1+u_1 & t_0(j)^{-1} &= 1+ \zeta u_1 & t_0(k)^{-1} &= 1+ \zeta^2 u_1.
\end{align*}
We have $t_0(i)^{-8} \equiv t_0(j)^{-8}\equiv t_0(k)^{-8}  \equiv 1 $ modulo $(2, u_1^8)$. Modulo $(2,u_1^8)$,
\begin{align*}
t_0(i)^{-3(1+2^{k+3}t)} &\equiv (1+u_1)^3 \equiv 1+u_1 +u_1^2+u_1^3  \\
t_0(j)^{-3(1+2^{k+3}t)} &\equiv (1+\zeta u_1)^3 \equiv 1+\zeta u_1 +\zeta^3 u_1^2+u_1^3 \\
t_0(k)^{-3(1+2^{k+3}t)} &\equiv (1+\zeta^2 u_1)^3 \equiv 1+\zeta^2 u_1 +\zeta u_1^2+u_1^3.
\end{align*}
Since $\phi_{\gamma}(v_2^n) = t_{0}(\gamma)^nv_2^n$,
\begin{align*} (id + \phi_i+\phi_j+\phi_k)(v_2^{1+2^{k+3}t})&\equiv v_1^3 v_2^{2^{k+3}t} \mod (2, u_1^8). \end{align*}
Hence,
\begin{align*}
d_1(\ob_n) &\equiv  v_1^{3(1+ 2^{k+1})}  \phi_{\pi}(v_2^{2^{k+3}t}) (\Delta')^{2^{k}} \mod (2,v_1^{3(2^{k+1}+1)+1}) \\
&\equiv  v_1^{3(1+ 2^{k+1})}  (\Delta')^{2^{k}(1+2t)} \mod (2,v_1^{3(2^{k+1}+1)+1}).
\end{align*}
The only element $\ob_n$ which has not been constructed is $\ob_1$. Its existence follows from \fullref{Lp2-2} below.
\end{proof}

\begin{lem}\label{Lp2-2}
There exists a sequence of elements $\{\ob_{1,n}\}$ such that
\begin{enumerate}[(1)]
\item \label{i1}$\ob_{1,n} \equiv v_2$ modulo $(u_1^6)$,
\item \label{i2}$d_1(\ob_{1,n}) \equiv 0$ modulo $(u_1^{3(1+4n)})$,
\item \label{i3}$\ob_{1,{n+1}} - \ob_{1,n} \equiv  0 $ modulo $(u_1^{6n})$.
\end{enumerate}
If $(E_{\cC})_6V(0)$ is given the topology induced by  the maximal ideal $\mathfrak{m} = (u_1)$,
then the limit
\begin{align*}
\ob_1 := \lim_{n\rightarrow \infty} \ob_{1,n}
\end{align*}
exists. The element $\ob_1$ satisfies equation (\ref{cv}) and
$d_1(\ob_1) = 0$.
\end{lem}

\begin{proof}
The construction of $\{\ob_{1,n}\}$ is by induction on $n$. First, define 
$\ob_{1,1} := v_2$
and note that
\begin{align*}
\ob_{1,1} + \phi_{\alpha^{-1}}(\ob_{1,1}) &\equiv v_1^3 +u_1^6\epsilon.
\end{align*}
The $\F_4$-vector space with basis
\begin{align*}
\{v_1^3, v_1^{3\cdot 5}\Delta'^{-1}, v_1^{3\cdot 9}\Delta'^{-2}, \ldots, v_1^{3(1+4s)}\Delta'^{-s}, \ldots \}
\end{align*} 
is dense in $\left((E_{\cC})_6V(0)\right)^{G_{24}'}$. Hence,
$d_1(\ob_{1,1}) \equiv 0$ modulo $(u_1^6)$. 

Suppose that $\ob_{1,n}$ has been defined. If $d_1(\ob_{1,n})=0$, then let $\ob_{1,{N}} : = \ob_{1,n}$ for all $N \geq n$. Otherwise,
\begin{align}\label{d1cn}
d_1(\ob_{1,n}) &= v_1^{3+12s_n}\Delta'^{-s_n} + \ldots 
\end{align}
for $s_n \geq n$. Let $s_n= 2^{k_n}(1+2t_n)$ and let $m_n =3\cdot 2^{k_n+1}(1+4t_n)$. Then 
$m_n \geq 6n$.
For
\[r_n = 1+2^{k_n+1}+2^{k_n+2}+2^{k_n+3}(-t_n-1),\]
(\ref{d1cn}) together with the fact that
\[d_1(\ob_{r_n}) = v_1^{3(1+2^{k_n+1})}\Delta'^{2^k_n(1+2(-t_n-1)} +\ldots,  \] 
implies that $d_1(\ob_{1,n}) =v_1^{m_n} d_1(\ob_{r_n}) + \ldots$
Define $\ob_{1,{n+1}} 
:=\ob_{1,n} +v_1^{m_n} \ob_{r_n}$. 
Then $\ob_{1,{n+1}}$ satisfies properties (\ref{i1}), (\ref{i2}) and (\ref{i3}).

Now consider the sequence $\{\ob_{1,n}\}$. Since $m_{n+k} \geq 6n$ for $k\geq 0$,
\begin{align*}
\ob_{1,{n+k}} - \ob_{1,{n}} &= v_1^{m_{n+1}}\ob_{r_{n+1}}+\ldots +  v_1^{m_{n+k}}\ob_{r_{n+k}} \in (u_1)^{6n},
\end{align*}
so the sequence $\{\ob_{1,n}\}$ is Cauchy. 
Since $((E_{\cC})_{6}V(0))^{C_6}$ is complete with respect to $\mathfrak{m}$, the limit exists and $\ob_1$ is well-defined.
The map $d_1$ is continuous, so that,
\begin{align*}
d_1(\ob_1)= \lim_{n\rightarrow \infty} d_1(\ob_{1,n}).
\end{align*}
But $d_1(\ob_{1,n}) \in \mathfrak{m}^{3(1+4N)}$ for all $n\geq N$, which implies that
\[d_1(\ob_1) \in \displaystyle \bigcap_{n=0}^{\infty} \mathfrak{m}^n =0.
\qedhere
\]
\end{proof}

\begin{rem}\label{rem:congruencesv}
Define
\begin{align*}
\dd_n &:= \left\{ \begin{array}{ll}
 \Delta^n[0] & \mbox{$\ \ \ \ \ \ \ \ \ \ \ \ \ \ \ \ \ \ \ \ \ \ \ \ \ \ \ \ \ \ \ \ \ \ \ \ \  n=2^k(1+2t)$}\\
1\cdot [0] & \mbox{$\ \ \ \ \ \ \ \ \ \ \ \ \ \ \ \ \ \ \ \ \ \ \ \ \ \ \ \ \ \ \ \ \ \  \ \ \   n=0$},\end{array} \right. \\
\od_n &:= \left\{ \begin{array}{ll}
 v_1^{-3(1+2^{k+1})}d_1(\ob_{1+2^{k+1}+2^{k+2} +2^{k+3}t})  & \mbox{$n=2^k(1+2t)$}\\
1 \cdot [3] & \mbox{$n=0$}.\end{array} \right. 
\end{align*}
Combining the results of this section to proves the first part of \fullref{thm:E2}. 

An analysis of the definition of the elements shows that the congruences stated in \fullref{thm:E2} can be improved as follows:
\begin{align*}
\dd_n &= \Delta^n[0]  \\
\db_n &= \begin{cases}  
v_2^{n}[1]  &\text{$n=0$ or $2^k( 3 + 4t)$}  \\
v_2^{n}[1] \mod (u_1^3) &\text{$n=1$ or $1+2^{k+2}(1+2t)$} \\
v_2^{n}[1] \mod (2, u_1^{3\cdot 2^k}) & n= 2^{k+1}(4t+1) 
\end{cases} \\
\ob_n & = \begin{cases}
v_2^{n}[2]  &n=0 \\
v_2^{n}[2] \mod (u_1^{3 \cdot 2^{k}}) & \text{$n=1$ or $n=2^{k+1}(1+2t)$} \\
v_2^{n}[2] \mod (u_1^{3}) &\text{$n=1+2^{k+1}(1+4t)$ or $n=1+2^{k+1}(3+4t)$} \\
\end{cases} \\
\od_n &=  \begin{cases} 1 \cdot [3]  &n=0 \\
 \Delta'^n[3] \mod (u_1^{12}) & n \neq 0
 \end{cases}
\end{align*}
\end{rem}

\subsection{The differentials $d_1 : E_1^{p,q} \ra E_1^{p+1,q}$ for $q>0$}\label{subsec:d1rest}
Although $V(0)$ is not a ring spectrum,
$(E_{\cC})_*V(0) \cong (E_{\cC})_*/2$, 
and a canonical generator is given by the image of the unit in $(E_{\cC})_0$ in the long exact sequence
\[\ldots  \ra (E_{\cC})_* \xra{2} (E_{\cC})_* \ra (E_{\cC})_*V(0)  \ra \ldots \]
Thus, $(E_{\cC})_*V(0)$ inherits a ring structure from $(E_{\cC})_*$. \fullref{S21-lin} implies that the ADSS for $(E_{\cC})_*V(0)$ is a module over $H^*(\Sn_{\cC}, (E_{\cC})_*V(0))$. The canonical inclusion of $\F_4$ into
$(E_{\cC})_*V(0)$
induces a map
\[H^*(\Sn_{\cC}^1, \F_4) \rightarrow H^*(\Sn_{\cC}^1, (E_{\cC})_*V(0))\]
and the ADSS for $(E_{\cC})_*V(0)$ is also a module over $H^*(\Sn_{\cC}^1, \F_4)$. 

Let 
 \begin{equation}\label{DRSS-ag-Tr}
 F_1^{p,q} = \Ext^q_{\Z_2[\![\Sn_{\cC}^1]\!]}(\sC_p, \F_4) \Longrightarrow H^{p+q}(\Sn_{\cC}^1,\F_4).\end{equation}
Let $k \in F_{1}^{0,4}$ be the periodicity generator for the cohomology of $G_{24}$,
(see \fullref{constant}). The extension
\[1 \ra K^1 \ra  \Sn_{\cC}^1 \ra G_{24} \ra 1 \]
is split. Therefore, the map
\[H^*(\Sn_{\cC}^1,\F_4) \ra H^*(G_{24},\F_4) \]
induced by the inclusion of $G_{24}$ in $\Sn_{\cC}^1$ is split surjective. This implies that the image of $k$ is a permanent cycle in $F_{1}^{0,4}$. Therefore, it represents a class
 \[k \in H^4(\Sn_{\cC}; \F_4),\]
and the differentials in the ADSS commute with the action of $k$. To make sense of this, we compute the action of $k$ on $E_1^{p,q}$. 

First, $k$ acts by multiplication by the element of the same name in $E_1^{0,q}$ and $E_1^{3,q}$. Further, the map
\[H^*(\Sn_{\cC}^1; \F_4) \rightarrow H^*(C_6;  (E_{\cC})_*V(0))\]
factors through the map
\[H^*(G_{24};  (E_{\cC})_*V(0)) \ra H^*(C_6;  (E_{\cC})_*V(0))\] 
induced by the inclusion of $C_6$ in $G_{24}$. Therefore, $k$ acts by multiplication by $h^4$ on $E_1^{p,q}$ for $p=1$ and $p=2$. We collect these remarks in the following lemma.
\begin{lem}\label{g0-lin}
The differentials in the ADSS are $k$--linear, where the action of $k$ is given by multiplication by $k$ on $E_r^{0,*}$ and $E_r^{3,*}$, and by multiplication by $h^4$ on $E_r^{p,*}$ for $p=1,2$.
\end{lem}
This will allow us to compute some of the differentials $d_1: E_1^{p,q} \rightarrow E_1^{p+1,q}$ for $q>0$ based on our results for $q=0$.
\begin{lem}\label{lem:d1klin}
Let $x \in E_1^{0,q}$. The differential $d_1: E_1^{0,q} \rightarrow E_1^{1,q}$ is zero unless $x = v_1^r\eta^s\Delta^t$ or $x =  v_1^r k^s  \Delta^t$, in which case it is given by
\[d_1( v_1^r \eta^s\Delta^t) =  v_1^r\eta^s d_1(\Delta^t)\]
and
\[d_1(  v_1^r k^s \Delta^t) =  v_1^r h^{4s} d_1(\Delta^{t}).\]
\end{lem}
\begin{proof}
There is no $v_1$--torsion in $E_1^{1,q}$, and $d_1$ is $v_1$--linear. Therefore, if $x$ is $v_1$--torsion, we must have $d_1(x) =0$. The only classes in $E_1^{0,q}$ which are not $v_1$--torsion are of the form $x = v_1^r\eta^s\Delta^t$ or $x = v_1^r k^s  \Delta^t$. The statement then follows from the $\eta$ and $k$--linearity of the differentials.
\end{proof}

\begin{lem}\label{lem:d1hlin}
Let $x \in E_1^{1,q}$. The differential $d_1: E_1^{1,q} \rightarrow E_1^{2,q}$ satisfies
\[h^kd_1(x) = d_1(h^k x).\]
\end{lem}

\begin{proof}
This follows from the fact that the differentials are $\eta = hv_1$ and $v_1$--linear. Indeed, since $\eta = v_1 h$, we have the following equalities
\[ v_1^k h^k  d_1(x) = \eta^k d_1(x) =  d_1(\eta^k x) = d_1(v_1^k  h^k x) = v_1^k d_1 (h^k x) .\]
There is no $v_1$ or $h$--torsion in $E_1^{1,q}$ and $E_1^{2,q}$, so $h^kd_1(x) = d_1(h^k x)$.
\end{proof}

Understanding the differential $d_1: E_1^{2,q} \rightarrow E_1^{3,q}$ is more subtle as there is $v_1$--torsion in $E_1^{3,q}$ for $q > 0$. We will use the following result. Its proof is postponed until the end of the section.
\begin{lem}\label{v13-div}
For $x \in E_1^{2,0}$, there is a unique $y \in E_1^{3,0}$ with
$d_1(x)=v_1^3y$.
\end{lem}
\begin{proof}
Let $\tau' = \pi \tau \pi^{-1}$.
Recall that $d_1 : E_1^{2,q} \ra E_1^{3,q}$ is given by 
\[  \phi_{\pi} (id+\phi_{i}+\phi_{j}+\phi_{k})  (e-\phi_{\alpha}^{-1}) \phi_{\pi}^{-1}=(id+\phi_{i'}+\phi_{j'}+\phi_{k'}) (e-\phi_{\alpha}^{-1}). \]
Further, this factors as the composite
\[\xymatrix{E_1^{2,q} \ar[rrr]^-{e-\phi_{\alpha}^{-1}} & & & E_1^{2,q} \ar[rrr]^-{(id+\phi_{i'}+\phi_{j'}+\phi_{k'})} & & & E_1^{3,q}.}  \]
Let $x \in E_1^{2,0}$. It follows from \fullref{alphav2} that
there exists $z \in E_1^{2,q}$ such that
$(e-\phi_{\alpha}^{-1})(x) = v_1^3 z$. Then, by $v_1$--linearity, 
\[d_1(x) = v_1^3 (id+\phi_{i'}+\phi_{j'}+\phi_{k'}) (z),\]
and $y = (id+\phi_{i'}+\phi_{j'}+\phi_{k'}) (z)$. This element is uniquely determined since $E_1^{3,0}$ is $v_1$--torsion free.
\end{proof}

\begin{lem}\label{lem:d1khlin}
Let $x$ be an element of $E_1^{2, 0}$. Consider $d_1 : E_1^{2,q} \rightarrow E_1^{3,q}$. Let $q= 4t+s$ for $0\leq s \leq 3$. Then
$d_1(h^q x) =k^t \eta^s (v_1^{-s} d_1(x))$,
where $ (v_1^{-s} d_1(x))$ is uniquely determined since $E_1^{3,0}$ is $v_1$--torsion free.
\end{lem}
\begin{proof}
Let $y$ be as in \fullref{v13-div} so that $d_1(x) = v_1^3 y$.
Since $\alpha$ is in the centralizer of $C_6$ in $\mathbb{S}_2$, it has a trivial action on $H^*(C_6, \F_4) \cong \F_4[h]$. Hence, $e-\phi_{\alpha}^{-1}: E_1^{2,*} \ra E_1^{2,*}$ is $h$--linear. Let $q = 4t+s$ as in the statement of the result. Recall that $\eta = v_1h$, that $k = h^{4}$ and that all maps are $v_1$, $k$ and $\eta$--linear. Therefore, 
\begin{align*}
d_1(h^q x) &=  (id+\phi_{i'}+\phi_{j'}+\phi_{k'})(e-\phi_{\alpha^{-1}})(h^q x) \\
&=  k^t (id+\phi_{i'}+\phi_{j'}+\phi_{k'}) (h^s (e-\phi_{\alpha^{-1}})(x)) \\
&=    k^t  (id+\phi_{i'}+\phi_{j'}+\phi_{k'}) (\eta^s v_1^{3-s} z)\\
&=   k^t \eta^s v_1^{3-s} y \\
&=   k^t \eta^s v_1^{-s} d_1(x). \qedhere
\end{align*}

\end{proof}
This completes the computation of the $E_2$--term (see \fullref{fig:SSRS-E2}).

\subsection{Higher Differentials}\label{sec:higher}
In this section, we prove that all differentials $d_r: E_r^{0,q} \rightarrow E_{r}^{r, q-r+1}$ for $r\geq2$ are zero. Because of the sparsity of the spectral sequence, the only differentials $d_r$ for $r\geq 2$ which do not have a zero target are 
\begin{align*}
d_2 &: E_2^{0,q} \rightarrow E_{2}^{2, q-1}, \ q\geq 2\\
d_2 &: E_2^{1,q} \rightarrow  E_{2}^{3, q-1},\ q\geq 2\\
d_3 &: E_3^{0,q} \rightarrow E_{3}^{3, q-2},  \ q\geq 3.
\end{align*}

The proof of the following result is a direct computation similar to that of \fullref{HC6}.
\begin{lem}
Let $v_1$ have degree $(s,t) = (0,2)$, $v_2$ have degree $(0,6)$, and $h$ have degree $(1,0)$. 
Then
\[ H^*(C_6 ; (E_{\cC})_*) \cong \W[\![u_1^3]\!][v_1^2, v_1v_2, v_2^{\pm 1}, h]/(2h). \]
\end{lem}
As in \fullref{lem-bock}, let $\beta$ be the connecting homomorphism for the long exact sequence in cohomology associated to
\[0 \ra (E_{\cC})_*/2\to (E_{\cC})_*/4 \ra (E_{\cC})_*/2 \ra 0\] 
and note that $(E_{\cC})_*V(0) \cong (E_{\cC})_*/2$.

\begin{lem}\label{lem:d213}
All differentials $d_2:  E_2^{1,q} \rightarrow  E_{2}^{3, q-1}$ are zero.
\end{lem}

\newpage
\hvFloat[%
floatPos=H,%
capVPos=c,%
capWidth=w,
rotAngle=90,
objectPos=c]{figure}{\includegraphics[height=0.86\textwidth]{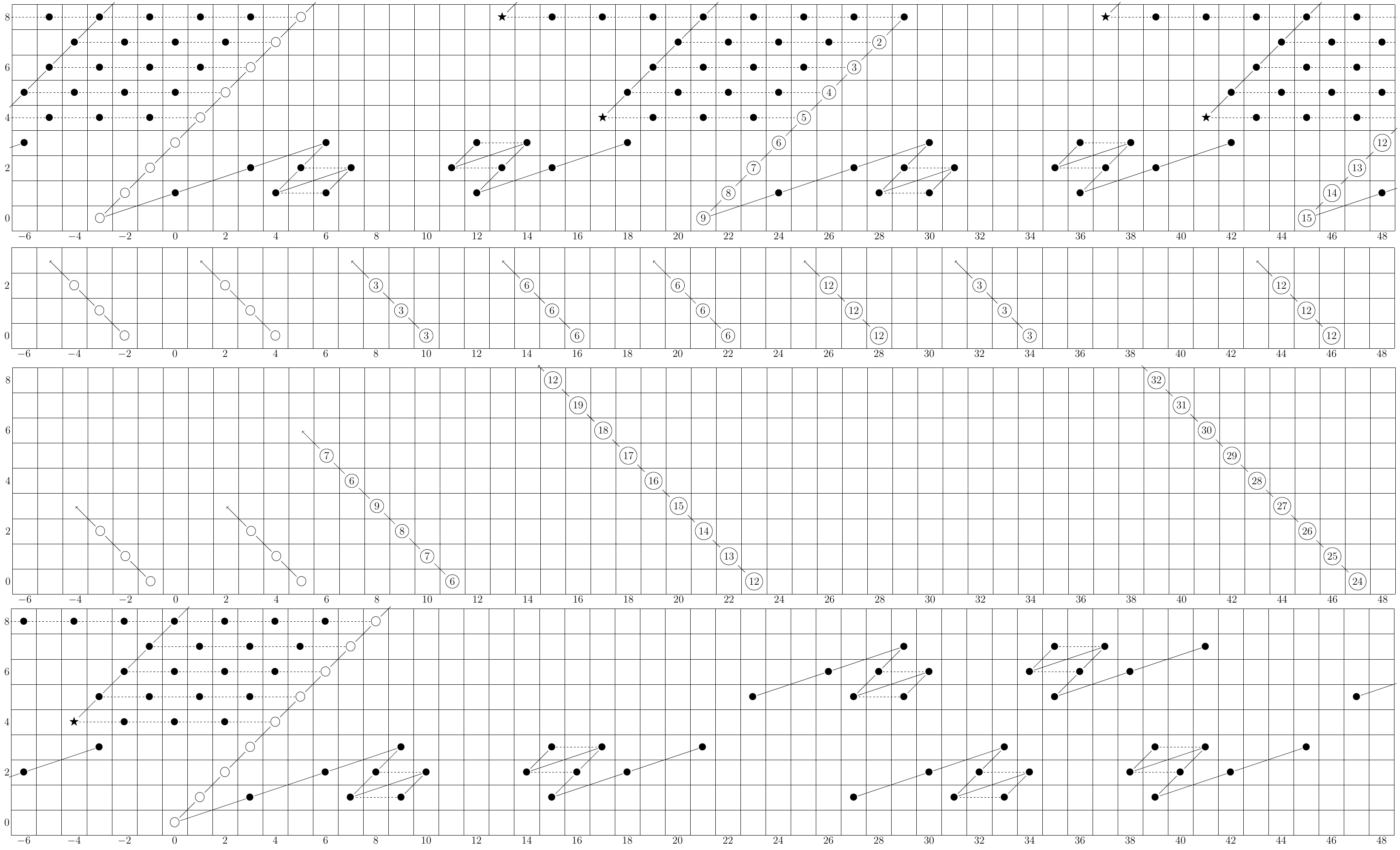}}%
{The $E_2$-term of the ADSS with coefficients $(E_{\cC})_*V(0)$. The notation and grading is as in \fullref{fig:SSRS-E1}. In addition, a $\copyright$ is a copy of $\mathbb{F}_4[v_1]/(v_1^{\text{c}})$.
In \fullref{sec:higher}, we prove that $E_2 \cong E_{\infty}$. Therefore, this is also the $E_{\infty}$-term of the ADSS.}{fig:SSRS-E2}

\begin{proof}
Let $\db_n$ be as in \fullref{partial1}. The set 
\[B = \{h^k\db_n \mid n=0,1, 2^{s+1}(1+4t), \ 0\leq k \leq 3, \ 0\leq s\}\]
generates $E_2^{1,*}$ as an $\F_4[v_1, k]$--module, for $k$ as in \fullref{g0-lin}. Because the differentials are $\F_4[v_1, k]$--linear and $k$ acts via multiplication by $h^4$, it suffices to show that the $d_2$--differentials are zero on the elements of $B$. First, note that $d_2(\db_n) =0$ for all $n$, since the targets of these differentials are zero. Hence, it suffices to show that $d_2(h^k\db_n)=0$ for $1 \leq k \leq 3$. 

The first remark is that, if $d_2(h^k\db_n) = 0$, then 
\[v_1d_2(h^{k+1}\db_n)  =d_2(v_1h^{k+1}\db_n) = d_2(\eta h^k \db_n)= \eta d_2(h^k \db_n)=0.\]
Hence, if $d_2(h^k\db_n)=0$, then $v_1d_2(h^{k+1}\db_n)=0$. Further,
\[v_1^kd_2(h^k\db_n)= d_2(\eta^k\db_n) = \eta^k d_2(\db_n).\]
Since $d_2(\db_n)=0$, we must have that $v_1^kd_2(h^k\db_n)=0$ for all $k\geq 0$.

Let $1\leq k \leq 3$. Then $d_2 (h^k\db_0)$ is an element of internal degree $t=0$ in $E_{2}^{3,k-1}$. Since $d_2(\db_0)=0$, $v_1d_2(h\db_0)=0$. However, there is no $v_1$--torsion in $(E_{2}^{3,0})_0$, hence $d_2(h\db_0)=0$. Further, $(E_{2}^{3,1})_0$ and $(E_{2}^{3,2})_0$ are zero and $d_2(h^k\db_0) =0$ for $k=2,3$.

Next, consider the elements of the form $h^k\db_1$ for $1\leq k \leq 3$. Since $d_2(h^k\db_1)$ is an element of internal degree $t=6$ in $E_2^{3,k-1}$ and there is no $v_1$--torsion in $(E_2^{3,k-1})_6$ for $1\leq k \leq 3$, these differentials must be zero. 

The classes $h^k\db_{2^{s+1}(1+4t)}$ have internal degree $3 \cdot 2^{s+2}(1+4t)$. Hence, their degree is congruent to zero modulo $3$. First, consider the case when $k=1$. The possible targets for the $d_2$ differentials on these classes are in $E_2^{3,0}$ and must be annihilated by $v_1$. Therefore, they must be of the form 
\[v_1^{3(1+2^{s'+1})-1} \od_{2^{s'}(1+2t')}.\] 
However, such classes have internal degree congruent to $1$ modulo $3$, since the degree of $\od_{2^{s'}(1+2t')}$ is $24 \cdot 2^s(1+2t')$ and the degree of $v_1$ is $2$. Hence, there is no appropriate target for these differentials. Further, this implies that $d_2(h^2\db_n)$ is annihilated by $v_1$. 

The classes which are annihilated by $v_1$ in $E_2^{3,1}$ are of one of the forms 
\[v_1^{3(1+2^{s'+1})-2} \eta \od_{2^{s'}(1+2t')},\]
$\nu \od_{2^{s'}(1+2t')}$, $v_1 x \od_{2^{s'}(1+2t')}$, or $y \od_{2^{s'}(1+2t')}$. Here, $\nu$ has internal degree $4$, $x$ has internal degree $8$ and $y$ has internal degree $16$. Again, such classes have internal degree congruent to $1$ modulo $3$, so there is no possible target for the differentials. This, in turn, implies that $d_2(h^3\db_n)$ is annihilated by $v_1$. 

The classes in $E_2^{3,2}$ which are annihilated by $v_1$ are of one of the forms 
\[v_1^{3(1+2^{s'+1})-3} \eta^2 \od_{2^{s'}(1+2t')},\]
${\nu}^2 \od_{2^{s'}(1+2t')}$,  ${v_1 \eta x} \od_{2^{s'}(1+2t')}$, $ \eta y \od_{2^{s'}(1+2t')}$ or $ {\nu}y \od_{2^{s'}(1+2t')}$. Of these classes, both ${v_1 \eta x}  \od_{2^{s'}(1+2t')}$ and $ {\eta}y \od_{2^{s'}(1+2t')}$ have internal degree congruent to $0$ modulo $3$, so we must make a more careful analysis.

Note that $3 \cdot 2^{s+2}(1+4t) \equiv 0  \mod 24$ if $s\geq 1$, and $3 \cdot 2^{2}(1+4t) \equiv 12  \mod 24$. Since the internal degree of $ {\eta}y \od_{2^{s'}(1+2t')}$ is congruent to $18$ modulo $24$, it cannot be hit by a differential. The internal degree of ${v_1 \eta x} \od_{2^{s'}(1+2t')}$ is $12$ modulo $24$. Therefore, there are possible differentials $d_2(h^3 \db_{2(1+4t)})$ with targets ${v_1 \eta x} \od_{2t}$. However, by \fullref{lem:beta0} and \fullref{lem:betanot0} below,
$\beta(d_2(h^3\db_{2(1+4t)})) = 0$ and $\beta({v_1 \eta x}\od_{2t}) \neq 0$. Therefore, 
$d_2(h^3\db_{2(1+4t)})\neq {v_1 \eta x} \od_{2t}$
and we must have $d_2(h^3\db_{2(1+4t)}) =0$.
\end{proof}

\begin{lem} \label{lem:beta0}
$\beta(d_2(h^3\db_{2(1+4t)})) = 0$
\end{lem}

\begin{proof}
The maps 
$H^0(C_6, (E_{\cC})_{12(1+4t)}) \ra H^0(C_6, (E_{\cC})_{12(1+4t)}/2)$
are surjective and factor through $H^0(C_6, (E_{\cC})_{12(1+4t)}/4)$. 
Hence,  
$\beta(\db_{2(1+4t)}  ) =0$.
Since $\beta(h^{2t +1}) = h^{2t+2}$,
it follows that
\[\beta(h^3 \db_{2(1+4t)}) = \beta(h^3) \db_{2(1+4t)} + h^3 \beta( \db_{2(1+4t)})=h^4 \db_{2(1+4t)}.\]
By \fullref{lem-bock},
\[\beta(d_2(h^3 \db_{2(1+4t)} ))= d_2(\beta(h^3 \db_{2(1+4t)})) = d_2(h^4 \db_{2(1+4t)})= {k}d_2( \db_{2(1+4t)})=0.\qedhere\]
\end{proof}

\begin{lem}  \label{lem:betanot0}
$\beta({v_1 \eta x} \od_{2t}) \equiv {\nu}^3 \od_{2t}$ modulo $(v_1^{12})$ and, hence, is non-zero.
\end{lem}
\begin{proof}
By \fullref{rem:congruencesv}, $\od_{2t} \equiv \Delta'^{2t}$ modulo $(v_1^{12})$. Since $\beta$ is a derivation and we are working in characteristic two, it is zero on squares. Hence,
\begin{align*}
\beta( {v_1 \eta x}\od_{2t}) \equiv  \beta({v_1 \eta x}) \Delta'^{2t}  \mod (v_1^{12}).
\end{align*}
So it is enough to prove that $\beta(v_1 \eta x)= \nu^3$. By definition, $\beta(v_1) = \eta$, so $\beta(\eta)=0$. Further, using the fact that $\eta^2x=\nu^3$,
\[\beta(v_1\eta x) = \beta(v_1 \eta) x+ v_1 \eta \beta(x) = \eta^2 x+ v_1 \eta \beta(x) = \nu^3 + v_1 \eta \beta(x) .\] 
However, since $v_1^2x=0$, $\beta(x)$ is $v_1$--torsion and in $H^*(G_{24}', (E_{\cC})_8V(0))$, the $v_1$--torsion is annihilated by $(v_1)$. Hence, $v_1 \eta \beta(x) =0$. 
\end{proof}

The next few results will be necessary to prove that all remaining higher differentials are zero. Let $C_2 = \{\pm 1\}$ in $\mathbb{S}_2$. For any group $G \subseteq \mathbb{S}_{\cC}$ that contains $C_2$, let $PG = G/C_2$. For any $\Z_2[\![\mathbb{S}_2]\!]$--module $M$ on which $C_2$ acts trivially, the action of $G$ descends to an action of $PG$. Further, there is a homomorphism
\[ H^*(PG, M) \to H^*(G, M)  \]
which is an isomorphism when $*=0$.

Note that $PG_{24} \cong A_4$ and $PC_6 \cong C_3$ and that the algebraic duality resolution is a resolution of $P\Sn_{\cC}^1$--modules, where $\sC_{0} \cong \Z_2[\![P\Sn_{\cC}^1/A_4]\!]$,  $\sC_1 \cong \sC_2 \cong \Z_2[\![P\Sn_{\cC}^1/C_3]\!]$ and  $\sC_3 \cong \Z_2[\![P\Sn_{\cC}^1/A_4']\!]$ for $A_4' = PG_{24}'$. Further, since $C_2$ acts trivially on $(E_{\cC})_*V(0)$, the action of $\Sn_{\cC}^1$ descends to an action of $P\Sn_{\cC}^1$. There is a corresponding ADSS
\begin{align}\label{eqn:algssPS}
F_{1}^{p,q} = \Ext_{\Z_2[\![P\Sn_{\cC}^1]\!]}^{q}(\sC_p, (E_{\cC})_*V(0)) \Longrightarrow H^{p+q}(P\Sn_{\cC}^1,(E_{\cC})_*V(0)), \end{align}
with $F_{1}^{p,q} \cong H^q(PF_p,  (E_{\cC})_*V(0))$
 and a map of spectral sequences 
 \[ \varphi: F_{r}^{p,q} \rightarrow E_{r}^{p,q}\] 
 induced by the projection from $\Sn_{\cC}^1$ to $P\Sn_{\cC}^1$.

We will relate the computation of some differentials in $E_{r}^{p,q}$ to the computation of differentials $F_{r}^{p,q}$.  The advantage of this method is that the spectral sequence $F_{r}^{p,q}$ is sparser than $E_{r}^{p,q}$. Indeed, $\sC_1$ and $\sC_2$ are projective $\Z_2[\![P\Sn_{\cC}^1]\!]$--modules. Hence, for $p=1$ or $p=2$,
\[ F_{1}^{p,q} \cong \Ext^q_{\Z_2[\![P\Sn_{\cC}^1]\!]}(\Z_2[\![P\Sn_{\cC}^1/C_3]\!], (E_{\cC})_*V(0)) \]
 is zero when $q>0$. Hence, $F_{r}^{p,q} = 0$ when $q\geq 0$ for $p=1$ or and $p=2$. Further, the induced maps $F_1^{p,0} \rightarrow E_1^{p,0}$ are isomorphisms as noted above. In fact, the complexes  $E_1^{*,0} \cong F_1^{*,0}$ are isomorphic. Therefore, the computation of $F_2^{p,q}$ follows immediately from that of $E_2^{p,q}$ and
$F_2^{p,0} \cong E_2^{p,0}$.

Let $A_4 = G_{24}/C_2$. Since $C_2$ acts trivially on $(E_{\cC})_*V(0)$, the action of $G_{24}$ descends to an action of $A_4$. 
\begin{lem}\label{prop:a4v1tor}
Let $R^{\wedge} =  \F_4[\![j]\!][v_1, \Delta^{\pm 1}]/(v_1^{12}= j\Delta)$. The inclusion
\[H^1(A_4, (E_{\cC})_*V(0))  \to H^1(G_{24}, (E_{\cC})_*V(0))\]
gives an isomorphism of $R^{\wedge}$--modules
\[H^1(A_4, (E_{\cC})_*V(0)) \cong R^{\wedge}/(v_1^2)\{x\} \oplus R^{\wedge}/(v_1)\{\nu, y\} .\]
In particular, the module $H^1(A_4, (E_{\cC})_*V(0))$ is annihilated by $v_1^2$. 
\end{lem}
\begin{proof}
Let $S_*(\rho)$ be as in \fullref{Arem:rho}.
It suffices to prove that, for $R = \F_4[v_1,\Delta]$,
\[H^1(A_4, S_*(\rho)) \cong R /(v_1^2)\{x\} \oplus R/(v_1)\{\nu, y\} .\]
Consider the spectral sequence for the group extension
\[ 1 \to C_2 \to G_{24} \to A_4 \to 1.\]
From the associated inflation-restriction exact sequence, using the fact that $S_*(\rho)^{C_2} = S_*(\rho)$, we obtain an exact sequence
\[ 0  \ra H^1(A_4, S_*(\rho)) \ra H^1(G_{24}, S_*(\rho) ) \ra H^1(C_2, S_*(\rho))^{A_4}. \]
From \fullref{mainthm}, for $R = \F_4[v_1, \Delta]$, we have 
$S_*(\rho)^{G_{24}} \cong S_*(\rho)^{A_4} \cong R$
and
\[ H^1(G_{24}, S_*(\rho) )  \cong R\{\eta\} \oplus R/(v_1^2)\{x\} \oplus R/(v_1)\{\nu, y\} .  \]
Further, $H^1(C_2, S_*(\rho))^{A_4} \cong R\{h\}$
for $h$ of internal degree $0$. Since $R\{h\}$ is $v_1$--free, $R/(v_1^2)\{x\} \oplus R/(v_1)\{\nu, y\}$ maps to zero and
the map $H^1(G_{24}, S_*(\rho) ) \ra H^1(C_2, S_*(\rho))^{A_4}$ sends $\eta$ to $v_1h$. 
\end{proof}

\begin{prop}\label{A4-G24}
The map 
\[\phi : H^*(A_4, (E_{\cC})_*V(0)) \rightarrow H^*(G_{24}, (E_{\cC})_*V(0))/({\eta})\] 
induced by the projection $G_{24} \ra G_{24}/C_2 \cong A_4$ is an isomorphism in degree $\ast = 0$ and is surjective in degrees $\ast \leq 3$.
\end{prop}
\begin{proof}
It suffices to prove that $\phi$ is surjective if we replace $(E_{\cC})_*V(0)$ by $S_*(\rho)$. It follows from \fullref{mainthm} that $H^*(G_{24}, S_*(\rho))/({\eta})$ is generated by the fixed points $H^0(G_{24}, S_*(\rho))$ and the elements $\nu$, $x$ and $y$ in degrees $0 \leq * \leq 3$. Since $C_2$ has a trivial action, $S_*(\rho)^{A_4} \cong S_*(\rho)^{G_{24}}$. Hence, $\phi$ is an isomorphism in degree zero. By \fullref{prop:a4v1tor}, $\nu$, $x$ and $y$ are in the image of $\phi$. The result follows from the fact that $\phi$ is a ring homomorphism.
\end{proof}

\begin{rem}\label{rem:relFE}
It follows from \fullref{A4-G24} that the map 
\[\varphi : F_1^{0,*} \rightarrow E_1^{0,*}/({\eta})\] 
induced by the projection $G_{24} \ra G_{24}/C_2 \cong A_4$ is surjective in degrees $\ast \leq 3$.
All classes of degree $q\geq 4$ in $E_r^{0,q}$ are multiples of ${k}$, so their differentials will be determined by differentials on classes of degree $q\leq 3$. Further, by ${\eta}$--linearity it suffices to show that the differentials on the classes in the image of $\varphi$ of \fullref{A4-G24} are zero. It is therefore sufficient to compute some of the differentials  $d_r : F_r^{0,q} \rightarrow F_{r}^{r, q-r+1}$ for $q\leq 3$.
 \end{rem}

The following results are generalizations of results that can be found in Henn, Karamanov and Mahowald \cite[Section 6]{HKM}. The first is \cite[Lemma 6.1]{HKM}.
\begin{lem}[Henn-Karamanov-Mahowald]\label{mapj}
Let $R$ be a $\Z_2$-algebra and $M$ be an $R$--module. Let
\[0 \rightarrow \sC_3 \xra{\partial_3} \sC_2 \xra{\partial_3} \sC_1 \xra{\partial_2} \sC_0 \xra{\partial_1} \Z_2 \rightarrow 0\]
be an exact sequence of $R$--modules such that $\sC_1$ and $\sC_2$ are projective. Define $N_i$ recursively by $0 \rightarrow N_i \rightarrow \sC_i \xra{\partial_i} N_{i-1} \rightarrow 0$, and let 
 $F_r^{s,t}$ be the first quadrant spectral sequence of the exact couple 
 \[ \xymatrix{\Ext_R(N_i, M) \ar@{.>}[rr]  & &  \Ext_R(N_{i-1}, M) \ar[dl] \\ &  \Ext_R(\sC_i, M)  \ar[ul] }\]
 Then $E_1^{p,q} = 0$ for $0<p<3$ and $q>0$. Further, there are isomorphisms
\[ \Ext_{R}^q(N_0, M) \cong \begin{cases} \ker(F_{1}^{1,0} \rightarrow F_{1}^{2,0}  ) &q=0 \\
  F_2^{q+1,0} \cong F_3^{q+1,0}  &q=1,2 \\
  F_3^{q-2,0} &q\geq 3. 
 \end{cases}\]
Let $j : N_0 \rightarrow \sC_0$ be the inclusion. The only possible non-zero higher differentials are of the form $d_r : F_{r}^{0, q} \rightarrow  F_{r}^{r, q-r+1}$, and they can be identified with the map $\Ext_{R}^q(\sC_0, M ) \rightarrow \Ext_{R}^q(N_0, M )$ induced by $j$.
\end{lem}
\noindent
Note that the algebraic duality resolution viewed as a resolution of the trivial $\Z_2[\![P\Sn_{\cC}^1]\!]$--module $\Z_2$ and the associated spectral sequence (\ref{eqn:algssPS}) satisfy the conditions of \fullref{mapj}.

Let $P_{-1}= \Z_2[\![P\Sn_{\cC}^1/A_4]\!] $ and $P_0 = \Z_2[\![P\Sn_{\cC}^1/C_3]\!]$. Let $P_0 \xra{\varepsilon} P_{-1}$ be the natural augmentation which sends the coset $[C_3]$ to $[A_4]$. Complete this to a projective resolution $P_*$ of $ \Z_2[\![P\Sn_{\cC}^1/A_4]\!]$. Let $P_*'$ be any projective resolution of  $\Z_2[\![P\Sn_{\cC}^1/A_4']\!]$. Let $N_0$ be defined by the exact sequence 
\[0 \rightarrow N_0 \rightarrow \sC_0 \xra{\varepsilon} \Z_2 \rightarrow 0.\] 
Letting $Q_{-1} = N_0$, $Q_q = \sC_{q+1} = \Z_2[\![P\Sn_{\cC}^1/C_3]\!]$ for $q = 0,1$, the complex 
\begin{align}\label{eqn:Qis}
Q_{2} \ra Q_{1} \ra Q_0 \ra Q_{-1},
\end{align}
with maps as in the algebraic duality resolution, is the beginning of a projective resolution of $N_0$. The kernel of $Q_{2} \ra Q_{1}$ is isomorphic to $\Z_2[\![P\Sn_{\cC}^1/A_4']\!]$. Splicing (\ref{eqn:Qis}) with $Q_{*}=P_{*-3}'$ gives a projective resolution $Q_*$ of $N_0$. 
\begin{lem}
There is a map $\phi: Q_{*} \rightarrow P_{*}$ such that
\[ \phi_0: Q_0 \rightarrow P_0\]
covers the map $j: N_0 \rightarrow \sC_0 = P_{-1}$ which sends $e_1 \rightarrow (e-\alpha)e_0$.
\end{lem}
\begin{proof}
Note that $Q_0 \cong P_0 \cong \Z_2[\![P\Sn_{\cC}^1/C_3]\!]$. So the map which sends the generator of $e \ox 1 \in Q_0$ to $(e-\alpha) \ox 1 \in P_0$ is well defined and covers $j$. Hence, it extends to a chain map $\phi$.
\end{proof}
The following is an observation in Henn, Karamanov and Mahowald \cite[Section 6]{HKM}. It follows from \fullref{mapj}.
\begin{lem}\label{lem:Tstarstar}
Let $T_{*,*}$ be the double complex satisfying $T_{*, 0} = P_{*}$ and $T_{*, 1} = Q_{*}$ with vertical differentials $\delta_P$ and $\delta_Q$ and horizontal differentials $\phi_s : Q_s \rightarrow P_s$. Up to reindexing, the filtration of the spectral sequence of this double complex agrees with that of the ADSS.
\end{lem}

The following result is an adaptation of part of \cite[Lemma 6.5]{HKM}. 
\begin{lem}\label{hkm-65}
Let $s>0$. Let $z \in H^s(A_{4},(E_{\cC})_*V(0))$ be such that
$v_1^r z = 0$.
Let $c \in \Hom_{\Z_2[\![P\Sn_{\cC}^1]\!]}(P_{s}, (E_{\cC})_*V(0))$ be a cocyle which represents $z$. Choose an element $h$ in $\Hom_{\Z_2[\![P\Sn_{\cC}^1]\!]}(P_{s-1}, (E_{\cC})_*V(0)) $ such that
$\delta_P(h) = v_1^r c$.
Let 
\[ \phi^* :  \Hom_{\Z_2[\![P\Sn_{\cC}^1]\!]}(P_{*}, (E_{\cC})_*V(0)) \rightarrow \Hom_{\Z_2[\![P\Sn_{\cC}^1]\!]}(Q_{*}, (E_{\cC})_*V(0)) \]
be induced by $\phi$.
There are elements $d$ and $d'$ in $\Hom_{\Z_2[\![P\Sn_{\cC}^1]\!]}(Q_{s-1}, (E_{\cC})_*V(0)) $ and an element $d''$ in $\Hom_{\Z_2[\![P\Sn_{\cC}^1]\!]}(Q_{s}, (E_{\cC})_*V(0))$ such that
\begin{align}\label{phib} \phi^*_{s-1}(h) = d' + v_1^rd \end{align}
and $\delta_Q(d') = v_1^r d''$. For $d''$ as above and $j$ as in \fullref{mapj},
\[ j^*(z) = [d''] \in \Ext^s_{\Z_2[\![P\Sn_{\cC}^1]\!]}(N_0,(E_{\cC})_*V(0)).\]
\end{lem}

\begin{proof}
Note that $\Hom_{\Z_2[\![P\Sn_{\cC}^1]\!]}(T_{*,*}, (E_{\cC})_*V(0))$ for $T_{*,*}$ as in \fullref{lem:Tstarstar} is a double complex of $\F_4[v_1]$--modules which have no $(v_1)$--torsion. 
We can write
\[ \phi^*_{s-1}(h) = d' + v_1^rd. \] 
To prove (\ref{phib}), note that by $v_1$--linearity,
\begin{align*} 
\delta_Q(d') + v_1^r\delta_Q(d) 
=  \delta_Q(\phi^*_{s-1}(h) )
= \phi^*_{s}(\delta_P(h ))
=v_1^r \phi^*_{s}( c).
\end{align*}
Hence,  $\delta_Q(d') \equiv 0$ modulo $(v_1^r)$, that is, $\delta_Q(d') = v_1^r d''$
for some $d''$.

Now, note that
\begin{align*} 
v_1^rj^*(c) 
=\phi_s^*( v_1^r c) 
=\phi_s^*( \delta_P(h)) 
=  \delta_Q(\phi^*_{s-1}(h) )
=  v_1^r d'' + v_1^r\delta_Q(d).
\end{align*}
Since there is no $v_1$--torsion in $\Hom_{\Z_2[\![P\Sn_{\cC}^1]\!]}(T_{*,*}, (E_{\cC})_*V(0)) $, we must have
$j^*(c) =  d'' + \delta_Q(d)$. 
This reduces to
\[ j^*(z) = [d''] \in \Ext^s_{\Z_2[\![P\Sn_{\cC}^1]\!]}(N_0,(E_{\cC})_*V(0)).\qedhere\]
\end{proof}

\begin{lem}\label{d20}
Let $z$ be in $F_2^{0, q}$. Then $d_2(z) =0$.
\end{lem}

\begin{proof}
If $q>1$, then $d_2(z) = 0$ since the target of the differential is zero. Suppose that $q=1$. Then $z$ is $v_1$--torsion. Let $r$ be the smallest integer such that $v_1^r z = 0$.
By \fullref{prop:a4v1tor} of \ref{sec:appcoh}, we can choose $r=1$ or $r=2$.
 Choose $h$ as in \fullref{hkm-65} and write
 \[ \phi_{0}(h) = (e-\phi_{\alpha})(h) = d' + v_1^r d.\]
However, $\phi_{\alpha} \equiv id$ modulo $(2,u_1^3)$. So we must have $d'=0$. By \fullref{mapj} and \fullref{hkm-65}, this implies that $d_2(z) = 0$ in the ADSS for $P\Sn_{\cC}^1$.
 \end{proof}

\begin{cor}\label{2-line-d2}
All differentials $d_2 : E_{2}^{0,q} \rightarrow E_2^{2, q-1}$ are zero.
\end{cor}
\begin{proof}
This follows from \fullref{rem:relFE} and \fullref{d20}.
\end{proof}

\begin{lem}\label{d30}
All differentials $d_3 : E_{3}^{0,q} \rightarrow E_3^{3, q-1}$ are zero.
\end{lem}
\begin{proof}
Differentials $d_3:  E_{3}^{0,q} \rightarrow E_{3}^{3,q-2}$ are zero for degree reasons if $0 \leq q <2$. By \fullref{2-line-d2}, the classes ${\nu}\Delta^{s}$ survive to the $E_3$-term, and hence they must be permanent cycles. Thus, they represent cohomology classes in $H^*(\Sn_{\cC}^1,(E_{\cC})_*V(0))$. By \fullref{S21-lin}, the differentials are ${\nu}\Delta^s$--linear for all $s \in \Z$. Using this fact and linearity with respect to ${\eta}$ and $v_1$, the problem reduces to verifying the claim for $ x^2\Delta^{s}$. However, by the same argument, $x$ is a permanent cycle and 
$d_3(x^2\Delta^s) = xd_3(x\Delta^s)=0$.
\end{proof}

\begin{lem}\label{dr0}
All differentials $d_r : E_r^{0,q} \rightarrow E_{r}^{r, q-r+1}$ are zero.
\end{lem}
\begin{proof}
By \fullref{d20} and \fullref{d30}, $E_2^{*,*} \cong E_4^{*,*}$ and the spectral sequence collapses at $E_4$ since the targets for higher differentials are zero.
\end{proof}

\section{The action of the Morava stabilizer group}\label{sec:appC}

The goal of this section is to approximate the action of elements of $\Sn_{\cC}$ on $(E_{\cC})_*$. Some of our results are stronger than needed for the computations of this paper, but the better estimates are necessary for future computations. Note that the results of \fullref{subsec:act} rely on this section.

\subsection{The formal group laws}\label{subsec:fgls}
Let $\cE$ be an elliptic curve with Weierstrass equation
\[\mathcal{E} : y^2 +a_1xy +a_3y = x^3 + a_2x^2+a_4x+a_6.\]
Let $F_{\cE}(z_1, z_2)$ be the formal group law of $\cE$, where the coordinates $(z,w)$ at the origin are chosen so that
\begin{equation}\label{weiereq}
w(z) = z^3+a_1zw(z)+a_2z^2w(z)+a_3w(z)^2+a_4zw(z)^2+a_6w(z)^3.
\end{equation}
That the group $\Sn_{\cC}$ acts on $(E_{\cC})_*$ is a consequence of the fact that the formal group law $F_{E_{\cC}}$ of $E_{\cC}$ is a universal deformation of the formal group law $F_{\cC}$ of the elliptic curve
\[\cC: y^2+y=x^3\]
defined over any field extension of $\F_2$. Further, $F_{E_{\cC}}$ is the formal group law of an elliptic curve, namely
\[\cC_U  : y^2+3u_1xy+(u_1^3-1)y = x^3\]
defined over $(E_{\cC})_0$. That is, 
$F_{E_{\cC}} = F_{\cC_U}$.

We start by gathering information about $F_{\cC_U}$. 
We will also compute information about the formal group law of the curve
\[\cC_{\W} :  y^2-y=x^3 \]
defined over $\W$. The curve $\cC_{\W}$ is a lift of $\cC$ to $\W$, and $\cC_U$ reduces to $\cC_{\W}$ modulo $(u_1)$. We will derive information about $F_{\cC_U}$ from knowledge of $F_{\cC_{\W}}$.

The following results are proved using the methods described in Silverman \cite[Section 4]{silverman}. We recall the key tools here. We restrict to elliptic curves $\cE$ with homogeneous Weierstrass equation of the form
\[\mathcal{E} : y^2z +a_1xyz +a_3yz^2 = x^3.\] 
Let $z=-\frac{x}{y}$ and $w = -\frac{z}{y}$, so that $(z, w(z))$ is a coordinate chart of $\mathcal{E}$ at the origin, with
\begin{equation}\label{eqn:funcw}
w(z) = z^3 +a_1zw(z)+a_3w(z)^2 .\end{equation}
This can be used to write $w(z)$ as a power series in $z$. Letting
\begin{equation*}
\lambda(z_1,z_2) = \frac{w(z_2) -w(z_1)}{z_2-z_1},
\end{equation*}
the line through the points $(z_1, w(z_1))$ and $(z_2, w(z_2))$ has equation 
\begin{equation}\label{lineq}
w(z) = \lambda(z_1,z_2) z + w(z_1)-\lambda(z_1,z_2)z_1.\end{equation}
(Note that there is a sign mistake in Silverman \cite[Section 4.1]{silverman} in the equation of the connecting line.) Substituting (\ref{lineq}) in (\ref{weiereq}), we obtain a monic cubic polynomial whose roots are $z_1$, $z_2$ and $[-1]_{F_\mathcal{E}}\left(F(z_1,z_2)\right)$. The coefficient of $z^2$ is $ a_1\lambda(z_1,z_2) + a_3  \lambda(z_1,z_2)^2$. This implies that
\begin{equation}\label{minusell}
[-1]_{F_\mathcal{E}}\left(F(z_1,z_2)\right) =-z_1-z_2-a_1\lambda(z_1,z_2) -a_3\lambda(z_1,z_2)^2.
\end{equation}
Noting that
\[\lambda(z,z) = \lim_{s \ra z} \frac{w(s)-w(z)}{s-z} = w'(z),\]
it follows that 
\[\![-2]_{F_\mathcal{E}}(z) = -2z- a_1w'(z)-a_3w'(z)^2.\]
Finally, the series $[-1]_{F_\mathcal{E}}(z)$, which is $[-1]_{F_\mathcal{E}}\left(F(z,0) \right)$, is given by
\begin{equation*}
[-1]_{F_\mathcal{E}}(z) = -z -a_1\frac{w(z)}{z} - a_3 \frac{w^2(z)}{z^2}
\end{equation*}
so that $F_{\mathcal{E}}$ can be computed as $[-1]_{F_\mathcal{E}}([-1]_{F_\mathcal{E}}\left(F(z_1,z_2)\right))$. For example, 
\begin{align}\label{eqn:Fmodx5}
F_{\mathcal{E}}(x,y) \equiv x+y -a_1 x y -2 a_3 xy( x^2+y^2) -3 a_3 x^2 y^2 \mod (x,y)^5.
\end{align}

The following two results give formulas for the formal group law of the curve $\cC_U$ and of its $[-2]$-series, both integrally and modulo $2$. \fullref{mod2u2} was observed computationally by the author, but was proved by Henn.
\begin{prop}\label{intminustwo}
Modulo $(x,y)^5$, 
\[F_{\cC_U}(x,y) \equiv x+y-3u_1xy-2(u_1^3-1)xy(x^2+y^2)-3(u_1^3-1)x^2y^2. \]
The formal group law $F_{\cC_U}$ has $[-2]$-series
\[\![-2]_{F_{\cC_U}}(z)= -2z-9z\frac{z u_1-2 z^2 u_1^2+z^3(u_1^3-1)}{1-6 z u_1+9 z^2 u_1^2-4 z^3(u_1^3-1)} ,\]
so that
\begin{align*}
[-2]_{F_{\cC_U}}(z) = -2z-9u_1z^2-36u_1^2z^3+9z^4-144u_1^3z^4 + O(z^5).
\end{align*}
\end{prop}
\begin{proof}
The first claim follows directly from (\ref{eqn:Fmodx5}). For the curve $\cC_{U}$, we have
\[w'(z) = \frac{3(z^2+u_1w(z))}{1-3u_1z-2(u_1^3-1)w(z)}.\]
Combining
\[\![-2]_{F_{\cC_U}}(z)  =   -2z- 3u_1w'(z)-(u_1^3-1)w'(z)^2 \]
and 
\[(u_1^3-1)w(z)^2 = w(z)-z^3-3u_1zw(z),\]
gives the result for $[-2]_{F_{\cC_U}}(z)$. Its Taylor expansion is the last estimate.
\end{proof}
\begin{cor}\label{mod2u2}
\begin{align*}
[-2]_{F_{\cC_U}} (x)\equiv  u_1x^2+ \sum_{k\geq 0}u_1^{2k}x^{2k+4} \mod (2).
\end{align*}
\end{cor}
\begin{proof}
It follows from \fullref{intminustwo} that modulo $2$,
\begin{align*}[-2]_{F_\mathcal{C_U}}(z)      &\equiv    \frac{u_1z^2+u_1^3z^4+z^4}{1+u_1^2z^2}.
        \end{align*}
        Therefore, modulo $2$,
\begin{align*}[-2]_{F_\mathcal{C_U}}(z)
           &\equiv    \left({u_1z^2+u_1^3z^4+z^4}\right) \sum_{k\geq 0}u_1^{2k}z^{2k} \\
                 &\equiv u_1z^2 +  \sum_{k\geq 0}u_1^{2k}z^{2k+4}. \qedhere
\end{align*}
\end{proof}

Some of the key ingredients for the proof of the next result were brought to the author's attention by Inna Zakharevich. Let
$C_k = \frac{(2k)!}{k!(k+1)!}$
be the $k$'th Catalan number. Let
\begin{equation}\label{catnum}
C(y) = \sum_{k\geq 0} C_ky^k = \frac{1-\sqrt{1-4y}}{2y}
\end{equation}
be their generating series (see, for example, Wilf \cite[(2.3.9)]{MR2172781}). Let $D(y) = yC(y)$, so that
\begin{equation*}
D(y) = \frac{1-\sqrt{1-4y}}{2}.
\end{equation*}

\begin{prop}\label{twoseries}
Let $\cC_{\W}$ be the elliptic curve defined over $\W$ by the Weierstrass equation
$\cC_{\W} : y^2-y =x^3$.
Then
\[ [-2]_{\cC_{\W} } (z) = -2z+9z^4 \sum_{n\geq 0} (-1)^n4^nz^{3n}.\]
For $(z, w(z))$ a coordinate chart at the origin with $w(z) = z^3-w(z)^2$, 
\[w(z) = -D((-z)^3) = \sum_{n\geq 0} (-1)^{n}C_{n}z^{3(n+1)} = \frac{\sqrt{1+4z^3}-1}{2}.\]
Further,
\[\![-1]_{\cC_{\W}}F_{\cC_{\W}}(z_1, z_2) =-z_1-z_2+\frac{(z_1^3+z_2^3)+D(-(z_1^3+z_2^3+4z_1^3z_2^3))}{(z_2-z_1)^2} .\]
\end{prop}
\begin{proof}
It follows from \fullref{intminustwo} that, modulo $u_1$,
\begin{align*}
[-2]_{\cC_{\W} } (z) &= -2z+9z^4\frac{1}{1+4z^3}.
\end{align*}
This proves the first claim. The second claim is equivalent to showing that
$w(z) =z^3C((-z)^3)$.
It follows from (\ref{catnum}) that
$C(z) = 1+zC(z)^2$.
Therefore, 
\[C((-z)^3) = 1 +(-z)^3C((-z)^3)^2,\]
so that
\[z^3C((-z)^3) = z^3 -(z^3C((-z)^3))^2.\]
It also follows from (\ref{eqn:funcw}) that, for the curve $\cC_{\W}$, 
$w(z)=z^3-w(z)^2$.
Since $w(z)$ and $z^3C((-z)^3)$ satisfy the same functional equation, they must be equal. Further, this implies that
\[w(z) = \frac{\sqrt{1+4z^3}-1}{2}.\] 

Finally, note that
\begin{align*}
\lambda(z_1,z_2) &= \frac{1}{z_2-z_1}\left(\frac{\sqrt{1+4z_2^3}-1}{2} - \frac{\sqrt{1+4z_1^3}-1}{2}\right) \\ 
&= \frac{\sqrt{1+4z_2^3}-\sqrt{1+4z_1^3}}{2(z_2-z_1)}.
\end{align*}
Using (\ref{minusell}), it follows that
\begin{align*}[-1]_{F_{\cC_{\W}}}\left(F_{\cC_{\W}}(z_1,z_2)\right) &=-z_1-z_2+\lambda(z_1,z_2)^2 \\
&=-z_1-z_2+\frac{(z_1^3+z_2^3)+D(-(z_1^3+z_2^3+4z_1^3z_2^3))}{(z_2-z_1)^2} . \qedhere
\end{align*}
\end{proof}

\begin{prop}\label{minustwomod2}
Let $\cC$ be defined over $\F_4$ by the Weierstrass equation
$\cC : y^2+y =x^3$.
The local uniformizer at the origin $w(z) = z^3+w(z)^2$, satisfies
$w(z) = \sum_{k\geq 0}z^{3\cdot 2^k}$.
Further, $[-2]_{F_{\cC}}(z) = z^4$
and
\begin{equation*}
[-1]_{F_{\cC}}\left(F_{\cC}(z_1,z_2)\right)  = z_1+z_2 +\sum_{k\geq 1}\sum_{n=0}^{3\cdot2^{k-1}-1 } (z_1^{2(3\cdot 2^{k-1}-1-n)}z_2^{2n} ).
\end{equation*}
Finally, $[-1]_{F_{\cC}}(z)  = \sum_{k\geq 0}z^{3 \cdot 2^k-2}$,
so that
\begin{align*}
F_{\mathcal{C}}(z_1,z_2) &= z_1+z_2+z_1^2 z_2^2+z_1^6 z_2^4+z_1^4 z_2^6+z_1^8 z_2^8+z_1^{12} z_2^4+z_1^4 z_2^{12}+\ldots
\end{align*}
where the next term has order $22$. 
\end{prop}

\begin{proof}
One can compute directly that $w(z) = \sum_{k\geq 0}z^{3\cdot 2^k}$. This implies that $C_n \neq 0$ modulo $2$ if and only if $n+1 = 2^k$. Therefore, we have the following identity of power series
\[D(y) = \sum_{n\geq 0} C_ny^{n+1} =\sum_{k\geq 0} y^{2^k}. \]
Hence, using \fullref{twoseries} modulo $(2)$, we obtain
\begin{align*}[-1]_{F_{\cC}}\left(F_{\cC}(z_1,z_2)\right)  &=z_1+z_2+\frac{(z_1^3+z_2^3)+D(z_1^3+z_2^3)}{z_2^2+z_1^2}\\
&=z_1+z_2+\frac{1}{z_2^2+z_1^2}\left(\sum_{k\geq 1} (z_1^2)^{3\cdot 2^{k-1}}+ (z_2^2)^{3\cdot 2^{k-1}}\right) \\
&= z_1+z_2 +\sum_{k\geq 1}\sum_{n=0}^{3\cdot2^{k-1}-1 } (z_1^{2(3\cdot 2^{k-1}-1-n)}z_2^{2n} ).
\end{align*}
This gives the result for $[-1]_{F_{\cC}}\left(F_{\cC}(z_1,z_2)\right)$. Letting $z_1=z$ and $z_2=0$ gives it for $[-1]_{F_{\cC}}(z)$.
A direct computation gives the estimate for $F_{\cC}(z_1, z_2)$.
\end{proof}

\subsection{The technique for computing the action of $\Sn_{\cC}$}\label{technique}
The method presented here is an adaptation of the techniques used by Henn, Karamanov and Mahowald in \cite{HKM}. Let $\gamma$ be in $\Sn_{\cC}$. Then $\gamma \in \F_4[\![x]\!]$ is a power series which satisfies
\[ \gamma ( F_{\cC}(x,y) )= F_{\cC}(\gamma(x),\gamma(y)).\]
Recall from \fullref{sec:leftact} that $\gamma$ gives rise to isomorphisms $\phi_{\gamma}: (E_{\cC})_* \rightarrow (E_{\cC})_*$ and
$h_{\gamma}: \phi_{\gamma}^*F_{E_{\cC}} \rightarrow F_{E_{\cC}}$,
where 
$h_{\gamma} \in (E_{\cC})_0[\![x]\!]$.
The action of $\gamma$ on $(E_{\cC})_*$ is given precisely by $\phi_{\gamma}$.

The isomorphism $\phi_{\gamma}$ is linear over $\W$; hence it is sufficient to specify $\phi_{\gamma}(u)$ and $\phi_{\gamma}(u_1)$. The morphism $h_{\gamma}$ is a power series
\[h_{\gamma}(x) = t_{0}(\gamma) x + t_{1}(\gamma) x^2+ t_{2}(\gamma)x^3 +\ldots\]
where 
\[t_{i}(\gamma) :  \mathbb{S}_{\cC} \to  (E_{\cC})_0 = \W[\![u_1]\!]\]
are continuous maps. By (\ref{eqn:phiu}) $\phi_{\gamma}(u) = h_{\gamma}'(0) u =t_{0}(\gamma) u$,
which gives the action of $\gamma$ on $u$.

The morphism $h_{\gamma}$ must satisfy
\begin{align}\label{h}
h_{\gamma}([-2]_{\phi_{\gamma}^*F_{E_{\cC}}}(x)) = [-2]_{F_{E_{\cC}}}(h_{\gamma}(x)).
\end{align}
This imposes a set of relations on the parameters $t_{i}(\gamma)$ and $\phi_{\gamma}(u_1)$. 
Further, $h_{\gamma}$ is a lift of $\gamma$, so that $h_{\gamma} \equiv \gamma$ modulo $(2, u_1)$.
This specifies the parameters $t_i(\gamma)$ modulo $(2, u_1)$. With this information, the relations imposed by (\ref{h}) are sufficient to approximate $\phi_{\gamma}$. Before executing this program, we prove a preliminary result.

\begin{prop}\label{Intaction}
If $\gamma \in \Z_2^{\times} \cap \Sn_{\cC}$, so that $\gamma = \sum_{i\geq 0}a_i T^{2i}$,
for $a_i \in \{0,1\}$. Let 
$\ell = \sum_{i\geq 0}a_i (-2)^i $
in $\Z_2^{\times} \subseteq (E_{\cC})_0$. Then $\phi_{\gamma}(u_1) = u_1$ and $\phi_{\gamma}(u) = \ell u$.
\end{prop}
\begin{proof}
The element $\gamma$ is given by
\[\gamma(x) = a_0x +_{F_{\cC}} a_1 [-2]_{F_{\cC}}(x)  +_{F_{\cC}} a_2 [4]_{F_{\cC}}(x) +_{F_{\cC}}  \ldots \]
Let $g$ be the lift for $\gamma$ given by
\[g(x) = a_0x +_{F_{E_{\cC}}} a_1 [-2]_{F_{E_{\cC}}}(x)  +_{F_{E_{\cC}}} a_2 [4]_{F_{E_{\cC}}}(x) +_{F_{E_{\cC}}}\ldots\]
Then $g$ is an automorphism of $F_{E_\cC}$, hence $\phi_{\gamma} :  (E_{\cC})_0 \ra  (E_{\cC})_0 $ is the identity and
$h_{\gamma}(x) = g(x)$.
Since $[n]_{F_{E_{\cC}}}(x) \equiv nx $ modulo $(x^2)$ for $n\in \mathbb{Z}$, $g(x) \equiv \ell x$ modulo $(x^2)$. Therefore, $g'(0) =\ell$ and $\phi_{\gamma}(u) = \ell u$.
\end{proof}

\begin{thm}\label{int} 
Let $\gamma \in \mathbb{S}_{\cC}$ and $t_i =t_i(\gamma)$. Then 
\begin{align*}
\phi_{\gamma}(u_1) &= u_1t_{0}+\frac{2}{3}\frac{t_{1}}{t_0}.
\end{align*}
In particular, $\phi_{\gamma}(u_1) \equiv  t_0 u_1$ and
$\phi_{\gamma}(u) \equiv t_0 u$ modulo $(2)$ and $v_1=u_1u^{-1}$ is fixed by the action of $\Sn_{\cC}$ modulo $(2)$.
\end{thm}
\begin{proof}
Recall from (\ref{h}) that
\[h_{\gamma}([-2]_{\phi_{\gamma}^*F_{E_{\cC}}}(x)) = [-2]_{F_{E_{\cC}}}(h_{\gamma}(x)).\]
Using \fullref{intminustwo}, one obtains the following relation on the coefficients of $x^2$,
\begin{align*}
-9\phi_{\gamma}(u_1)t_{0} +4t_{1} &= -9u_1t_{0}^2-2t_{1}.
\end{align*}
Because $\phi_{\gamma}$ is an isomorphism, $t_{0}$ is invertible. Isolating $\phi_{\gamma}(u_1)$ and dividing both sides by $-9t_{0}$ proves the claim.
\end{proof}

Therefore, to approximate the action of an element $\gamma$ in $\Sn_{\cC}$ on $(E_{\cC})_*$, it suffices to approximate the parameters $t_{0}(\gamma)$ and $t_1(\gamma)$.

\subsection{Approximations for the parameters $t_i(\gamma)$}\label{parameters}
In this section, we use the technique described in \fullref{technique} to approximate the parameters $t_i({\gamma})$.

\begin{cor}\label{cor:eqnmod2u16}
Modulo $(2,u_1^6)$,
\begin{align*}
t_s &\equiv t_s^4+u_1t_{2s+1}^2+ \binom{s+2}{2}t_{0}^2t_{s+1}u_1^2+\sum_{i=0}^{s-1}u_1^2t_i^{4}t_{2s-1-2i}^2  \\
&+\left(\binom{s}{1}t_0^4t_{s-1}  +\binom{s}{2}t_0^4t_{s-1}+ \binom{s+3}{4}t_0^4t_{s+2}  + \binom{s+2}{1}t_{\frac{s-1}{2}}^8\right)u_1^4.
\end{align*}
\end{cor}
\begin{proof}
Let $h_{\gamma}(x) = \sum_{i=0}^{\infty}t_{i}x^{i+1}$. Using  \fullref{mod2u2}, we obtain
\begin{align*}
h_{\gamma}([-2]_{\phi_{\gamma}^*F_{E_{\cC}}}(x)) &=  \sum_{i=0}^{\infty}t_{i}\left( t_0 u_1 x^2+x^4+ \sum_{i=1}^{\infty}(t_0 u_1 )^{2i}x^{4+2i} \right)^{i+1} \\
&\equiv  \sum_{i=0}^{\infty}t_{i}\left( t_0 u_1 x^2+x^4+t_0^2u_1^2x^6+ t_0^4u_1^4x^8 \right)^{i+1}  \\
&\equiv  \sum_{i=0}^{\infty}t_{i}\bigg(x^{4(i+1)}+ \binom{i+1}{1}(t_0 u_1 x^{4i+2}+t_0^2u_1^2x^{4i+6} +t_0^4u_1^4x^{4i+8})  \\
&+  \binom{i+1}{2}( t_0^2 u_1^2x^{4i} +t_0^4u_1^4x^{4i+8})\\
& +  \binom{i+1}{3}( t_0^3 u_1^3x^{4 i-2}+ t_0^4 u_1^4x^{4 i+2}+t_0^5 u_1^5x^{4 i+6} )\\
&+  \binom{i+1}{4} t_0^4u_1^4x^{4i-4} + \binom{i+1}{5} t_0^5u_1^5x^{4i-6} \bigg).
\end{align*}
Further,
\begin{align*}
[-2]_{F_{E_{\cC}}}(h_{\gamma}(x)) &= u_1\left( \sum_{i=0}^{\infty}t_{i}x^{i+1}\right)^2    +\left( \sum_{i=0}^{\infty}t_{i}x^{i+1}\right)^4   + \sum_{k=1}^{\infty}u_1^{2k}\left( \sum_{i=0}^{\infty}t_{i}x^{i+1}\right)^{2k+4}  \\
&\equiv  \sum_{i=0}^{\infty}\left(u_1 t_{i}^2x^{2(i+1)} +t_{i}^4x^{4(i+1)} +u_1^{4}t_{i}^8x^{8(i+1)}\right)  +u_1^{2}\left( \sum_{i=0}^{\infty}t_{i}^2x^{2(i+1)}\right)^{3} 
\end{align*}
Next, note that
$\left(\sum_{i\geq 0}a_ix^i\right)^3 \equiv \sum_{k \geq 0}\sum_{2i+j=k}a_i^2a_jx^k$.
Therefore,
\begin{align*}
u_1^2\left( \sum_{i=0}^{\infty}t_{i}^2x^{2(i+1)}\right)^{3}   & \equiv  \sum_{k \geq 0}\sum_{2i+j={k}}u_1^2t_i^{4}t_j^2x^{2k+6}.
\end{align*}
Now, using (\ref{h}), the coefficient of $x^{4(s+1)}$ gives the relation
\begin{align*}
t_s &\equiv t_s^4+u_1t_{2s+1}^2+ \binom{s+2}{2}t_{0}^2t_{s+1}u_1^2+\sum_{2i+j={2s-1}}u_1^2t_i^{4}t_j^2 \\
& \ \ \ \ \ +\left(\binom{s}{1}  +\binom{s}{2} \right)t_0^4t_{s-1}u_1^4 + \binom{s+3}{4}t_0^4t_{s+2}u_1^4  + \binom{s+2}{1}t_{\frac{s-1}{2}}^8u_1^4
\end{align*}
(Note that the the coefficient of the last term is chosen to be zero when $s$ is even, so that when $t_{\frac{s-1}{2}}$ has a non-zero coefficient, $(s-1)/2$ is an integer.)
\end{proof}

\begin{prop}\label{prop:mod(2u1)2}
For $t_i=t_i(\gamma)$ where $\gamma \in \Sn_{\cC}$, then
\[t_i \equiv t_i^4 + u_1t_{2i+1}^2 + 2t_{4i+3}+2\sum_{\substack{r+s=2i \\  0\leq r< s}}t_r^2t_s^2  \mod (2, u_1)^2.\]
\end{prop}

\begin{proof}
Modulo $(4, u_1)$, we have $[-2]_{F_{\cC_U}}(x) \equiv 2x+x^4$. 
This gives
\begin{align*} 
h_{\gamma}([-2]_{\phi_{\gamma}^*F_{E_{\cC}}}(x))  \equiv \sum_{i=0}^{\infty} t_i\left(x^{4(i+1)}+2\binom{i+1}{1}x^{4i+1} \right) 
\end{align*}
and
\begin{align*} 
[-2]_{F_{E_{\cC}}}(h_{\gamma}(x)) &\equiv \sum_{i=0}^{\infty}2t_ix^{i+1} + \left( \sum_{i=0}t_ix^{i+1}\right)^4\\
&\equiv \sum_{i=0}^{\infty}2t_ix^{i+1}     + \sum_{i=0}t_{i}^4x^{4(i+1)}   + \sum_{i=1}^{\infty}x^{4+2i}2 \sum_{\substack{r+s=i \\  0\leq r< s}}t_r^2t_s^2.
\end{align*}
Using (\ref{h}), the coefficient of $x^{4(i+1)}$ gives the relation
\begin{align*}
t_{i} &=2t_{4i+3}+t_i^4+2\sum_{\substack{r+s=2i \\  0\leq r< s}}t_r^2t_s^2 \mod (4,u_1).
\end{align*}
The claim then follows from \fullref{cor:eqnmod2u16}.
\end{proof}
\begin{prop}\label{prop:t0t1eqn}
Modulo $(4)$
\begin{align*}
t_0&\equiv t_0^4+2 t_3+3 t_1^2 u_1+2 t_0 t_2 u_1+3t_0^2 t_1 u_1^2.
\end{align*}
Modulo $(2)$,
\begin{align*}
t_1 &\equiv t_1^4+t_3^2 u_1+t_0^4 t_1^2 u_1^2+t_0^2 t_2 u_1^2+t_0^5 u_1^4+t_0^8 u_1^4+t_0^4 t_3 u_1^4.
\end{align*}
\end{prop}
\begin{proof}
Modulo $(4)$, the coefficient of $x^4$ in $h_{\gamma}([-2]_{\phi_{\gamma}^*F_{E_{\cC}}}(x))$ is
$t_0+\phi_{\gamma}(u_1)^2 t_1 $
and the coefficient of $[-2]_{F_{E_{\cC}}}(h_{\gamma}(x))$ is given by
\begin{align*}
 t_0^4+2 t_3+3t_1^2 u_1+2 t_0 t_2 u_1 
\end{align*}
Recall from \fullref{int} that $\phi_{\gamma}(u_1) = u_1t_{0}+\frac{2}{3}\frac{t_{1}}{t_0}$. This and (\ref{h}) imply that
\begin{align*}
  t_0+t_0^2t_1u_1^2 \equiv  t_0^4+2 t_3+3t_1^2 u_1+2 t_0 t_2 u_1 .
\end{align*}
Isolating $t_0$ proves the first claim. A similar argument using the coefficients of $x^8$ give the desired relation for $t_1$.
\end{proof}

Recall that $\gamma \in \Sn_{\cC}$ has an expansion of the form
\[\gamma = a_0+a_1T+a_2T^2+a_3T^3 +\ldots\]
Here the $a_i$ are solutions to the equation $x^4-x=0$. Recall from \fullref{sec:TvS} that if $\omega^s \in \End(F_{\cC})$ is a solution to the equation $x^4-x=0$, then it corresponds the automorphism
\[\omega^s(x) = \zeta^s x,\]
where $\zeta \in \F_4=(E_{\cC})_*/(2, u_1)$. There is a copy of $\F_4$ in $\End(F_{\mathcal{C}})$ given by the ring generated by the automorphism $\omega(x)$. Further, $(E_{\cC})_*/(2, u_1)$ is isomorphic to $\F_4$, with generator the image of $\zeta$. Define a map
\[f: \F_4 \subseteq  \End(F_{\mathcal{C}}) \ra (E_{\cC})_*/(2, u_1)\cong \F_4\] 
by $f(\omega^s(x)) = \zeta^s$.
If $\gamma$ is as above, using the fact that $T(x) = x^2$, 
\[\gamma(x) = f(a_0)x +_{F_{\mathcal{C}}} f(a_1)x^2  +_{F_{\mathcal{C}}} f(a_2)x^4 +_{F_{\mathcal{C}}} f(a_3)x^8+ \ldots\]
For simplicity, we will identify $a_i$ with $f(a_i)$ and write
\begin{align}\label{gequation}
\gamma(x) = a_0x +_{F_{\mathcal{C}}} a_1x^2  +_{F_{\mathcal{C}}} a_2x^4 +_{F_{\mathcal{C}}} a_3x^8 +\ldots
\end{align}

\begin{prop}\label{prop:seriesandmod2u12}
For $\gamma \in S_{\cC}$, modulo $(x^{18})$,
\begin{align*}
\gamma(x)&=x+a_1x^2+a_2x^4  +a_1^2x^6+a_3x^8   +a_2^2x^{10} + a_1^2 a_2^2x^{12}+a_1x^{14}  +(a_1^3+a_4)x^{16} .
\end{align*}
\end{prop}
\begin{proof}
This is a direct computation using (\ref{gequation}) and the formal group law of \fullref{minustwomod2}, noting that for $\gamma \in S_{\cC}$, $a_0 = 1$.
\end{proof}

\begin{cor}\label{cor:mod2u1}
Let $t_i=t_i(\gamma)$ where $\gamma \in S_{\cC}$. Modulo $(2, u_1)$, $t_0 \equiv 1$, $t_{2i} \equiv 0$ for $i \neq 0$ and
\begin{align*}
t_1 &\equiv a_1 &    t_5 &\equiv a_1^2 & t_9 &\equiv a_2^2 &  t_{13} &\equiv a_1   \\
 t_3 &\equiv a_2 &    t_7 &\equiv a_3   &   t_{11} &\equiv a_1^2 a_2^2    &  t_{15}&\equiv  a_1^3+ a_4.
\end{align*}
\end{cor}
\begin{proof}
This follows from \fullref{prop:seriesandmod2u12}, noting that $t_i $ is congruent to the coefficient of $x^{i+1}$ modulo $(2,u_1)$.
\end{proof}

\begin{prop}\label{prop:mod2u12}
Let $t_i = t_i(\gamma)$ for $\gamma \in S_{\cC}$. Then modulo $(2, u_1^2)  $,
\begin{align*}
t_0 & \equiv 1 + a_1^2u_1        &			
t_1  & \equiv a_1+ a_2^2u_1   	 &		
t_2  & \equiv a_1 u_1                	\\		
t_3 &\equiv a_2 + a_3^2 u_1 	 &		  
t_4 &\equiv a_2u_1                  	&		  
t_5 & \equiv a_1^2 +  a_1 a_2 u_1 \\			
t_6&\equiv a_1^2u_1	&				
t_7 &\equiv a_3+(a_1^3+ a_4^2)u_1  &		
\end{align*}
\end{prop}
\begin{proof}
This follows from \fullref{cor:mod2u1} and \fullref{cor:eqnmod2u16}.
\end{proof}

\begin{cor}\label{cor:S21modu14}
Let $t_i=t_i(\gamma)$ where $\gamma \in S_{\cC}$. Then 
\begin{align*}
t_0 &\equiv 1+a_1^2u_1+a_1u_1^2+ (a_2+a_2^2)u_1^3 \mod (2, u_1^4).
\end{align*}
\end{cor}
\begin{proof}
This follows from \fullref{prop:mod2u12} and \fullref{prop:t0t1eqn}.
\end{proof}
We will need better estimates for elements which are in $F_{2/2}\Sn_{\cC}$. Therefore, for the remainder of this section, we will always assume that $\gamma \in F_{2/2}\Sn_{\cC}$.
\begin{prop}\label{prop:modu13}
Let $t_i=t_i(\gamma)$ where $\gamma \in F_{2/2}\Sn_{\cC}$. Modulo $(2, u_1^4)$,
\begin{align*}
t_3 &\equiv  a_2+a_3^2u_1+a_4 u_1^3.
\end{align*}
Modulo $(2,u_1^3)$, $t_1 \equiv a_2^2u_1$ and $t_5 \equiv (a_2 +a_2^3) u_1^2$.
Modulo $(2, u_1^6)$,
\begin{align*}
t_2 &\equiv a_2^2 u_1^2+a_3 u_1^4+(a_2 +a_2^3) u_1^5.
\end{align*}
\end{prop}
\begin{proof}
It follows from \fullref{cor:eqnmod2u16} that, modulo $(2,u_1^4)$,
\begin{align*}
t_3 &\equiv t_3^4+t_7^2 u_1+t_1^2 t_2^4 u_1^2+t_1^4 t_3^2 u_1^2+t_0^4 t_5^2 u_1^2 \\
t_5 &\equiv t_5^4+t_{11}^2 u_1+t_3^6 u_1^2+t_1^2 t_4^4 u_1^2+t_2^4 t_5^2 u_1^2+t_0^2 t_6 u_1^2+t_1^4 t_7^2 u_1^2+t_0^4 t_9^2 u_1^2.
\end{align*}
The results for $t_3$ and $t_5$ then follow from \fullref{cor:mod2u1} and \fullref{prop:mod2u12}. It also follows from \fullref{cor:eqnmod2u16} that, modulo $(2, u_1^6)$,
\begin{align*}
t_2 &\equiv t_2^4+t_5^2 u_1+t_1^6 u_1^2+t_0^4 t_3^2 u_1^2+t_0^4 t_1 u_1^4+t_0^4 t_4 u_1^4.
\end{align*}
The identity for $t_2$ then follows from the \fullref{cor:mod2u1} and \fullref{prop:mod2u12}, using the identity for $t_5$ modulo $(2,u_1^3)$.
\end{proof}

\begin{prop}\label{t0t1again}
Let $\gamma \in F_{2/2}\Sn_{\cC}$. Modulo $(2, u_1^8)$, 
\begin{align*}
t_1(\gamma) &\equiv a_2^2 u_1+a_3 u_1^3+a_3^2 u_1^5+a_3 u_1^6+(a_2^2 +a_2^3 +a_4 +a_4^2) u_1^7.
\end{align*}
Modulo $(2, u_1^{10})$, 
\begin{align*}
t_0(\gamma) &\equiv  1+(a_2 +a_2^2) u_1^3+a_3 u_1^5+a_3 u_1^8+(a_2 +a_2^2 +a_4+a_4^2) u_1^9.
\end{align*}
\end{prop}
\begin{proof}
Apply \fullref{prop:t0t1eqn}, \fullref{cor:S21modu14} and \fullref{prop:modu13} for the estimate for $t_1$. The result for $t_0$ then follows from \fullref{prop:t0t1eqn}.
\end{proof}

\begin{prop}\label{t0gammafinally}
Let $\gamma \in F_{2/2}\Sn_{\cC}$. Modulo $(4, 2u_1^2,u_1^{10})$, 
\begin{align*}
 t_{0}(\gamma) &\equiv 1+2a_2+2a_3^2u_1+(a_2 + a_2^2)u_1^3 +a_3u_1^5 +a_3 u_1^8+(a_2 +a_2^2 +a_4+a_4^2) u_1^9.
\end{align*}
\end{prop}
\begin{proof}
Apply \fullref{prop:t0t1eqn}, \fullref{prop:mod2u12} and \fullref{t0t1again}.
\end{proof}



\appendix

\section{The cohomology of $G_{24}$ (by Hans-Werner Henn)}\label{sec:appcoh}
  \setcounter{Athm}{0}
\makeatletter
\let\c@equation=\c@Athm
\makeatother
\renewcommand{\theequation}{\Alph{section}.\arabic{equation}}  
\noindent
The following consists of unpublished notes by Hans-Werner Henn, which were edited by the author. She thanks him for letting her include them here.

Let $\cC$ be the supersingular elliptic curve over $\FF_4$ with equation $y^2+y=x^3$.
In this appendix, we calculate the cohomology of the automorphism group of this elliptic curve 
with coefficients in the Lubin--Tate module $(E_{\cC})_*V(0)$ (see \fullref{sec:Etheory} for definitions).

None of the results are original. In some sense, this appendix redoes calculations by other people, for 
example by Bauer \cite[Section 7]{tbauer} and by Rezk \cite[Section 18]{rezksupp}. The basic ideas go back to Hopkins. 
Bauer and Rezk calculate the cohomology of the Weierstrass Hopf algebroid 
and the calculation here can, in principle, be deduced from their calculation 
by inverting the discriminant $\Delta$ and passing to a suitable completion. 

Our purpose is to give an independent and self contained calculation 
of the group cohomology including the complete multiplicative structure. 
Furthermore, all elements in cohomology are defined via ``Greek letter  constructions" 
avoiding any explicit cocycles or Massey products. 
Everything is 
deduced from the know{\-}ledge of $H^*(G_{24},(E_{\cC})_*/(2,u_1)) \cong H^*(G_{24},\FF_4[u^{\pm 1}])$, from the structure of the $G_{24}$--invariants 
of the symmetric algebra of a certain two dimensional representation $\rho$ of $G_{24}$, and 
from the structure of $v_1^{-1}H^*(G_{24},S_*(\rho)/(\Delta S_*(\rho)))$, 
where the discriminant $\Delta$ is an invariant class 
in degree $24$. This knowledge is established in \fullref{constant}, 
\fullref{Ginvariants} and \fullref{invariants} below. The main computation is that of $H^*(G_{24},S_*(\rho))$ and is given in \fullref{mainthm}. The computation of $H^*(G_{24}, (E_{\cC})_*V(0))$ then follows and is recorded in \fullref{thm:cohG24}.

\begin{Alem}\label{constant}
(a) There are classes $z\in H^4(Q_8,\FF_4)$, $\widetilde{x}\in H^1(Q_8,\FF_4)$ and
$\widetilde{y}\in H^1(Q_8,\FF_4)$ and an isomorphism of graded algebras with $C_3$--linear algebra action 
$$
H^*(Q_8,\F_4)\cong 
\F_4[\widetilde{x},\widetilde{y},z]/(\widetilde{x}\widetilde{y},\widetilde{x}^3+\widetilde{y}^3)  , 
$$ 
where $C_3$ acts on $Q_8$ via the conjugation action of $G_{24}$ on its normal subgroup $Q_8$, and 
$C_3$ acts on the right hand side by $\omega_*(\widetilde{x})=\zeta \widetilde{x}$, $\omega_*(\widetilde{y})=\zeta^2 \widetilde{y}$
and $\omega_*(z)=z$.

\noindent
(b) There are classes $k\in H^4(G_{24},\FF_4[u^{\pm1}]_0)$, $a\in H^1(G_{24},\FF_4[u^{\pm1}]_2)$ and
$b\in H^1(G_{24},\FF_4[u^{\pm1}]_4)$ and an isomorphism of graded algebras  
\[H^*(G_{24},\FF_4[u^{\pm1}]) \cong \F_4[v_2^{\pm 1},k,a,b]/(ab,b^3-v_2a^3),\]
where $v_2=u^{-3}$.
\noindent
(c) The subalgebra $H^*(G_{24},\FF_4[u^{-1}])\subseteq H^*(G_{24},\FF_4[u^{\pm 1}])$ is generated as  
an $\FF_4$--algebra by the elements $v_2$, $k$, $a$, $b$, $v_2^{-1}b^2$ and $v_2^{-1}a^3$. 
\end{Alem} 

\begin{proof} (a) We start from the well known result  
that there is an isomorphism of graded algebras 
$$
H^*(Q_8,\F_2)\cong \F_2[x,y,z]/(x^2+xy+y^2,x^2y+xy^2)
$$
where $x$ and $y$ are in dimension $1$ and $z$ is in dimension $4$. 
The action of $G_{24}/Q_8$ on this algebra is trivial on $z$ and the unique nontrivial action 
on $H^1\cong (\Z/2)^2$. With suitable choices of $x$ and $y$, we get $\omega_*x=y$, 
$\omega_*y=x+y$. This action has no eigenvectors over $\F_2$, but it has two eigenvectors 
over $\FF_4$ with Galois conjugate eigenvalues equal to $\zeta$ and $\zeta^2$. 
If we define $\widetilde{x}=x+\zeta y$ and $\widetilde{y}=x+\zeta^2y$, 
then $\omega_*\widetilde{x}=\zeta \widetilde{x}$ and $\omega_*\widetilde{y}=\zeta^2 \widetilde{y}$. 
Furthermore,
\[\widetilde{x}\widetilde{y}=x^2+xy+y^2=0\]
and 
\[\widetilde{x}^3=x^3+\zeta x^2y+\zeta^2 xy^2+y^3=x^3+\zeta^2x^2y+\zeta xy^2+y^3=\widetilde{y}^3 . \]
The claim now follows. 

\noindent
(b) Let $a=u^{-1}\widetilde{x}$, $b=u^{-2}\widetilde{y}$ and $k=z$. 
Then we get an isomorphism of bigraded algebras 
$$
\F_4[v_2^{\pm 1},k,a,b]/(v_2a^3-b^3,ab)\cong H^*(Q_8,\FF_4[u^{\pm 1}])^{C_3}.
$$

\noindent 
(c) This follows from the fact that $H^2(Q_8,\FF_4[u^{-1}]_2)^{C_3}\cong \FF_4$ 
is generated by $v_2^{-1}b^2$ and $H^3(Q_8,\FF_4[u^{-1}]_0)^{C_3}\cong \FF_4$ 
is generated by $v_2^{-1}a^3$. 
\end{proof}

Consider $u^{-1}$ and $v_1=u_1u^{-1}$ as elements in $(E_{\cC})_2$. By \fullref{subsec:AutC}, the action of $G_{24}$ is given by
\begin{align*}
\omega_*(v_1) &= v_1 & \omega_*(u^{-1})&= \zeta^2u^{-1} \\
i_*(v_1)&= \frac{v_1+2u^{-1}}{\zeta^2-\zeta}  &   i_*(u^{-1}) &= \frac{v_1-u^{-1}}{\zeta^2-\zeta} \\
j_*(v_1)&= \frac{v_1+2\zeta^2 u^{-1}}{\zeta^2-\zeta} &  j_*(u^{-1}) &= \frac{\zeta v_1-u^{-1}}{\zeta^2-\zeta} \\
k_*(v_1)&= \frac{v_1+2\zeta u^{-1}}{\zeta^2-\zeta} &   k_*(u^{-1}) &= \frac{\zeta^2 v_1-u^{-1}}{\zeta^2-\zeta}.
\end{align*}
\begin{Arem}\label{Arem:rho}
The two dimensional $\W$--module generated by $u^{-1}$ and $v_1$ is a 
representation of $G_{24}$ which we denote by $\varrho$. 
We denote its mod--$2$ reduction by $\rho$ and the respective graded symmetric algebras by $S_*(\varrho)$ and $S_*(\rho)$. 
Because $i^2_*(u)=-u$ and $i^2_*(u_1)=u_1$, we see that for each integer $n\geq 0$ 
the action of $G_{24}$ on $S_{2n}(\varrho)$ factors through the quotient $A_4  = G_{24}/(\pm 1) $. 
Likewise, the action of $G_{24}$ on all of $S_*(\rho)$ factors through an action of $A_4$. 

It follows that the element 
\[\widetilde{\Delta}:=\prod_{g\in V_4}g_*(u^{-1})\]  
in $S_*(\rho)$ is a $Q_8$--invariant. One computes that
$\widetilde{\Delta} \equiv u^{-1}(u^{-3}+v_1^3)$ modulo $(2)$. It is an eigenvector for the residual action of $G_{24}/Q_8$; in fact,  
\begin{equation}\label{omega-Delta}
\omega_*(\widetilde{\Delta})=\zeta^2\widetilde{\Delta}.
\end{equation} 
Hence, $\widetilde{\Delta}^3$ is a $G_{24}$--invariant and equal to the mod--$2$ reduction of the discriminant $\Delta$ (see \fullref{subsec:E1}).
\end{Arem}

The proof of the following theorem is similar to that of Goerss, Henn and Mahowald \cite[Proposition 2]{ghmv1}.
\begin{Athm}\label{completion}
Completion at the maximal ideal $I\subseteq S_*(\rho)[\widetilde{\Delta}^{-1}]$ induces a continuous isomorphism of $\F_4[G_{24}]$--algebras. 
$$
S_*(\rho)[\widetilde{\Delta}^{-1}]^{\ \widehat{ }}_{I}\to (E_{\cC})_* V(0) .
$$
\end{Athm}

Therefore to compute the cohomology $H^*(G_{24},(E_2)_*V(0))$, we start by analyzing $S_*(\rho)$ 
and its invariants. Let $n\geq 0$ be an integer and
\[\widetilde{S}_{n}(\rho) = \begin{cases} {S}_{n}(\rho) & 0\leq n\leq 3\\
 S_{n}(\rho)/(\widetilde{\Delta}S_{n-4}(\rho)) &   4\leq n .\end{cases}\]

\begin{Alem}\label{Sn} 
Multiplication with 
$\widetilde{\Delta}$ induces a split short exact sequence of $\F_4[Q_8]$--modules 
$$
0\to \Sigma^8S_n(\rho)\xra{\widetilde{\Delta} } S_{n+4}(\rho)\to \widetilde{S}_{n+4}(\rho)\to 0 \ . 
$$ 
Further, for $n\geq 3$, multiplication by $v_1$ induces isomorphisms of $\F_4[Q_{8}]$--modules
\[ v_1: \Sigma^2\widetilde{S}_{n}(\rho) \xra{\cong} \widetilde{S}_{n+1}(\rho).\]
\end{Alem}

\begin{proof} 
The $n+1$ elements 
$\widetilde{\Delta}u^{-k}v_1^{n+1-k}$, $k=0,\ldots n$, together with the 
four elements $u^{-l}v_1^{n+4-l}$, $l=0,\ldots, 3$ form a basis of the $\F_4$--vector space $S_{n+4}(\rho)$. Therefore, the sequence is exact as a sequence of 
$\F_4[Q_8]$--modules. Furthermore, the subspace generated by 
$u^{-l}v_1^{n+4-l}$ with $0\leq l\leq 3$ is $Q_8$--invariant. This gives the splitting. 

If $n\geq 3$, the image of the elements $u^{-l}v_1^{n-l}$ for $0\leq l\leq 3$ form a basis for $\widetilde{S}_{n}(\rho)$. Multiplication by $v_1$ sends this basis of $\widetilde{S}_{n}(\rho)$ to the corresponding basis of $\widetilde{S}_{n+1}(\rho)$. 
\end{proof} 

The next step is to identify the invariants.  
\begin{Aprop}\label{Qinvariants} The $Q_8$--invariants of $S_*(\rho)$ are given as 
$\FF_4[v_1,\widetilde{\Delta}]$. 
\end{Aprop}

\begin{proof} The split short exact exact sequence of \fullref{Sn} 
gives a short exact sequence of $Q_8$--invariants. Consider the following
commutative diagram 
$$
\xymatrix{
0\ar[r] & \Sigma^8\F_4[v_1,\widetilde{\Delta}] \ar[r]^{\widetilde{\Delta}} \ar[d] & \F_4[v_1,\widetilde{\Delta}]
\ar[d] \ar[r] & \F_4[v_1]\ar[r] \ar[d]\ & 0 \\  
0\ar[r] & \Sigma^8S_*(\rho)^{Q_8}\ar[r]^{\widetilde{\Delta}}  & S_*(\rho)^{Q_8}
\ar[r] & (\widetilde{S}_*)^{Q_8}\ar[r] & 0 . &   \\ }
$$  
From part (c) of \fullref{Sn} and an easy calculation, we get
$(\widetilde{S}_n)^{Q_8}\cong \F_4\{v_1^n\}$ for each $n\geq 0$. 
The claim now follows from an induction over the internal degree 
and the five lemma.  
\end{proof}

\begin{Acor}\label{Ginvariants} The $G_{24}$--invariants of $S_*(\rho)$ are given as 
$\FF_4[v_1,\Delta]$.      
\end{Acor}

We turn towards analyzing the cohomology algebra $H^*(G_{24},(E_{\cC})_*V(0))$. 
We begin by introducing certain classes in $H^1(G_{24},(E_{\cC})_*V(0))$. For this we consider the  
exact sequence $G_{24}$--modules 
\begin{equation}\label{v0-exact-seq}
0\to \rho\buildrel{ 2}\over\longrightarrow \varrho/(4) \to  \rho \to 0 
\end{equation}
with associated Bockstein $\delta$.

\begin{Alem}\label{eta} (a) The class $v_1$ in $\rho$ is $G_{24}$--invariant. The class 
$\eta:=\delta(v_1)$ in $H^1(G_{24},\rho)$ is nontrivial. 

\noindent 
(b) $\eta$ is $v_1$--torsion free in $H^*(G_{24},S_*(\rho))$. 

\noindent 
(c) $\eta$ is not divisible by $v_1$. 
\end{Alem} 
\begin{proof} (a) A direct computation shows that $v_1$ is invariant modulo $(2)$, but not invariant modulo $(4)$.
Hence, $\eta$ is non-trivial. 
 
\noindent 
(b) More generally, $v_1^{2k}$ is invariant modulo $(4)$ while $v_1^{2k+1}$ is not. 
This shows that $\delta(v_1^{2k+1})=v_1^{2k}\eta$ is non-trivial.   

\noindent 
(c) If $\eta$ is divisible, then there is a class $\eta'\in H^1(G_{24},S_0(\rho))$ such 
that $v_1\eta'=\eta$. However, $S_0(\rho)=\FF_4$ and $H^1(G_{24};S_0(\rho))=0$ 
by \fullref{constant}. 
\end{proof}

Consider the graded $\F_4[v_1][G_{24}]$--algebras 
\[\widetilde{S}_*(\rho):=S_*(\rho)/(\widetilde{\Delta}S_*(\rho)) \]
and
\[\overline{S}_*(\rho):=S_*(\rho)/(\Delta S_*(\rho)).\] 
By part (a) of \fullref{Sn}, multiplication with $v_1$ determines an isomorphism 
$\widetilde{S}_n(\rho)\to \widetilde{S}_{n+1}(\rho)$ if $n\geq 3$. Therefore, $\widetilde{S}_*(\rho)[v_1^{-1}]$ is a graded $\F_4[v_1^{\pm 1}][G_{24}]$--algebra with 
\[(\widetilde{S}_*(\rho)[v_1^{\pm 1}])_{2k}\cong v_1^{k-3}S_3(\rho)\] 
for every integer $k$. 

\smallskip 

\begin{Alem}\label{invariants} (a) There is an isomorphism of $G_{24}$--modules 
\[S_3(\rho) \cong \FF_4[G_{24}]\otimes_{\FF_4[C_6]}\FF_4.\] 

\noindent
(b) The $Q_8$--cohomology of $\widetilde{S}_*(\rho)[v_1^{\pm 1}]$ is given by 
\[H^*(Q_8,\widetilde{S}_*(\rho)[v_1^{\pm 1}])\cong \FF_4[v_1^{\pm 1},\eta].\]  
It is $C_3$ invariant so that 
$H^*(G_{24},\widetilde{S}_*(\rho)[v_1^{\pm 1}])\cong \FF_4[v_1^{\pm 1},\eta]$. 

\noindent
(c) The $G_{24}$--cohomology of $\overline{S}_*(\rho)[v_1^{\pm 1}]$ is given by 
\[H^*(G_{24},\overline{S}_*(\rho)[v_1^{\pm 1}])\cong \FF_4[v_1^{\pm 1},\eta].\] 
\end{Alem}

\begin{proof} (a) The $C_6$--linear map which sends $1$ to $u^{-3}$ extends to a $G_{24}$--linear isomorphism 
$\FF_4[G_{24}]\otimes_{\FF_4[C_6]}\FF_4\to S_3(\rho)$. 
 
\noindent 
(b) The isomorphism of (a) restricts to an isomorphism of 
$Q_8$--modules
\[\FF_4[Q_8]\otimes_{\FF_4[C_2]}\FF_4\cong S_3(\rho).\] The isomorphisms $(\widetilde{S}_*(\rho)[v_1^{\pm 1}])_{2k}\cong v_1^{k-3}S_3(\rho)$ and part (a) 
show that there is a class $h'\in H^1(Q_{8},S_3(\rho)) \cong H^1(C_2,\FF_4)\cong \FF_4$
and, for $h=v_1^{-3}h'$, an isomorphism of graded algebras $H^*(Q_{8},\widetilde{S}_*(\rho)[v_1^{\pm 1}])\cong \FF_4[v_1^{\pm 1},h]$. By \fullref{eta} 
we know that $\eta$ is $v_1$--torsion free in 
\[H^*(G_{24},\widetilde{S}_*(\rho)) \cong H^*(Q_8,\widetilde{S}_*(\rho))^{C_3}, \] 
hence it is also $v_1$--torsion free in
$H^*(Q_8,\widetilde{S}_*(\rho))$. Therefore, $v_1^2\eta$ must agree 
with $h'$ up to a scalar. This gives the isomorphism 
$H^*(Q_8,\widetilde{S}_*(\rho)[v_1^{\pm 1}])\cong \FF_4[v_1^{\pm 1},\eta]$. 
The invariance with respect to the residual action of $G_{24}/Q_8$ follows from the fact that
both $v_1$ and $\eta$ are invariant. 

\noindent
(c) By iterated use of \fullref{Sn}, we obtain an isomorphism 
$$
S_*(\rho)\cong \widetilde{S}_*(\rho)\oplus \widetilde{\Delta}\widetilde{S}_*(\rho)
\oplus \widetilde{\Delta}^2\widetilde{S}_*(\rho)\oplus \widetilde{\Delta}^3\widetilde{S}_*(\rho) 
$$ 
of $G_{24}$-modules and therefore an isomorphism 
$$
\overline{S}_*(\rho)\cong \widetilde{S}_*(\rho)\oplus \widetilde{\Delta}\widetilde{S}_*(\rho)
\oplus \widetilde{\Delta}^2\widetilde{S}_*(\rho). 
$$ 
Part (b) implies that
$$
H^*(Q_8,\overline{S}_*(\rho)[v_1^{\pm 1}])\cong 
\FF_4[v_1^{\pm 1},\eta]\oplus \widetilde{\Delta}\FF_4[v_1^{\pm 1},\eta]
\oplus\widetilde{\Delta}^2\FF_4[v_1^{\pm 1},\eta].
$$ 
The result follows by observing that, for the residual action of $C_3\cong G_{24}/Q_8$, 
the first summand is invariant while the other two summands 
are eigenspaces for the eigenvalues $\zeta^2$ and $\zeta$ respectively (see (\ref{omega-Delta})). 
\end{proof}

Next we observe that multiplication by $v_1^k$ determines exact sequences of graded 
$G_{24}$--modules 
\begin{equation}\label{v1-exact-seq}
0\to \Sigma^{2k}S_n(\rho)\buildrel{v_1^k}\over\longrightarrow S_{n+k}(\rho)
\buildrel{p}\over\longrightarrow (\FF_4[u^{-1},v_1]/(v_1^{k}))_{2(n+k)}\to 0 
\end{equation}
with associated Bockstein $\delta_k$. In the remainder of this chapter, the Bockstein associated 
with the exact sequence (\ref{v0-exact-seq}) will not play any role, and therefore 
we take the liberty to simply write $\delta$ instead of $\delta_1$.

\begin{Arem}\label{rem:xynu}
The classes $v_2$, $v_2^2$ and $v_2^3$ are invariant in $H^0(G_{24},\F_4[u^{\pm1}])$ but do not lift to invariants in $H^0(G_{24}, S_*(\rho))$. Further, $v_2^2$ is invariant in $\F_4[u^{-1}, v_1]/(v_1^2)$ but is not invariant in $\F_4[u^{-1}, v_1]/(v_1^3)$. Therefore, $\delta(v_2)$, $\delta(v_2^2)$, $\delta(v_2^3)$ and $\delta_2(v_2^2)$ are non-trivial. Define
\begin{align*}
\nu &:=\delta(v_2) \in H^1(G_{24}, S_2(\rho)),\\
 y &:=\delta(v_2^3) \in H^1(G_{24}, S_8(\rho)), \\ 
 x &:= \delta_2(v_2^2) \in H^1(G_{24}, S_4(\rho)).
 \end{align*}
 \end{Arem}
Given a ring $R$ and a set $X$, we will use the notation $RX$ to denote the free $R$--module generated by the elements of $X$.
\begin{Alem}\label{H1}
\begin{enumerate}[(a)]
\item There are relations $v_1\nu=0$, $v_1y=0$ and $v_1^2x=0$.
\item There is a relation $v_1x= \delta(v_2^2)$.
\item The classes $\nu$, $x$ and $y$ are not divisible by $v_1$.
\end{enumerate}
\end{Alem}
\begin{proof} 
(a) This follows from the definition of these classes and exactness of the long exact cohomology 
sequences. 

\noindent 
(b) This is straightforward by comparing the two short exact sequences (\ref{v1-exact-seq}) 
for $k=1$ and $k=2$. 

\noindent 
(c) One verifies by a direct computation that
\[H^0(G_{24},\F_4[u^{-1},v_1]/(v_1^2))\cong \F_4[v_2^2,v_1]/(v_1^2)\{1\} 
\oplus\F_4[v_2^2,v_1]/(v_1)\{v_1v_2\} \]
and that
\begin{align*}
H^0(G_{24},\F_4[u^{-1},v_1]/(v_1^3))&\cong  \F_4[v_2^4,v_1]/(v_1^3)\{1\} 
\oplus (\F_4[v_2^4,v_1]/(v_1^2))\{v_1v_2^2\} \\ 
  &  \ \ \  \oplus (\F_4[v_2^4,v_1]/(v_1))\{v_1^2v_2,v_1^2v_2^3\}. 
\end{align*}
If $\nu$ is $v_1$--divisible, then there is a non-trivial class in $H^1(G_{24},S_1(\rho))$ which is 
annihilated by $v_1^2$ and, hence, is in the image of $\delta_2$. 
However, this contradicts the triviality of the group   
$H^0(G_{24},\F_4[u^{-1},v_1]/(v_1^2))$ in internal degree $6$. Likewise, 
if $y$ is $v_1$--divisible, then there is a non-trivial class in $H^1(G_{24},S_7(\rho))$ which is 
annihilated by $v_1^2$. Again, this is a contradiction since 
$H^0(G_{24},\F_4[u^{-1},v_1]/(v_1^2))$ is trivial in internal degree $18$. Finally, if  $x$ is $v_1$--divisible then there is a non-trivial class in $H^1(G_{24},S_3(\rho))$ which is 
annihilated by $v_1^3$ and, hence, in the image of $\delta_3$. 
However, the group  $H^0(G_{24},\F_4[u^{-1},v_1]/(v_1^3))$ is trivial
in internal degree 12. 
\end{proof}

Here is the main theorem giving the complete structure of the cohomology algebra 
$H^*(G_{24},S_*(\rho))$.

\begin{Athm}\label{mainthm}
Let $R = \F_4[v_1, \Delta, k]$, where $k\in H^4(G_{24},\FF_4)$ is the 
cohomological periodicity generator.

\noindent
(a) As an $R$--module, $H^*(G_{24}, S_*(\rho))$ is isomorphic to 
\[  R\{1,\eta,\eta^2, \eta^3\} \oplus R/(v_1^2)\{x, \eta x,x^2,\eta x^2\} \oplus R/(v_1)\{\nu, y, \nu^2, \nu y, \nu^3, \Delta^{-1}\nu^2 y\},  \]
where $\Delta^{-1}\nu^2 y$ is a class in $H^3(G_{24},S_0(\rho))$ such that $\Delta(\Delta^{-1}\nu^2 y)=\nu^2y$. 

\noindent
(b) The products $\eta^2$, $\nu^2$, $\nu y$, $\eta x$ and $x^2$ are $R$--module generators and the remaining five products satisfy $\eta \nu =0$, $xy =0$ and 
\begin{align*}
\eta y &= v_1 x^2, &\nu x & = v_1 \eta x, &
y^2 &= \nu^2 \Delta.
\end{align*}

\noindent
(c) The element $\Delta^{-1}\nu^2 y$ and the products $\eta^3$, $\nu^3$ and $\eta x^2$ are $R$--module generators.
There are relations:
\begin{align*}
x^3&= \nu^2 y, &
\eta^2 x &= \nu^3, &
\nu x^2 &= v_1 \eta x^2, \\
\eta^2 y &= v_1 \eta x^2, &
\nu y^2 &= \nu^3 \Delta, &
 y^3 &= \nu^2 y \Delta
\end{align*}
and the remaining possible threefold products are zero.

\noindent
(d) All products of $\eta$, $x$ and $y$ with $\Delta^{-1}\nu^2 y$ and all fourfold products among $\eta$, $\nu$, $x$ and $y$ are trivial except for
$\eta^4 = v_1^4k$.
\end{Athm}

In order to prove \fullref{mainthm}, we calculate the cohomology of $\overline{S}_*(\rho)$ as a module over $\F_4[v_1,k]$. 
In fact we will find the following result. 

\begin{Aprop}\label{nodelta}
Let $\overline{R}:=\F_4[v_1,k]$. As an $\overline{R}$--module, $H^*(G_{24}, \overline{S}_*(\rho))$ is isomorphic to
\begin{align*}
\overline{R}\{1,\eta,\eta^2, \eta^3\}\oplus \overline{R}/(v_1^2)\{x, \eta x,x^2,\eta x^2\} \oplus \overline{R}/(v_1)\{\nu, y, \nu^2, \nu y, \nu^3, \Delta^{-1}\nu^2 y\}.\end{align*}
\end{Aprop}
\noindent
Note that, in \fullref{nodelta}, we are using the product structure in order to describe the generators of 
$H^*(G_{24},\overline{S}_*(\rho))$, but are not yet attempting
to describe the cohomology as an algebra.

\fullref{nodelta} is deduced from the following three lemmas. We write $s \doteq t$ if $s = \ell t$ for some $\ell \in \F_4$.
\begin{Alem} \label{H^1}  

\noindent 
(a) There is an isomorphism of $\F_4[v_1]$--modules 
\[H^1(G_{24},\overline{S}_*(\rho))\cong \F_4[v_1]\{\eta\}\oplus \F_4[v_1]/(v_1)\{\nu,y\}\oplus 
\F_4[v_1]/(v_1^2)\{x\}.\]

\noindent
(b) The reduction homomorphism  
\[p_*: H^1(G_{24},\overline{S}_*(\rho))\to H^1(G_{24},\overline{S}_*(\rho)/(v_1))\]
satisfies
\begin{align*}
p_*(\eta) &\doteq a, & p_*(\nu) &\doteq b, & p_*(x) &\doteq v_2 a, & p_*(y) &\doteq v_2^2b,
\end{align*}
for $a$ and $b$ as in \fullref{constant}. 

\noindent
(c) The connecting homomorphism associated to the exact sequence (\ref{v1-exact-seq}) 
is trivial on  $a$, $v_2a$, $b$ and $v_2^2b$. It is nontrivial on the generators $v_2^2a$, $v_2^3a$, 
$v_2b$ and $v_2^3b$. In fact,
\begin{align*}
\delta(v_2^2a) &\doteq v_1 \eta x, & \delta(v_2^3a) &\doteq v_1  x^2, & \delta(v_2 b) &\doteq \nu^2, & \delta(v_2^3b) &\doteq \nu y.
\end{align*}
Further,
$\nu x  \doteq  \eta v_1 x $ and $\eta y \doteq v_1x^2$.
\end{Alem} 

\begin{proof} (a) Any finitely generated graded $\F_4[v_1]$--module 
is a direct sum of a free module and of cyclic torsion modules of the form $\F_4[v_1]/(v_1^n)$. 
On the other hand, we know from \fullref{invariants} that the free part of 
$H^1(G_{24},\overline{S}_*(\rho))$ is of rank one. 
By the long exact sequence in cohomology associated to the short exact 
sequence 
$$
0\to \overline{S}_*(\rho)\buildrel{v_1}\over\longrightarrow \overline{S}_*(\rho)
\buildrel{p}\over\longrightarrow \overline{S}_*(\rho)/(v_1) \cong \F_4[u^{-1}]/(u^{-12})\to 0 ,
$$ 
we know that the submodule of $H^1(G_{24},\overline{S}_*(\rho))$ which is annihilated by $v_1$ is 
generated by the classes $\nu$, $v_1x$ and $y$. 
By \fullref{eta} and \fullref{H1} the classes $\eta$, $\nu$, $x$ and $y$ are not divisible by 
$v_1$, proving (a). 

\noindent
(b) We know from \fullref{constant} 
that $H^1(G_{24},\F_4[u^{-1}]/(u^{-12}) )$ is an $\F_4$--vector space   
on generators of the form $v_2^ia$, $v_2^ib$  for $0\leq i\leq 3$ 
and the four generators of $H^1(G_{24},\overline{S}_*(\rho))$ have to 
map to four of these classes. The four other generators map via $\delta$ to non--trivial elements of
$H^2(G_{24},\overline{S}_*(\rho))$. The claim in (b) then follows for degree reasons.

\noindent 
(c) From (a) and (b), it is clear by exactness that $\delta$ is trivial on $a$, $b$, $v_2a$ and $v_2^2b$ 
and nontrivial on the four other generators. We use that $\delta$ is $p_*$--linear (i.e. $\delta(p_*(t) s ) =t\delta(s)$) and 
\fullref{H1} to conclude that
\begin{align*}
\delta(v_2^2a) &\doteq \delta(v_2^2 p_*(\eta)) = \eta  \delta(v_2^2) = \eta v_1 x \\
\delta(v_2^3a) &\doteq \delta(v_2^2 p_*(x)) = x  \delta(v_2^2) = v_1 x^2 \\
\delta(v_2 b) &\doteq \delta(v_2 p_*(\nu)) = \nu  \delta(v_2) = \nu^2 \\
\delta(v_2^3 b) &\doteq \delta(v_2^3 p_*(\nu)) = \nu  \delta(v_2^3) = \nu y.
\end{align*}
Note further that 
\begin{align*}
\delta(v_2^2a) &\doteq \delta(v_2 p_*(x)) = \nu x, & \delta(v_2^3a) &\doteq \delta(v_2^3 p_*(\eta)) = \eta y
\end{align*}
and hence $\nu x \doteq  \eta v_1 x$ and $\eta y \doteq v_1x^2$.
\end{proof}

\begin{Alem}\label{H^2}
(a) The elements $\eta^2$, $\nu^2$, $\eta x$, $x^2$ and $\nu y$ are non-trivial in $H^2(G_{24}, \overline{S}_*(\rho))$ and
\begin{align*}
p_*(\eta^2) &\doteq a^2 ,& p_*(\nu^2) &\doteq b^2, & p_*(\eta x) &\doteq v_2a^2, &
 p_*(x^2) &\doteq v_2^2a^2, & p_*(\nu y) &\doteq v_2^2b^2.
\end{align*}

\noindent
(b) There is an isomorphism of $\F_4[v_1]$--modules 
$$
H^2(G_{24},\overline{S}_*(\rho))\cong \F_4[v_1]/\{\eta^2\}\oplus \F_4[v_1]/(v_1)\{\nu^2,\nu y\}\oplus 
\F_4[v_1]/(v_1^2)\{\eta x, x^2\}\ . 
$$ 

\noindent
(c) The connecting homomorphism associated to the exact sequence (\ref{v1-exact-seq}) 
is trivial on the generators $a^2$, $v_2a^2$, $v_2^2a^2$, $b^2$ and $v_2^2b^2$. 
It is nontrivial on the generators $v_2^3a^2$, $v_2b^2$ and $v_2^{-1}b^2$, and
\begin{align*}
\delta (v_2^3a^2) &\doteq v_1\eta x^2 ,& \delta (v_2 b^2) &\doteq \nu^3 ,&\delta(v_2^{-1}b^2)&\doteq\Delta^{-1}\nu^2y .
\end{align*}
Further, $v_1\eta x^2  \doteq  \nu x^2 \doteq \eta^2 y $.
\end{Alem} 

\begin{proof}(a) This follows from the fact that $p_*$ is a map of algebras and the images 
of the given elements are non-trivial.

\noindent
(b) We claim that $H^2(G_{24},\overline{S}_*(\rho))$ is generated as a $\F_4[v_1]$--module 
by the elements $\eta$, $\nu^2$, $\eta x$, $x^2$ and $\nu y$. In fact, by part (c) of \fullref{H^1} the submodule of $H^2(G_{24},\overline{S}_*(\rho))$ which is 
annihilated by $v_1$ is generated by the elements $v_1\eta x$, 
$v_1x^2$, $\nu^2$ and $\nu y$. In particular, the $v_1$--torsion submodule is of rank four. 
Furthermore, we know from \fullref{invariants} that the $v_1$--torsionfree part is of rank one and hence 
$H^2(G_{24},\overline{S}_*(\rho))$ is generated as a $\F_4[v_1]$--module by five elements. By (a) 
these must be the five elements $\eta^2$, $\nu^2$, $\eta x$, $x^2$ and $\nu y$. 
By part (c) of \fullref{H^1},  
the $v_1$--torsion submodule is as specified. This leaves $\eta^2$ which must generate a free 
$\F_4[v_1]$ submodule by \fullref{invariants}.

\noindent 
(c) By (a) and exactness, $\delta$ is trivial on $a^2$, $b^2$, $v_2a^2$, 
$v_2^2a^2$ and $v_2^2b^2$.  
Finally we get from \fullref{H1} and $p_*$--linearity  
\begin{align*}
\delta (v_2^3a^2) &\doteq x \eta \delta(v_2^2 ) \doteq v_1\eta x^2 \\
\delta (v_2 b^2) &\doteq \nu^2 \delta (v_2) \doteq \nu^3 \\
\delta(v_2^{-1}b^2)&=\delta(v_2^{-4}v_2^3b^2)\doteq \Delta^{-1}\delta(v_2^3b^2)=\Delta^{-1}\nu^2y . 
\end{align*}
(The last calculation is a calculation in the cohomology of $\Delta^{-1}S_*(\rho)$ which, in 
that of $\overline{S}_*(\rho)$, gives the desired relation. It uses the relation $p_*(\Delta^{-1}) = v_2^{-4}$.) 

Note further that, $\delta (v_2^3a^2) \doteq x^2 \delta(v_2) \doteq \nu x^2$ and $\delta (v_2^3a^2) \doteq \eta^2 \delta (v_2^3) \doteq \eta^2 y$.
Hence, $v_1\eta x^2  \doteq  \nu x^2 \doteq \eta^2 y $. This finishes the proof of (c).
\end{proof}

\begin{Alem}\label{H^3} 
(a) The elements $\eta^3$, $\nu^3$, $\eta x^2$ and $\Delta^{-1}\nu^2 y$ are non-trivial in $H^3(G_{24}, \overline{S}_*(\rho))$ and 
\begin{align*}
p_*(\eta^3)&\doteq a^3 & p_*(\nu^3)&\doteq b^3 = v_2a^3 \\
  p_*(\eta x^2)&\doteq v_2^2a^3 & p_*(\Delta^{-1}y \nu^2)&\doteq v_2^{-4}v_2^2b^3 = v_2^{-1}a^3 .
\end{align*}

\noindent
(b) There is an isomorphism of $\FF_4[v_1]$--modules 
$$
H^3(G_{24},\overline{S}_*(\rho))\cong \F_4[v_1]\{\eta^3\}\oplus 
\F_4[v_1]/(v_1)\{\Delta^{-1}\nu^2 y,\nu^3\}\oplus \F_4[v_1]/(v_1^2)\{\eta x^2\}\ . 
$$ 
\end{Alem} 

\begin{proof} 
(a) The first part follows from the fact that $p_*$ is a map of algebras. 

\noindent 
(b) From (a) we see that $p_*: H^3(G_{24},\overline{S}_*(\rho))\to H^3(G_{24},\overline{S}_*(\rho)/(v_1))$ 
is onto and hence $H^3(G_{24},\overline{S}_*(\rho))$ is generated by $\eta^3$, $\nu^3$, $\eta x^2$ 
and $\Delta^{-1}\nu^2y$. By part (c) of \fullref{H^2}, $\nu^3$ and  
$\Delta^{-1}\nu^2y$ are annihilated by $v_1$. Further, $v_1\eta x^2$ is nontrivial and also annihilated 
by $v_1$. By \fullref{invariants}, the torsionfree part is of rank one, so this proves the claim. 
\end{proof}

We finally turn towards the proof of \fullref{mainthm}. 
\begin{proof}[Proof of  \fullref{mainthm}.]
(a) Consider the short exact sequence of $\F_4[G_{24}]$--modules
\[ 0 \to \Sigma^{24} S_*(\rho) \xra{\Delta}S_{*+12}(\rho) \to \overline{S}_{*+12}(\rho) \to 0.\]
As in the proof of \fullref{Sn}, one can show that it is split. Therefore, the maps 
\[H^*(G_{24}, S_{*+12}(\rho) )  \to H^*(G_{24},\overline{S}_{*+12}(\rho) )\] 
are surjective and the long exact sequence on cohomology groups gives rise to short exact sequences in each degree. The claim then follows from \fullref{nodelta} and the five lemma.

\noindent
(b) Note that $\nu x \doteq v_1 \eta x$ and $\eta y \doteq v_1x^2$ follow from part (c) of \fullref{H^1}. For the other relations, first note that $p_*(\Delta)=v_2^4$ so that
\begin{align*}
p_*(\eta \nu) &\doteq ab =0, &
p_*(y^2) &\doteq v_2^4b^2 , \\
p_*(x y) &\doteq v_2^3 ab =0, &
  p_*(\Delta \nu^2) &\doteq v_2^4b^2 .
\end{align*}
Hence, up to elements in the kernel of 
\[p_*: H^2(G_{24}, S_*(\rho)) \to H^2(G_{24}, S_*(\rho)/(v_1)),\] 
the relations hold in $H^2(G_{24}, S_*(\rho))$. However, in these degrees, the non-zero elements in the kernel of $p_*$ are $v_1$--torsion free. Since $\eta \nu$, $xy$ and $y^2-\ell \Delta \nu^2$ are $v_1$--torsion, the relations hold as claimed.

\noindent
(c) Note that $v_1\eta x^2  \doteq  \nu x^2 \doteq \eta^2 y $ follow from part (c) of \fullref{H^2}.
Further,
\[p_*(\nu^3) = b^3 = v_2a^3 \doteq p_*(x \eta^2).\]
Therefore, for an appropriate $\ell \in \F_4$, $\nu^3 - \ell \eta^2 x$ must be a $v_1$ multiple of $\eta^3$. It must be zero since it is $v_1$--torsion and $\eta^3$ is $v_1$--torsion free. Hence, $\nu^3 \doteq x \eta^2$.
Finally, since 
\[ p_*(x^3) \doteq v_2^3 a^3 = v_2^2 b^3 \doteq p_*(\nu^2 y),\]
for an appropriate choice of $\ell$ in $\F_4$, we must have that $x^3 - \ell \nu^2y$ is a $v_1$--multiple of $\eta$, and again, must therefore be zero since it is $v_1$--torsion and $\eta^3$ is $v_1$--torsion free.
That the other threefold products vanish follows from part (b).

\noindent
(d)
With the exception of $\eta^4$ which is $v_1$--free, all fourfold products and all products of 
$\eta$, $\nu$, $x$ and $y$ with $\Delta^{-1}\nu^2y$ are $v_1$--torsion. However, 
\[H^4(G_{24};S_*(\rho))\cong \F_4[v_1,\Delta]\{k\}\] 
is $v_1$--torsion free. By  
\fullref{invariants} the class $\eta^4$ is $v_1$--torsion free and, for degree reasons, it must be equal to 
$v_1^4k$, up to a nontrivial scalar in $\FF_4$.

Finally, note that the generators $v_1$, $\Delta$ and $k$ are Galois invariant classes.  
Likewise, $\eta$ as the mod--$2$ Bockstein of $v_1$, and $\nu$, $x$ and 
$y$ as Bocksteins of Galois invariant classes are also Galois invariant. 
It follows that the multiplicative relations, which we have only proved modulo units in $\FF_4$, 
do hold on the nose. 
\end{proof}

\begin{Arem}
If we extend $G_{24}\cong SL_2(\F_3)$ by the Galois group to 
$G_{48}=GL_2(\F_3)$, then the $G_{48}$-cohomology of $S_*(\rho)$ is obtained from that of $G_{24}$ by taking Galois invariants. This is the content of the following result. (See \fullref{fig:hg24}).
\end{Arem}

\begin{Athm}\label{thm:Srho}
There is a ring isomorphism
\[H^*(G_{48}, S_*(\rho)) \cong \F_{2}[v_1,\Delta, k, \eta, \nu, x, y, \Delta^{-1}\nu^2y]/(\sim)\] 
where $(\sim)$ is the ideal generated by
\begin{align*}
v_1 \nu,  & &v_1^2x, & &v_1y, 
\end{align*}
in cohomological degree $1$,
\begin{align*}
&\eta \nu, & &  \nu x-v_1\eta x,  & &\eta y-v_1x^2, & &xy,   & & y^2-\nu^2\Delta, 
\end{align*}
in cohomological degree $2$,
\begin{align*}
 \eta^2x -\nu^3, & & x^3-\nu^2 y,  & & \Delta (\Delta^{-1}\nu^2y) - \nu^2 y 
\end{align*}
in cohomological degree $3$ and
\begin{align*}
\eta^4-v_1^4 k
\end{align*}
in cohomological degree $4$. Further, 
\[ H^*(G_{24}, S_*(\rho) )\cong \F_4 \otimes_{\F_2} H^*(G_{48}, S_*(\rho)).\]
\end{Athm}
We deduce the main result from \fullref{thm:Srho} and the following lemma, whose proof is analogous to that of Goerss, Henn and Mahowald \cite[Theorem 6]{ghmv1}.
\begin{Alem}\label{lem:completion}
Let $m =48$ or $m=24$. There is an isomorphism
\[H^*(G_{m}, S_*(\rho)[\Delta^{-1}]^{\wedge}_{(j)}) \cong (H^*(G_{m}, S_*(\rho))[\Delta^{-1}])^{\wedge}_{(j)}.\]
\end{Alem}

\begin{Athm}\label{thm:cohG24}
There is an ismorphism
\[ H^*(G_{48}, (E_{\cC})_*V(0))\cong \F_2[\![j]\!][v_1,\Delta^{\pm 1}, k, \eta, \nu, x, y]/(\sim)\]
where $(\sim)$ is the ideal generated by
\begin{align*}
 v_1^{12}-j\Delta 
\end{align*} 
 in cohomological degree $0$
\begin{align*}
 v_1 \nu,  & &  v_1^2x , & & v_1y , 
\end{align*}
in cohomological degree $1$,
\begin{align*}
\eta \nu,&  & \nu x-v_1\eta x, &  & \eta y-v_1x^2, & & xy,  &  &  y^2-\nu^2\Delta, 
\end{align*}
in cohomological degree $2$,
\begin{align*}
  \eta^2x - \nu^3,  & &  x^3-\nu^2 y,  
 \end{align*} 
 in cohomological degree $3$ and
 \begin{align*}
   \eta^4 -v_1^4 k 
\end{align*}
in cohomological degree $4$. Further, 
\[ H^*(G_{24},  (E_{\cC})_*V(0) )\cong \F_4 \otimes_{\F_2} H^*(G_{48}, (E_{\cC})_*V(0))).\]
\end{Athm}

\begin{figure}[H]
 \captionsetup{width=\textwidth}
\includegraphics[width=\textwidth]{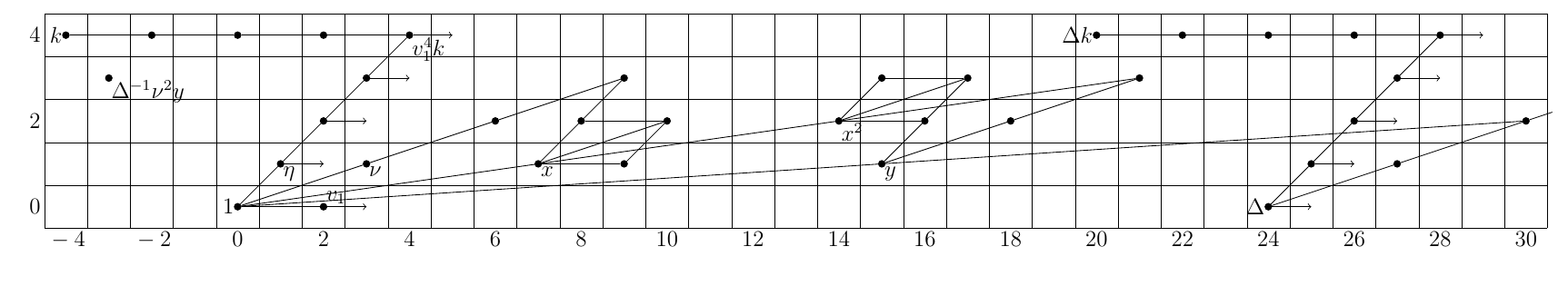}
\caption{The cohomology $H^*(G_{24}, S_*(\rho))$, drawn in the Adams grading $(t-s, s)$. 
It is periodic with respect to $s$ with period $4$ and periodicity generator $k$ and with respect to $t$ with period $24$ and periodicity generator $\Delta$. 
 A $\bullet$ denotes a copy of $\F_4$. Lines of slope $1$ denote multiplication by $\eta$ and lines of slope $1/3$ denote multiplication by $\nu$. Lines of slope $1/7$ denote multiplication by $x$ and those of slope $1/15$ multiplication by $y$. Horizontal lines denote multiplication by $v_1$. 
 Classes attached to horizontal arrows are free over $\F_4[v_1]$. }
\label{fig:hg24}
\end{figure}	


 \bibliographystyle{plain}


\end{document}